\documentclass[10pt]{amsart}

\usepackage{amsmath,amssymb,amscd,amsthm,amsxtra}

\usepackage{graphicx, mathrsfs}
\usepackage{amssymb,amsmath,amsthm}
\usepackage{mathrsfs,dsfont}

\makeatletter
\DeclareRobustCommand\widecheck[1]{{\mathpalette\@widecheck{#1}}}
\def\@widecheck#1#2{%
   \box\z@\hbox{\m@th$#1#2$}%
   \box\tw@\hbox{\m@th$#1%
      \widehat{%
         \vrule\@width\z@\@height\ht\z@
         \vrule\@height\z@\@width\wd\z@}$}%
   \dp\tw@-\ht\z@
   \@tempdima\ht\z@ \advance\@tempdima2\ht\tw@ \divide\@tempdima\thr@@
   \box\tw@\hbox{%
      \raise\@tempdima\hbox{\scalebox{1}[-1]{\lower\@tempdima\box\tw@}}}%
   {\ooalign{\box\tw@ \cr \box\z@}}}
\makeatother

\newtheorem{theorem}{Theorem} [section]

\newtheorem{proposition}[theorem]{Proposition}
\newtheorem{remark}[theorem]{Remark}
\newtheorem{example}{Example}

\newtheorem{definition}[theorem]{Definition}
\newtheorem{corollary}[theorem]{Corollary}

\begin{document}
\title{Randomization and the Gross-Pitaevskii hierarchy}

\author[V. Sohinger]{Vedran Sohinger}
\address{
University of Pennsylvania, Department of Mathematics, David Rittenhouse Lab, Office 3N4B, 209 South 33rd Street, Philadelphia, PA 19104-6395, USA}
\email{vedranso@math.upenn.edu}
\urladdr{http://www.math.upenn.edu/~vedranso/}
\thanks{V. S. was partially supported by a Simons Postdoctoral Fellowship.}

\author[G. Staffilani]{Gigliola Staffilani}
\address{
Massachusetts Institute of Technology, Department of Mathematics, Building E17, Office 330, 77 Massachusetts Avenue, Cambridge, MA 01239-4301, USA}
\email{gigliola@math.mit.edu}
\urladdr{http://math.mit.edu/~gigliola/}
\thanks{G. S. was partially supported by NSF Grant DMS-1068815.}

\begin{abstract}
We study the Gross-Pitaevskii hierarchy on the spatial domain $\mathbb{T}^3$. By using an appropriate randomization of the Fourier coefficients in the collision operator, we prove an averaged form of the main estimate which is used in order to contract the Duhamel terms that occur in the study of the hierarchy. In the averaged estimate, we do not need to integrate in the time variable. An averaged spacetime estimate for this range of regularity exponents then follows as a direct corollary.  The range of regularity exponents that we obtain is $\alpha>\frac{3}{4}$. It was shown in our previous joint work with Gressman \cite{GSS} that the range $\alpha>1$ is sharp in the corresponding deterministic spacetime estimate. This is in contrast to the non-periodic setting, which was studied by Klainerman and Machedon \cite{KM}, where the spacetime estimate is known to hold whenever $\alpha \geq 1$. The goal of our paper is to extend the range of $\alpha$ in this class of estimates in a \emph{probabilistic sense}.

We use the new estimate and the ideas from its proof in order to study randomized forms of the Gross-Pitaevskii hierarchy. More precisely, we consider hierarchies similar to the Gross-Pitaevskii hierarchy, but in which the collision operator has been randomized. For these hierarchies, we show convergence to zero in low regularity Sobolev spaces of Duhamel expansions of fixed deterministic density matrices. We believe that the study of the randomized collision operators could be the first step in the understanding of a nonlinear form of randomization.

\end{abstract}


\maketitle

\section{Introduction}
\subsection{Setup of the problem}
\label{Setup of the problem}

In this paper, we study the effect of randomization in the context of the \emph{Gross-Pitaevskii hierarchy} on the three-dimensional torus $\mathbb{T}^3$. Let us recall that the Gross-Pitaevskii hierarchy is an infinite system of linear PDEs that arises naturally in the derivation of the \emph{nonlinear Schr\"{o}dinger equation}. More precisely, the Gross-Pitaevskii hierarchy on the spatial domain $\Lambda$ is given by:
\begin{equation}
\label{GrossPitaevskiiHierarchy}
\begin{cases}
i \partial_t \gamma^{(k)} + (\Delta_{\vec{x}_k}-\Delta_{\vec{x}_k'}) \gamma^{(k)}=\sum_{j=1}^{k} B_{j,k+1} (\gamma^{(k+1)})\\
\gamma^{(k)}\big|_{t=0}=\gamma_0^{(k)}.
\end{cases}
\end{equation}
Here, $(\gamma_0^{(k)})_k$ is a fixed sequence of \emph{density matrices}, i.e. of functions $\gamma_0^{(k)}: \Lambda^k \times \Lambda^k \rightarrow \mathbb{C}$, and $(\gamma^{(k)})_k=(\gamma^{(k)}(t))_k$ is a sequence of time-dependent density matrices. The set $\Lambda$ is a spatial domain, which we take to be either $\mathbb{R}^d$ or $\mathbb{T}^d$. We say that the density matrix $\gamma_0^{(k)}$ \emph{has order $k$}. The operators $\Delta_{\vec{x}_k}$ and $\Delta_{\vec{x}_k'}$ are defined to be the Laplacian in the first and second set of $k$ spatial variables respectively. In other words:
$$\Delta_{\vec{x}_k}:=\sum_{j=1}^{k} \Delta_{x_j}\,,\,\Delta_{\vec{x}_k'}:=\sum_{j=1}^{k} \Delta_{x'_j}.$$
The map $B_{j,k+1}$ denotes the \emph{collision operator}, which is precisely defined in Section \ref{Notation}. Let us note that, in this paper, we will not assume any additional symmetry properties of the solutions to \eqref{GrossPitaevskiiHierarchy}. 

The Gross-Pitaevskii hierarchy occurs in the study of Bose-Einstein condensation, which is a state of bosonic particles at temperatures which are close to absolute zero. The particles at this low temperature have a tendency to occupy a one-particle state, which corresponds to the solution of a nonlinear Schr\"{o}dinger equation. In this context, the nonlinear Schr\"{o}dinger equation is called the \emph{Gross-Pitaevskii equation} after the work of Gross \cite{Gross} and Pitaevskii \cite{Pitaevskii}. The aforementioned physical phenomenon was predicted by Bose \cite{Bose} and Einstein \cite{Einstein} in 1924-1925. Their theoretical prediction was experimentally verified by the teams of Cornell and Wieman \cite{CW} and Ketterle \cite{Ket}, who were jointly awarded the 2001 Nobel Prize in Physics for this discovery.

A strategy in deriving nonlinear Schr\"{o}dinger type equations, in which a major step is devoted to the study of solutions to the hierarchy \eqref{GrossPitaevskiiHierarchy} was first developed by Spohn \cite{Spohn}. In this work, the author gave a rigorous derivation of the nonlinear Hartree equation $iu_t+\Delta u = (V*|u|^2)u$, for bounded convolution potentials $V=V(x)$ on $\mathbb{R}^d$. In the approach given by Spohn, the hierarchy \eqref{GrossPitaevskiiHierarchy} comes from a limit of the related BBGKY hierarchy, which, in turn is obtained from the $N$-body Schr\"{o}dinger equation of a properly scaled $N$-body Hamiltonian. Making this convergence rigorous takes quite a bit of effort. On $\mathbb{R}^3$, the work of Spohn was subsequently extended to the case of Coulomb potentials $V(x)=\pm \frac{1}{|x|}$ by Bardos, Golse and Mauser \cite{BGM} and Erd\H{o}s and Yau \cite{EY}. 

In a sequence of monumental works, Erd\H{o}s, Schlein and Yau \cite{ESY2,ESY3,ESY4,ESY5} gave a rigorous derivation of the defocusing cubic nonlinear Schr\"{o}dinger equation on $\mathbb{R}^3$. In the aforementioned works, a significant step was to check uniqueness of solutions to \eqref{GrossPitaevskiiHierarchy}. The authors prove this claim by the use of Feynman graph counting techniques. In a subsequent work, Klainerman and Machedon \cite{KM} gave an alternative proof of this claim, under slightly stronger assumptions. In particular, the authors gave a combinatorial reformulation of the Feynman graph expansion in the form of a \emph{boardgame argument}. This technique allowed them to give a proof of uniqueness without the use of Feynman diagrams, under the assumption of an a priori bound on the solution. In \cite{KSS}, Kirkpatrick, Schlein, and the second author applied these ideas to the two-dimensional periodic setting. In particular, they obtained a rigorous derivation of the defocusing cubic nonlinear Schr\"{o}dinger equation on $\mathbb{T}^2$ from the dynamics of many-body quantum systems. In our previous joint work with Gressman, \cite{GSS}, we proved a conditional uniqueness result for the Gross-Pitaevskii hierarchy on $\mathbb{T}^3$, which is the $3D$ analogue of the uniqueness result used in \cite{KSS} for regularity strictly greater than $1$.  In a recent paper, T. Chen, Hainzl, Pavlovi\'{c}, and Seiringer \cite{CHPS} used the Quantum de Finetti theorem and gave an alternative proof of the uniqueness result on $\mathbb{R}^3$ from \cite{ESY2}.

Let us note that the periodic problem was first considered in the work of Erd\H{o}s, Schlein, and Yau \cite{ESY1} and Elgart, Erd\H{o}s, Schlein, and Yau \cite{EESY}. In particular, the authors study the Gross-Pitaevskii hierarchy on $\mathbb{T}^3$ and obtain all the steps of Spohn's strategy except for uniqueness. In the subsequent work \cite{VS2}, the first author proved the uniqueness step needed in the analysis of \cite{EESY,ESY1} and thus obtained a rigorous derivation of the defocusing cubic nonlinear Schr\"{o}dinger equation on $\mathbb{T}^3$ from the dynamics of many-body quantum systems. A more detailed discussion about related works is given in Section \ref{Previously known results}.

In the above works, a key step in the proof of uniqueness was the proof of a spacetime estimate of the type:
\begin{equation}
\label{SpacetimeBoundalpha}
\big\|S^{(k,\alpha)} B_{j,k+1}\,\mathcal{U}^{(k+1)}(t)\, \gamma^{(k+1)}_0\big\|_{L^2([0,T] \times \Lambda^k \times \Lambda^k)} \lesssim \big\|S^{(k+1,\alpha)}\gamma^{(k+1)}_0\big\|_{L^2(\Lambda^{k+1} \times \Lambda^{k+1})}.
\end{equation}
for a fixed regularity exponent $\alpha$ and for a fixed time $T \in (0,+\infty]$. Here, $\mathcal{U}^{(k)}(t)$ denotes the analogue of the free Schr\"{o}dinger evolution for the operator $i \partial_t + \big(\Delta_{\vec{x}_k}-\Delta_{\vec{x}'_k}\big)$ acting on density matrices of order $k$ and $S^{(k,\alpha)}$ denotes the operator of taking $\alpha$ fractional derivatives of density matrices of order $k$. A precise definition of both operators is given in Section \ref{Notation}. We remark that in \eqref{SpacetimeBoundalpha}, the implied constant depends on $T$.

The estimate \eqref{SpacetimeBoundalpha} is used in the step in which one wants to contract the Duhamel terms that occur in the study of \eqref{GrossPitaevskiiHierarchy}. The use of this spacetime bound is what allows one to reduce what is originally a very long Duhamel expansion. The range of exponents of $\alpha$ for which one can prove \eqref{SpacetimeBoundalpha} typically determines in which regularity class one can prove uniqueness of solutions to \eqref{GrossPitaevskiiHierarchy}, at least by using the general method given in \cite{KM}.

In \cite{GSS}, it was proved that \eqref{SpacetimeBoundalpha} holds on $\mathbb{T}^3$ for $\alpha>1$. As a result, we could prove a conditional uniqueness result for \eqref{GrossPitaevskiiHierarchy} on $\mathbb{T}^3$ in a class of density matrices possessing $\alpha>1$ fractional derivatives and satisfying an a priori bound. We refer the reader to \cite{GSS} for a precise definition. It was also shown that the factorized solutions to \eqref{GrossPitaevskiiHierarchy}, which come from the NLS, belong to this class. The exact definition of factorized solutions is given in Subsection \ref{Factorized solutions}.

By using a specific counterexample, we showed that \eqref{SpacetimeBoundalpha} does not hold on $\mathbb{T}^3$ when $\alpha=1$. We noted that, in the special case of factorized density matrices, the estimate does hold when $\alpha=1$. The fact that on $\mathbb{T}^3$ \eqref{SpacetimeBoundalpha} does not hold in the endpoint case $\alpha=1$ is in sharp contrast to the $\mathbb{R}^3$ case \cite{KM}. Our goal in this paper is to explore how we may extend the range of $\alpha$ on $\mathbb{T}^3$ in a \emph{probabilistic sense}.

We note that the estimate \eqref{SpacetimeBoundalpha} was also studied in its own right in the recent work of Beckner \cite{Beckner}. In these works, the author gives several higher-dimensional generalizations in the non-periodic setting. The motivation for studying the spacetime estimate in this context is to develop a method for understanding restriction to a non-linear sub-variety \cite{Beckner2}. 

In this paper, we fix the spatial domain to be $\Lambda=\mathbb{T}^3$. Our first goal is to prove an estimate of the type \eqref{SpacetimeBoundalpha} for a larger range of $\alpha$, in an \emph{averaged sense}, by using a \emph{randomization procedure}, which is precisely defined in Definition \ref{randomization}. In particular, we \emph{randomize the Fourier coefficients} by multiplying them by a sequence of independent identically distributed standard Bernoulli random variables (meaning that their expected value is equal to $0$ and their standard deviation is equal to $1$). In the nonlinear dispersive equation literature, this idea was first applied in the work of 
Bourgain \cite{B,B2,B3,B4} on almost-sure well-posedness theory for the nonlinear Schr\"{o}dinger equation in low regularities. These works build on a wide range of techniques on randomization in nonlinear dispersive equations, which were first developed in the work of Lebowitz, Rose and Speer \cite{LRS}, and Zhidkov \cite{Zhidkov}. A related approach for the local problem was recently developed by Burq and Tzvetkov \cite{BT1}. We note that the general idea in all of these works is to add randomness into the problem in order to extend the range of regularity exponents for which one can study the PDE. These and other relevant sources are explained in more detail in Subsection \ref{Previously known results}.

By using an appropriate randomization procedure, we  prove the following result:
\\
\\
\emph{\textbf{Theorem 1:}
Let $\alpha>\frac{3}{4}$ be given. There exists a constant $C_0$ depending only on $\alpha$ such that
for all $k \in \mathbb{N}$ and $1 \leq j \leq k$, the following bound holds:}
\begin{equation}
\label{Theorem1bound2}
\|S^{(k,\alpha)} \, [B_{j,k+1}]^{\omega} \,\gamma_{0}^{(k+1)}\|_{L^2(\Omega \times \mathbb{T}^{3k} \times \mathbb{T}^{3k})} \leq C_0 \|S^{(k+1,\alpha)} \gamma_0^{(k+1)}\|_{L^2(\mathbb{T}^{3(k+1)} \times \mathbb{T}^{3(k+1)})}.
\end{equation}

The operator $[B_{j,k+1}]^{\omega}$ is a \emph{randomized collision operator}, for a fixed $\omega$ belonging to the probability space $\Omega$. It is obtained from the collision operator $B_{j,k+1}$ by appropriately randomizing the Fourier coefficients by means of standard Bernoulli random variables. 
A precise definition is given in \eqref{Bjkomega}, \eqref{Bjkomega2}, \eqref{Bjkrandomized} below. 

Using \eqref{Theorem1bound2}, and the unitarity of $\mathcal{U}^{(k)}(t)$, it is possible to prove that, for all $T>0$:
\begin{equation}
\label{Theorem1bound}
\big\|S^{(k,\alpha)} \, [B_{j,k+1}]^{\omega} \, \mathcal{U}^{(k+1)}(t) \, \gamma_{0}^{(k+1)}\big\|_{L^2(\Omega \times [0,T] \times \mathbb{T}^{3k} \times \mathbb{T}^{3k})} \leq
\end{equation}
$$C_0 \sqrt{T} \big\|S^{(k+1,\alpha)} \gamma_0^{(k+1)}\big\|_{L^2(\mathbb{T}^{3(k+1)} \times \mathbb{T}^{3(k+1)})}.
$$
whenever $\alpha>\frac{3}{4}$. The constant $C_0$ is the same as in \eqref{Theorem1bound2}. 
In particular, we see that by \eqref{Theorem1bound}, the range of regularity exponents is extended to $\alpha>\frac{3}{4}$, if one is willing to take the $L^2$ norm in the probability space $\Omega$. This is in contrast to the deterministic setting in which it is known that one has to restrict to $\alpha>1$ \cite{GSS}. Even though the bound \eqref{Theorem1bound} is a direct extension of the known bound in the deterministic setting, in our further analysis, we have to use the stronger bound given by \eqref{Theorem1bound2}. The bound in \eqref{Theorem1bound2} is the content of Theorem \ref{Theorem 1} below. The bound in \eqref{Theorem1bound} is the content of Corollary \ref{Corollary 1}.

By using Markov's inequality and Corollary \eqref{Theorem1bound}, we can deduce the following \emph{large deviation bound}:

\emph{For $\alpha>\frac{3}{4}$, $T>0$, $k \in \mathbb{N}$, $1 \leq j \leq k$, and $\lambda>0$:
\begin{equation}
\label{largedeviation}
\mathbb{P}\Big(\big\|S^{(k,\alpha)} \, [B_{j,k+1}]^{\omega} \, \mathcal{U}^{(k+1)}(t) \, \gamma_{0}^{(k+1)}\big\|_{L^2([0,T] \times \Lambda^k \times \Lambda^k)} \geq \lambda \Big) 
\end{equation}
$$\leq \frac{C_0^2 \, T \, \big\|S^{(k+1,\alpha)} \gamma_0^{(k+1)}\big\|_{L^2(\Lambda^{k+1} \times \Lambda^{k+1})}^2}{\lambda^2}$$
}
Here, $C_0$ is again the constant from \eqref{Theorem1bound2}.
The large deviation bound \eqref{largedeviation} is the content of Corollary \ref{Corollary 2}. This type of large deviation bound was shown to be useful in the study of nonlinear dispersive equations. In the context of the GP hierarchy, one has to apply similar estimates many times when estimating the Duhamel terms. In this context, it is easier to apply an averaged estimate as in \eqref{Theorem1bound2} and \eqref{Theorem1bound} than such a large deviation bound.

In the context of the randomized collision operator $[B_{j,k+1}]^{\omega}$, one is led to the study of the \emph{randomized Gross-Pitaevskii hierarchy}, which is given by:
\begin{equation}
\label{RandomizedGPI}
\begin{cases}
i\partial_t \gamma^{(k)}+(\Delta_{\vec{x}_k} -\Delta_{\vec{x}'_k})\gamma^{(k)}=\sum_{j=1}^{k}[B_{j,k+1}]^{\omega}(\gamma^{(k+1)}) \\
\gamma^{(k)}\big|_{t=0}=\gamma_0^{(k)}.
\end{cases}
\end{equation}
As is noted in Subsection \ref{Link with NLS}, the randomized Gross-Pitaevskii hierarchy \emph{admits factorized solutions}. For a fixed $\omega \in \Omega$, the sequence of density matrices given by $(\gamma^{(k)})_k:=(|\phi^{\omega} \rangle  \langle \phi^{\omega}|^{\otimes k})_k$ solves \eqref{RandomizedGPI} whenever $\phi^{\omega}$ is the randomization of $\phi$ that solves the nonlinear Schr\"{o}dinger equation: 
$$i \partial_t \phi + \Delta \phi =|\phi|^2 \phi.$$  
We note that the existence of factorized solutions is an important property of the (deterministic) Gross-Pitaevskii hierarchy \eqref{GrossPitaevskiiHierarchy}.


An interesting open problem would be to obtain a (non-deterministic) criterion for uniqueness of solutions to \eqref{RandomizedGPI} in a class which contains the natural energy space of time-dependent density matrices.
This is the class of $(\gamma^{(k)}(t))$, for which there exists a constant $C>0$ such that for all $k \in \mathbb{N}$, the following energy bound holds:
$$\|S^{(k,1)}\gamma^{(k)}(t)\|_{L^2(\Lambda^k \times \Lambda^k)} \leq C^k$$
uniformly in time. Here, $S^{(k,1)}$ is the differentiation operator which is precisely defined in \eqref{FractionalDerivative}. As a first step in this direction, one would be interested in studying the Duhamel expansions corresponding to the homogeneous problem associated to \eqref{RandomizedGPI}. 

More precisely, in order to prove any type of uniqueness result for \eqref{RandomizedGPI}, one would have to argue as in the deterministic setting and study the problem:
\begin{equation}
\notag
\begin{cases}
i\partial_t \gamma^{(k)}+(\Delta_{\vec{x}_k} -\Delta_{\vec{x}'_k})\gamma^{(k)}=\sum_{j=1}^{k}[B_{j,k+1}]^{\omega}(\gamma^{(k+1)}) \\
\gamma^{(k)}|_{t=0}=0.
\end{cases}
\end{equation}
and one would have to find a criterion for $(\gamma^{(k)})_k$
such that the Duhamel terms:
\begin{equation}
\label{DuhamelRandomI1}
\sigma^{(k)}_{n;\,\omega}(t_k):=
\end{equation}
$$(-i)^n \int_{0}^{t_k} \int_{0}^{t_{k+1}} \cdots \int_{0}^{t_{n+k-1}} \mathcal{U}^{(k)}(t_k-t_{k+1}) \, [B^{(k+1)}]^{\omega} \,\mathcal{U}^{(k+1)}(t_{k+1}-t_{k+2})$$
$$ [B^{(k+2)}]^{\omega} \cdots \,\mathcal{U}^{(n+k-1)}(t_{n+k-1}-t_{n+k}) \, [B^{(n+k)}]^{\omega} \, \gamma^{(n+k)}(t_{n+k}) \,dt_{n+k} \cdots dt_{k+2} \, dt_{k+1}.$$ 
converge to zero in an appropriate norm as $n \rightarrow \infty$. Here, we assume that $k \in \mathbb{N}$ and $t_k \in \mathbb{R}$ are fixed and we use the shorthand notation:
$[B^{(\ell+1)}]^{\omega}:=\sum_{j=1}^{\ell}[B_{j,\ell+1}]^{\omega}.$
In the definition of $\sigma^{(k)}_{n;\,\omega}$, $k$ denotes the order of the density matrix, $n$ denotes the length of the Duhamel expansion and $\omega$ is a fixed parameter in the probability space $\Omega$.
There are several difficulties in applying the technique used in the deterministic setting. Arguing by analogy with the deterministic setting, one would want to apply the spacetime estimate \eqref{Theorem1bound} in order to contract the Duhamel term. The main issue is that, in \eqref{Theorem1bound}, the density matrix $\gamma_0^{(k+1)}$ is not allowed to depend on $\omega$. In the expansion \eqref{DuhamelRandomI1}, there is dependence on $\omega$ in each collision operator $[B^{(\ell+1)}]^{\omega}$ for $\ell=k,\ldots,n-1$. Moreover, the density matrix $\gamma^{(n+k)}$ in principle depends on $\omega$ in a complicated way. In any case, it is not possible to apply the spacetime bound in order to contract the terms as was the case in the deterministic setting. This point is explained in Subsection \ref{Difficulties arising from higher-order Duhamel expansions}.

One way of dealing with the difficulty of the $\omega$-dependence of all of the collision operators $[B^{(\ell+1)}]^{\omega}$ is to modify the problem and to look at the \emph{independently randomized Gross-Pitaevskii hierarchy}, i.e., given a sequence $(\omega_k)_{k \geq 2}$ of elements in the probability space $\Omega$, we look at the hierarchy:
\begin{equation}
\label{RandomizedGPI2}
\begin{cases}
i\partial_t \gamma^{(k)}+(\Delta_{\vec{x}_k} -\Delta_{\vec{x}'_k})\gamma^{(k)}=\sum_{j=1}^{k}[B_{j,k+1}]^{\omega_{k+1}}(\gamma^{(k+1)}) \\
\gamma^{(k)}|_{t=0}=\gamma_0^{(k)}.
\end{cases}
\end{equation}
In other words, the randomizations in the collision operators at different levels are now independent of each other. We symbolize this by denoting the randomization parameters $\omega_{k+1}$ differently. The hierarchy \eqref{RandomizedGPI2} is a generalization of \eqref{RandomizedGPI} in the sense that it reduces to the latter hierarchy when all of the $\omega_{k+1}$ are mutually equal. However, it is no longer true in general that \eqref{RandomizedGPI2} admits factorized solutions, see Subsection \ref{Properties of randomizedGP2}.

When studying the homogeneous problem associated to \eqref{RandomizedGPI2}, one is led to the study of the following Duhamel terms:

\begin{equation}
\label{DuhamelRandomI2}
\sigma^{(k)}_{n;\,\omega_{k+1},\omega_{k+2},\ldots,\omega_{n+k}}(t_k):=
\end{equation}
$$(-i)^n \int_{0}^{t_k} \int_{0}^{t_{k+1}} \cdots \int_{0}^{t_{n+k-1}} \mathcal{U}^{(k)}(t_k-t_{k+1}) \, [B^{(k+1)}]^{\omega_{k+1}} \,\mathcal{U}^{(k+1)}(t_{k+1}-t_{k+2})$$
$$ [B^{(k+2)}]^{\omega_{k+2}} \cdots \,\mathcal{U}^{(n+k-1)}(t_{n+k-1}-t_{n+k}) \, [B^{(n+k)}]^{\omega_{n+k}} \, \gamma^{(n+k)}(t_{n+k}) \,dt_{n+k} \cdots dt_{k+2} \, dt_{k+1}.$$ 
We note that the superscript $k$ denotes the order of the density matrix, the subscript $n$ denotes the length of the Duhamel expansion, and $\omega_{k+1},\omega_{k+2},\ldots,\omega_{n+k}$ are independently chosen elements of the probability space $\Omega$.
In this form, it is still the case that $\gamma^{(n+k)}$ has a complicated dependence on $(\omega_{k+1},\omega_{k+2},\ldots,\omega_{n+k})$. 

We deal with the last difficulty by fixing a \textbf{\emph{deterministic}} sequence of time-dependent density matrices $(\gamma^{(k)}(t))_k$ satisfying the a priori bound
\begin{equation}
\label{aprioriboundI}
\|S^{(k,\alpha)}\gamma^{(k)}(t)\|_{L^2(\mathbb{T}^{3k} \times \mathbb{T}^{3k})} \leq C_1^k
\end{equation}
for some constant $C_1>0$, which is independent of $k$ and $t$. Here, we do not assume that $(\gamma^{(k)})_k$ solves \eqref{RandomizedGPI2}. The a priori condition is natural in the sense that it is satisfied globally in time for factorized solutions to \eqref{GrossPitaevskiiHierarchy} and \eqref{RandomizedGPI} when $\alpha=1$, and it is satisfied locally in time for factorized solutions when $\alpha>1$. The latter observation can be deduced from \cite{1B}. Moreover, sequences of density matrices satisfying such a priori bounds were shown to arise naturally in the study of the Cauchy problem for the GP hierarchy in the recent work of T. Chen and Pavlovi\'{c} \cite{CP1,CP2,CP3,CP4,CP}.

We now define $\sigma^{(k)}_{n;\,\omega_{k+1},\omega_{k+2},\ldots,\omega_{n+k}}$ analogously as in \eqref{DuhamelRandomI2}, keeping in mind that the sequence $(\gamma^{(k)}(t))_k$ is a fixed sequence of time-dependent density matrices which do not necessarily solve \eqref{RandomizedGPI2}. We now fix $n \in \mathbb{N}$ and we look at $\tilde{\gamma}^{(k)}$ for $k=1,\ldots,n$ given by:
\begin{eqnarray*}
\tilde{\gamma}^{(1)}&:=&\sigma^{(1)}_{n;\,\omega_2,\omega_3, \omega_4, \omega_5, \ldots, \,\omega_{n+1}}\\
\tilde{\gamma}^{(2)}&:=&\sigma^{(2)}_{n-1;\,\omega_3, \omega_4, \omega_5, \ldots,\, \omega_{n+1}}\\
\tilde{\gamma}^{(3)}&:=&\sigma^{(3)}_{n-2;\,\omega_4, \omega_5, \ldots, \,\omega_{n+1}}\\
&&\vdots\\
\tilde{\gamma}^{(n)}&:=&\sigma^{(n)}_{1;\,\omega_{n+1}}.
\end{eqnarray*}
The $\tilde{\gamma}^{(k)}$ then solve, by construction:

\begin{equation}
\notag
\begin{cases}
i \partial_t \tilde{\gamma}^{(k)} + (\Delta_{\vec{x}_k}-\Delta_{\vec{x}'_k})\tilde{\gamma}^{(k)}=\sum_{j=1}^{k} [B_{j,k+1}]^{\omega_{k+1}} (\tilde{\gamma}^{(k+1)})\\
\tilde{\gamma}^{(k)}\big|_{t=0}=0.
\end{cases}
\end{equation}
for all $k \in \{1,2,\ldots,n-1\}$. In other words, we obtain an \emph{arbitrarily long subset of solutions} to the full hierarchy \eqref{RandomizedGPI2} with zero initial data.

It is now possible to apply the estimate \eqref{Theorem1bound2} in order to study the Duhamel terms $\sigma^{(k)}_{n;\,\omega_{k+1},\omega_{k+2},\ldots,\omega_{n+k}}$. We can prove the following:
\\
\\
\emph{\textbf{Theorem 2:}
Suppose that $\alpha>\frac{3}{4}$ and $k \in \mathbb{N}$. There exists $T>0$ depending only on the constant $C_1$ in \eqref{aprioriboundI} and on $\alpha$ such that:
\begin{equation}
\label{Theorem2bound}
\,\, \sup_{t \in [0,T]} \, \big\|S^{(k,\alpha)} \sigma^{(k)}_{n;\,\omega_{k+1},\omega_{k+2},\ldots,\omega_{n+k}}(t)\big\|_{L^2 \big(\Omega_{k+1} \times \Omega_{k+2} \times \cdots \times \Omega_{n+k}; \,L^2(\mathbb{T}^{3k} \times \mathbb{T}^{3k})\big)} \rightarrow 0\,\,
\end{equation} 
as $n \rightarrow \infty.$
\\
\\
Moreover,
\begin{equation}
\label{Theorem2bound2}
\,\,\sup_{t \in [0,T]}\, \big\|S^{(k,\alpha)}\sigma^{(k)}_{n;\, \omega_{k+1},\omega_{k+2}, \ldots, \omega_{n+k}}(t)\big\|_{L^2\big(\prod_{m \geq 2} \Omega_m; L^2(\mathbb{T}^{3k} \times \mathbb{T}^{3k})\big)} \rightarrow 0
\end{equation}
as $n \rightarrow \infty.$
}
\\
\\
	This result is proved as Theorem \ref{Smallness Bound} in Subsection \ref{Application of randomized estimate}.  As is noted in Remark \ref{Remark2}, it is possible to prove Theorem 2 when $\Lambda=\mathbb{T}^d$, for $d \geq 1$ provided that $\alpha>\frac{d}{4}$. Throughout our paper, we mostly study the case $\Lambda=\mathbb{T}^3$ since in \cite{GSS} we were able to describe the full range of regularity exponents which are admissible in the estimate \eqref{Theorem1bound}. We note that the norm in \eqref{Theorem2bound2} is well-defined due to a result first proved by work of Kakutani \cite{Kakutani}, which builds on the previous work by Kolmogorov \cite{Kolmogorov}. The significance of taking this norm, compared to the one in \eqref{Theorem2bound}, is that, in this way, we measure convergence in a space which is independent of $n$. 

Let us emphasize that the time interval $T$ in the above theorem is independent of $n$. As we will see in the proof, it is the case that $T \sim \frac{1}{C_1}$. In other words, we obtain uniform convergence to zero of the Duhamel terms as $n \rightarrow \infty$ in a fixed norm on a fixed time interval.

As was remarked earlier, the spacetime estimate \eqref{Theorem1bound} is not directly applicable in the study of the randomized Gross-Pitaevskii hierarchy \eqref{RandomizedGPI}. In this context, one would like to prove an estimate of the form:
\begin{equation}
\label{DuhamelnI}
\big\|S^{(n,\alpha)}\,\mathcal{U}^{(n)}(t_1-t_2)\,[B^{\pm}_{j_1,k_1}]^{\omega}\,\mathcal{U}^{(n+1)}(t_2-t_3)\,[B^{\pm}_{j_2,k_2}]^{\omega} \cdots
\end{equation}
$$\cdots \, \mathcal{U}^{(n+\ell-1)}(t_{\ell}-t_{\ell+1}) \,[B^{\pm}_{j_{\ell},k_{\ell}}]^{\omega} \,\gamma_0^{(n+\ell)}\big\|_{L^2(\Omega \times \mathbb{T}^{3n} \times \mathbb{T}^{3n})} \leq C \big\|S^{(n+\ell,\alpha)} \gamma^{(n+\ell)}_0\big\|_{L^2(\mathbb{T}^{3(n+\ell)} \times \mathbb{T}^{3(n+\ell)})}$$ 
for all sequences of density matrices $(\gamma^{(k)}_0)_k$, and for $n,\ell \in \mathbb{N}$, $j_1,\ldots j_{\ell}, k_1, \ldots, k_{\ell} \in \mathbb{N}$, with $1 \leq j_1<k_1 \leq n+\ell, \ldots, 1 \leq j_{\ell}<k_{\ell} \leq n+\ell$ and $t_1,t_2, \ldots, t_{\ell+1} \in \mathbb{R}$.

The fact that \eqref{DuhamelnI} holds when $n=1$ and $\alpha>\frac{3}{4}$ follows from \eqref{Theorem1bound}. However, it is shown in Subsection \ref{Difficulties arising from higher-order Duhamel expansions} that, for general density matrices, we cannot use the randomization to obtain this type of bound. In this subsection, we note that the pairing of the frequencies as a result of the randomization does not allow us to close the estimate in the case $n=2$ as it did in the case when $n=1$. The argument for $n \geq 3$ similarly does not apply.

	The good news is that the estimate \eqref{DuhamelnI} holds if we restrict the density matrices to lie in an appropriate \emph{non-resonant class}. This is similar to the idea of \emph{Wick-ordering} \cite{B3,COh} and related ideas applied in \cite{NS} in the context of the nonlinear Schr\"{o}dinger equation. In fact, we will work in the non-resonant class $\mathcal{N}$, which is precisely defined in Subsection \ref{Estimate in N}. We note that this class does not contain the factorized solutions to \eqref{RandomizedGPI}. As we will see, it is an important matter to examine the behavior of the constant $C$ obtained in \eqref{DuhamelnI} in terms of $n$ and $\ell$. The exact estimate is given in Theorem \ref{Smallness Bound 2} of Subsection \ref{Estimate in N}.

	We again consider a time-dependent deterministic sequence of density matrices $(\gamma^{(k)}(t))$. Unlike as in the previous case, we assume that, for each fixed $t$, the sequence $(\gamma^{(k)}(t))_k$ belongs to the non-resonant class $\mathcal{N}$. As a part of the definition, we assume that the sequence of density matrices satisfies an a priori bound as in \eqref{aprioriboundI}. Having chosen such a sequence $(\gamma^{(k)})$, we define $\sigma^{(k)}_{n;\,\omega}(t)$ as in \eqref{DuhamelRandomI2}. In this way, we again obtain an arbitrarily long subset of solutions to the hierarchy \eqref{RandomizedGPI}.
	
	We are now in the position to apply the randomized spacetime estimate to the study of $\sigma^{(k)}_{n;\,\omega}$, i.e. of Duhamel expansions of order $n$ of elements of the non-resonant class $\mathcal{N}$. In particular, we can prove:
\\
\\
\emph{\textbf{Theorem 3:}	
Suppose that $\alpha \geq 0$ and $k \in \mathbb{N}$. There exists $T>0$, depending only on the constant $C_1$ in the definition of the class $\mathcal{N}$, $\alpha$, and on $k$ such that: 
\begin{equation}
\label{Theorem3bound}
\sup_{t \in [0,T]} \, \big\|S^{(k,\alpha)} \sigma^{(k)}_{n;\,\omega}(t)\big\|_{L^2 \big(\Omega \times \mathbb{T}^{3k} \times \mathbb{T}^{3k} \big)} \rightarrow 0
\end{equation}
as $n \rightarrow \infty.$
}

\begin{remark}
Let us observe that now, the range of regularity exponents has been extended to $\alpha \geq 0$. This is a significant step in the direction of the understanding of the low-regularity problem for Gross-Pitaevskii type hierarchies. Theorem 3 is given below as Theorem \ref{Smallness Bound 2}.  
\end{remark}

\begin{remark}
We note that, due to the presence of only one random parameter $\omega \in \Omega$, the norm in \eqref{Theorem3bound} of Theorem 3 is simpler than the one in Theorem 2.
\end{remark}

\begin{remark}
The analysis in Theorem 3 extends without any changes to the general case of $\Lambda=\mathbb{T}^d$, for $d \geq 1$.
\end{remark}

\begin{remark}
We note that the results of Theorem 2 and Theorem 3 estimate the Duhamel expansions similar to those obtained in the uniqueness analysis in the deterministic setting, as in \cite{CP1,CP2,GSS,KM,KSS}. The results of Theorem 2 and Theorem 3 show that these Duhamel expansions converge to zero in a class of density-matrices in a low-regularity space with a random component. 
\end{remark}

In the subsequent work of the first author \cite{VS}, it was shown that the full randomized hierarchies \eqref{RandomizedGPI} and \eqref{RandomizedGPI2} have local-in-time solutions for almost every value of the random parameter. This is an existence result which was motivated by truncation techniques used by T. Chen and Pavlovi\'{c} \cite{CP4} in the deterministic context. In \cite{VS}, such solutions of \eqref{RandomizedGPI} are obtained for regularities
is $\alpha>\frac{3}{4}$, and solutions of \eqref{RandomizedGPI2} are obtained for regularities $\alpha \geq 0$. As is the case in the assumptions of Theorem 3, the initial data for \eqref{RandomizedGPI2} in \cite{VS} needs to belong to an appropriate non-resonant class.
	
\subsection{Previously known results}
\label{Previously known results}

In addition to the references mentioned above, there is a vast literature on the connection between the derivation of NLS-type equations and hierarchies of the type \eqref{GrossPitaevskiiHierarchy} and related problems. 
A derivation of an NLS-type equation based on Fock space techniques was first studied in the work of Hepp \cite{Hepp} and Ginibre and Velo \cite{GV1,GV2}. The Fock space technique was later applied in \cite{FKP,FKS}. The derivation of a hierarchy similar to \eqref{GrossPitaevskiiHierarchy} coming from the limit of $N$-body Schr\"{o}dinger dynamics was obtained in \cite{ABGT,AGT,BGM,BePoSc1,BePoSc2,BdOS,CP,CT,XC3,XC4,ChenHolmer1,ChenHolmer2,ChenHolmer3,ES,FGS,FL,MichelangeliSchlein,Xie} in various different contexts. In these papers, the authors used the general strategy outlined by Spohn, instead of the Fock space technique. 

The first rate of convergence result was obtained by Rodnianski and Schlein \cite{RodnianskiSchlein} and subsequent rate of convergence results have been obtained in \cite{Anapolitanos,BdOS,ChenLeeSchlein,XC1,XC2,ErdosSchlein,FKP,GM,GMM1,GMM2,KP,Lee,Luhrmann,MichelangeliSchlein,Pickl1,Pickl2}. The Cauchy problem associated to a hierarchy as in \eqref{GrossPitaevskiiHierarchy} has been studied in its own right in \cite{CP1,CP2,CP3,CP4,CPT1,CPT2,CL}. A new unconditional uniqueness result for the cubic Gross-Pitaevskii hierarchy on $\mathbb{R}^3$, based on the Quantum de Finetti theorem has recently been obtained by T. Chen, Hainzl, Pavlovi\'{c}, and Seirenger in \cite{CHPS}. These techniques have been adapted in order to show scattering results in the context of the Gross-Pitaevskii hierarchy in \cite{ChHaPavSei2}. Low regularity extensions of this method have been obtained in \cite{HTX}. The periodic analogue of the methods from \cite{CHPS} was an important step in \cite{VS}. Recently, the Quantum de Finetti was used in the study of the Chern-Simons-Schr\"{o}dinger hierarchy in \cite{CS}.

Important results related to the Gross-Pitaevskii hierarchy have been studied on the level of the $N$-body Schr\"{o}dinger equation in the series of works \cite{LS,LSY,LSY2} with an expository account given in \cite{LSSY}. We refer the reader to the introduction of \cite{GSS} for a more precise discussion of the problem and of the references mentioned above. We would also like to recall the connection with certain optical lattice models that have been studied in \cite{ALSSY1,ALSSY2}. For a detailed introduction to the general problem, we also refer the interested reader to the lecture notes by Schlein \cite{Schlein}.

An approach based on randomization has been useful in the study of low-regularity solutions to nonlinear dispersive equations. In particular, randomization techniques have been applied in the context where the low regularity techniques of the \emph{high-low method} as in the work of Bourgain \cite{1B7} or the \emph{I-method}, as in the work of Colliander, Keel, Staffilani, Takaoka, and Tao \cite{CKSTT} are known not to work. As was mentioned in the introduction, this probabilistic approach was first used by Bourgain \cite{B,B2,B3,B4} and it has its origins in previous work of Lebowitz, Rose, and Speer \cite{LRS}, and Zhidkov \cite{Zhidkov}.
The main idea is that global existence can be studied by means of the existence of a Gibbs measure and its invariance under the flow. The properly constructed Gibbs measure is supported away from the set of initial data for which the above deterministic methods do not apply.

	The Gibbs measure approach is only known to be applicable in the context in which there exists a Hamiltonian  structure. In a more general context, the idea of a direct randomization of the Fourier coefficients without the use of an invariant Gibbs measure has been shown to be useful. One randomizes the function $f=\sum_n c_n e^{inx}$ by multiplying each Fourier coefficient $c_n$ with $h_n(\omega)$, where $h_n(\omega)$ are independent, identically distributed standard random variables whose expected value is zero. The precise definition is given in Definition \ref{randomization} below.
	
	In the context of nonlinear dispersive PDE, this idea was pioneered in the work of Burq and Tzvetkov \cite{BT1}. A key fact in this work is the fact that, due to the randomization, one can obtain improved $L^p$ regularity in the initial data almost surely. This is a phenomenon which was first noted in work of Rademacher \cite{Rademacher} and Paley and Zygmund \cite{PaleyZygmund1,PaleyZygmund2,PaleyZygmund3}. Related results were also proved in the work of Marcinkiewicz and Zygmund \cite{MarcinkiewiczZygmund} and Khintchine, see \cite{Wolff}. We also note the work \cite{AyacheTzvetkov}, in which the authors study the special case of Gaussian random series in which some of these observations can be deduced by alternative means. It is noted in \cite{BT1} that the randomization does not improve regularity on the Sobolev scale almost surely. Hence, all of the gain has to be obtained due to the almost sure gain in integrability. Let us remark that the idea of a gain in integrability due to randomization is related to the more general phenomenon of \emph{hypercontractivity} 
\cite{Federbush,Glimm,LGross1,LGross2,Nelson1}.

In addition to the mentioned works, there is a wide range of results on the application of randomization to nonlinear dispersive PDEs. We refer the interested reader to the works \cite{BenyiOhPocovnicu1,BenyiOhPocovnicu2,BourgainBulut,BourgainBulut2,BourgainBulut3,BTT,BurqTzvetkov,BT1,1BT2,BurqandTzvetkov,Cacciafesta_deSuzzoni,COh, Deng1, Deng2,DengTzvetkovVisciglia,LM,NORBS,NRBSS,Oh1,Oh2,Oh3,OhSulem,Richards,deSuzzoni2,deSuzzoni,deSuzzoni3,deSuzzoniTzvetkov,1Tho,TT,Tzv1,Tzv2,Tzv3,Xu}, as well as to the expository works \cite{1B24,Zh} and to the references therein. Furthermore, we note that the idea of randomization of the Fourier coefficients without the use of an invariant measure has also been applied in the context of the Navier-Stokes equations. The relevant works are \cite{DengCui1,DengCui2,NahmodPavlovicStaffilani,ZhangFang}. By applying this method, the authors were able to obtain existence results almost surely in supercritical regimes.

Probabilistic methods have previously been applied in the context of the $N$-body Schr\"{o}dinger equation, but in a slightly different context. In the recent work \cite{BenArousKSch}, it is proved that the fluctuations around the limiting dynamics given by the Hartree equation satisfy a central limit theorem. This approach builds on previous central limit theorems in a quantum setting \cite{CramerEisert,CushenHudson, GVV,Hayashi,HeppLieb,JaksicPautratPillet,Kuperberg}.

An alternative probabilistic approach was taken in \cite{AdCoKo}. Here, the authors consider a many-body system of mutually repellent bosons and they derive an explicit variational expression for the corresponding limiting free energy. In doing so, they use the cycle structure of random paths which appear in the Feynman-Kac formula. The main novelty of this paper is to recast the whole interacting system as an expectation with respect to a marked Poisson process. Methods involving point processes were previously applied in the study of Bose gases without interaction in \cite{Fichtner,Rafler}.

In a recent preprint \cite{ChatterjeeDiaconis}, the authors analyze the fluctuations of the Bose-Einstein condensate for a system of non-interacting bosons with a potential. The system is assumed to modeled according to the canonical ensemble. In this context, the authors give a rigorous proof of Bose-Einstein condensation with positive probability assuming that the temperature is sufficiently low. The probability measure on the space of configurations is the \emph{canonical Gibbs measure}.

We note that probabilistic techniques have been applied to the study of $N$-body Schr\"{o}dinger problems in the experimental literature as well. In particular, we note the work \cite{Zwierlein1}. Here, the experiment is based on applying a randomized Monte Carlo method of thermodynamic measurements of a unitary Fermi gas across the superfluid phase transition. In this way, it is possible to validate the theory of strongly interacting matter given by Bardeen, Cooper and Schrieffer \cite{BaCoSc1,BaCoSc2,Co}. It is also noted that similar probabilistic methods can be applied in the study of other physical systems, such as two-dimensional Bose and Fermi gases as well as fermions in optical lattices. The precise formulation of the randomization method, i.e. the \emph{Bold diagrammatic Monte Carlo (BDMC)}, was first developed in \cite{ProkofievSvistunov2,ProkofievSvistunov3,ProkofievSvistunov1}. This approach was also used in \cite{Zwierlein2} in the context of formal summation of Feynman graphs.

\subsection{Ideas and techniques used in the proofs}
Let us first recall the definition of the randomization of a function. Here $\Lambda=\mathbb{T}^d$ is the spatial domain.

\begin{definition}
Given $f \in L^2(\Lambda)$, with Fourier series $\sum_{n} c_n e^{inx}$ and a sequence of independent, identically distributed random variables with expected value zero $(h_n(\omega))$, we define the function $f^{\omega}$ by:
\label{randomization}
\begin{equation}
f^{\omega}:=\sum_n h_n(\omega) c_n e^{inx},
\end{equation} 
$f^{\omega}$ is called the randomization of $f$. 
\end{definition}

As was mentioned above, one of the big improvements obtained by randomization is the \emph{gain of integrability}, almost surely.
More precisely, given a sequence of independent identically distributed random Bernoulli or Gaussian random variables $(h_n(\omega))$, centered at zero, one obtains that for all $1<p<\infty$:
\begin{equation}
\label{GainOfIntegrability}
\big\|\sum_{n} h_n(\omega) c_n \big\|_{L^p(\Omega)} \leq C(p) \big(\sum_{n} |c_n|^2 \big)^{\frac{1}{2}}.
\end{equation}
We observe that this is a better bound than the bound $C(p) \big( \sum |c_n| \big)$, which immediately follows from the triangle inequality.
For a proof of \eqref{GainOfIntegrability}, we refer the reader to the proof of Lemma 4.2 in \cite{BT1}.

 We can then use Fubini's Theorem, \eqref{GainOfIntegrability} with $c_n:=\widehat{f}(n) \cdot e^{inx}$ and Plancherel's Theorem in order to deduce that:
$$\|f^{\omega}\|_{L^p(\Omega \times \Lambda)} \leq C(p) \|f\|_{L^2}.$$
In this way, the randomization gives us a gain in integrability, almost surely.

It is instructive to recall the main idea of the proof of \eqref{GainOfIntegrability}. One typically takes $p$ to be an even integer, since the other cases can be deduced by interpolation. Taking $p$-th powers of the left-hand side in \eqref{GainOfIntegrability}, we obtain a sum of terms of the type:
$$\int_{\Omega} h_{n_1}(\omega) \cdots h_{n_{k}}(\omega) \cdot \overline{h}_{m_{1}}(\omega) \cdots \overline{h}_{m_{\ell}}(\omega) \cdot c_{n_1} \cdots c_{n_{k}} \cdot \overline{c}_{m_1} \cdots \overline{c}_{m_{\ell}}\,dp(\omega)$$
for some $k,\ell \in \mathbb{N}$ with $k+\ell=p$.
Since the random variables are mutually independent and since they have expected value equal to zero, the above expression equals to zero if an $n_i$ or and $m_j$ occurs exactly once in the 
set $\{n_1,\ldots,n_k,m_1,\ldots,m_{\ell}\}$. Hence, each frequency must \emph{pair up} with at least one of the other frequencies or else the contribution equals zero. This observation reduces the original sum one is considering in \eqref{GainOfIntegrability}. The claim then follows by using H\"{o}lder's inequality. We note that this argument works in all dimensions.

In applications of randomization techniques to the study of nonlinear PDE, which were mentioned above, one usually randomizes in the initial data. In particular, if one considers an NLS-type equation:
\begin{equation}
\notag
\begin{cases}
i \partial_t u + \Delta u=N(u)\\
u\big|_{t=0}=u_0.
\end{cases}
\end{equation}
for some power-type nonlinearity $N(u)$, one typically takes the initial data to be \emph{random}. In other words, one considers $u_0=\phi^{\omega}$, in the sense discussed in Definition \ref{randomization} above. Here $\phi=\phi(x)$ is a fixed deterministic function. 
As was noted earlier, the function $\phi^{\omega}$ exhibits better integrability properties than the function $\phi$, on average. The key step is then to write:
$$u=e^{it\Delta} \phi^{\omega} + v$$
and to note that the function $v$ satisfies the following Cauchy problem, with homogeneous initial data:
\begin{equation}
\label{DifferenceEquation}
\begin{cases}
i \partial_t v + \Delta v = N(e^{it\Delta}\phi^{\omega}+v)\\
v\big|_{t=0}=0.
\end{cases}
\end{equation}
In solving for $v$ in \eqref{DifferenceEquation}, one uses a fixed point argument. One controls the terms coming from $e^{it\Delta}\phi^{\omega}$ in a higher $L^p$ space. More precisely, one uses the gain of integrability almost surely combined with a large deviation bound to note that the corresponding $L^p$ norm of the linear part is bounded for a large set of $\omega \in \Omega$. A quantitative estimate on the measure of the good set of $\omega$ is given by a large deviation bound, which is typically deduced by Markov's inequality. One can apply all of these facts and use H\"{o}lder's inequality to note that the nonlinear term satisfies the right estimate from which one can apply a fixed point argument, for $\omega$ belonging to a good subset of $\Omega$. An important step in this approach is to use the fact that, in the expression $N(e^{it\Delta}\phi^{\omega}+v)$, the only first-order terms in $v$ are multiplied with a factor of $e^{it\Delta}\phi^{\omega}$, which is bounded in a higher $L^p$ norm.

The above approach does not apply to the study of the Gross-Pitaevskii hierarchy since the problem is now linear. Namely, let us consider the problem with random initial data:
\begin{equation}
\notag
\begin{cases}
i \partial_t \gamma^{(k)} + (\Delta_{\vec{x}_k}-\Delta_{\vec{x}_k'}) \gamma^{(k)}=\sum_{j=1}^{k} B_{j,k+1} (\gamma^{(k+1)})\\
\gamma^{(k)}\big|_{t=0}=\gamma_{0,\omega}^{(k)}.
\end{cases}
\end{equation}
We now write:
$$\gamma^{(k)}=\mathcal{U}^{(k)}(t)\,\gamma^{(k)}_{0,\omega} + \tilde{\gamma}^{(k)}.$$ 
The $\tilde{\gamma}^{(k)}$ now solve the following problem with homogeneous initial data:
\begin{equation}
\label{DifferenceEquation2}
\begin{cases}
i \partial_t \tilde{\gamma}^{(k)} + (\Delta_{\vec{x}_k}-\Delta_{\vec{x}_k'}) \tilde{\gamma}^{(k)}=\sum_{j=1}^{k} B_{j,k+1} \big(\,\mathcal{U}^{(k+1)}(t) \gamma^{(k+1)}_{0,\omega}+\tilde{\gamma}^{(k+1)}\big)\\
\tilde{\gamma}^{(k)}\big|_{t=0}=0.
\end{cases}
\end{equation}
We note that the operator $B_{j,k+1}$ is linear so the right-hand side in \eqref{DifferenceEquation2} equals:
$$\sum_{j=1}^{k}B_{j,k+1}\big(\,\mathcal{U}^{(k+1)}(t)\gamma^{(k+1)}_{0,\omega}\big)+ \sum_{j=1}^{k}B_{j,k+1}\big(\tilde{\gamma}^{(k+1)}\big).$$
In other words, the free evolution of the random part given by $\mathcal{U}^{(k+1)}(t)\,\gamma^{(k+1)}_{0,\omega}$ and the remainder $\tilde{\gamma}^{(k+1)}$ on the right-hand side are \emph{completely decoupled}. As a result, there is no small factor multiplying the term which one would want to estimate, i.e. $\sum_{j=1}^{k}B_{j,k+1}\big(\tilde{\gamma}^{(k+1)}\big)$.

Due to the phenomenon discussed above, we \emph{randomize in the collision operator}, instead of in the initial data. This approach can be viewed as a first step in the direction of a \emph{nonlinear form of randomization}, since the collision operator on the level of the Gross-Pitaevskii hierarchy corresponds to the nonlinearity on the level of the nonlinear Schr\"{o}dinger equation. 
A remark on what is possible to say when one randomizes in the initial data in the context of the Gross-Pitaevskii hierarchy is given in Subsection \ref{Alternative form of randomization}.

An additional motivation for randomizing the collision operator is that doing so helps us obtain a gain in integrability and thus prove a randomized version of the spacetime estimate \eqref{SpacetimeBoundalpha} for a larger class of $\alpha$. More precisely (c.f. \eqref{FourierTransform1} and \eqref{FourierTransform2} below), on the Fourier transform side, the collision operator involves a sum in the frequencies as follows:

\begin{equation}
\notag
(B_{1,k+1} \gamma_0^{(k+1)})\,\,\widehat{}\,\,(\vec{\xi}_k; \vec{\xi'}_k)= 
\end{equation}
$$\sum_{{\xi}_{k+1}, \, {\xi'}_{k+1} \in \mathbb{Z}^3} \, \widehat{\gamma_0}^{(k+1)}(\xi_1-\xi_{k+1}+\xi'_{k+1}, \xi_2, \ldots, \xi_k, \xi_{k+1}; \xi'_1,\ldots, \xi'_k,\xi'_{k+1})$$
$$-\sum_{{\xi}_{k+1}, \, {\xi'}_{k+1} \in \mathbb{Z}^3} 
\,\widehat{\gamma_0}^{(k+1)}(\xi_1, \ldots, \xi_k, \xi_{k+1}; \xi'_1-\xi'_{k+1}+\xi_{k+1}, \xi'_2, \ldots, \xi'_k,\xi'_{k+1}).$$

Our goal is to multiply the summands in the above formula with appropriate random coefficients in such a way that we can apply the techniques used in the proof of \eqref{GainOfIntegrability} in the context of the spacetime estimate \eqref{SpacetimeBoundalpha}. The randomization given in \eqref{Bjkomega} and \eqref{Bjkomega2} below is the one which leads to the randomized spacetime bound given by \eqref{Theorem1bound2}. 


	We recall that in \cite{GSS}, the proof of the spacetime estimate relied on the dispersive effect of the free evolution $\mathcal{U}^{(k+1)}(t)$. In the randomized setting, we only use the fact that this operator is unitary on $L^2$-based spaces. It is the pairing of the frequencies which ultimately gives us the gain. This is a common phenomenon 
in the study of nonlinear dispersive equations by means of randomization. As a result of the proof, we will be able to prove an estimate on a quantity which does not involve an integral in time, as is seen in \eqref{Theorem1bound2}.
	
	It is also possible that the range of $\alpha$ in \eqref{Theorem1bound} could be extended by using dispersive properties of $\mathcal{U}^{(k+1)}(t)$ by using similar ideas as in \cite{B3}. We will not pursue this approach in our paper. We remark that such an improvement could only be applicable in estimates in which the norms involve integration in time. As we will see, norms which do not involve integrals in time are more applicable to the study of randomized Gross-Pitaevskii hierarchies.

As was noted above, the construction of the randomized collision operator leads to the study of the \emph{randomized Gross-Pitaevskii hierarchy} \eqref {RandomizedGPI} and of the \emph{independently randomized Gross-Pitaevskii hierarchy} \eqref{RandomizedGPI2}. The randomized Gross-Pitaevskii hierarchy \eqref {RandomizedGPI} shares a lot with properties of the regular Gross-Pitaevskii hierarchy, such as the existence of factorized solutions and of the applicability of the boardgame argument. As in the deterministic setting, we need to assume that our density matrices are invariant under permutation of the spatial variables $\vec{x}_k$ and $\vec{x}'_k$ in order to apply the boardgame argument. Due to the dependent randomization, we will not be able to apply the randomized spacetime estimate \eqref{Theorem1bound2} directly in this context. More precisely, we can only apply the averaged spacetime estimate to one Duhamel iteration, and as soon as we do at least two iterations of the Duhamel principle, we can no longer directly use Theorem \ref{Theorem 1}.
The reason is that we would have to estimate a quantity of the type:
$$\big\|S^{(k,\alpha)}[B_{j,k+1}]^{\omega} \gamma_{0,\omega}^{(k+1)}\big\|_{L^2(\Omega \times \Lambda^k \times \Lambda^k)}$$
where $\gamma_{0,\omega}^{(k+1)}$ \emph{depends on $\omega$}. 
The proof of Theorem \ref{Theorem 1} relies heavily on the fact that $\gamma^{(k+1)}_0$ is independent of $\omega$, which, in turn, allows us to deduce the pairing of the frequencies.

One way of getting around this difficulty is to work with the independently randomized GP hierarchy \eqref{RandomizedGPI2}. However, in this case we have to be careful and recall that it is not possible to apply the boardgame argument in this context.  The problem in applying the boardgame argument lies in the fact that, in the integrals given by the Duhamel expansion, it is possible to interchange the $t_k$ and $x_k$ variables, but it is not possible to interchange the $\omega_k$ variables. Nonetheless, in the context of the independently randomized Gross-Pitaevskii hierarchy, we can directly apply the spacetime estimate \eqref{Theorem1bound2} in order to show \eqref{Theorem2bound}.

An important observation in the proof of \eqref{Theorem2bound} is that, in \eqref{Theorem1bound2}, we do not need to put the free evolution $\mathcal{U}^{(k+1)}(t)$ and a time integral inside of the norm in order to obtain the estimate. As a result, we can just use the unitarity of $\mathcal{U}^{k+1}(t)$. An important fact which we use is the gain of $\frac{1}{n!}$ in the integral identity:
\begin{equation}
\label{integralidentity}
\int_{0}^{t_{k}} \int_{0}^{t_{k+1}} \cdots \int_{0}^{t_{n+k-1}} \,dt_{n+k}\,\cdots \,dt_{k+2} \,dt_{k+1} = \frac{t_k^n}{n!}\,.
\end{equation}
This gain was previously used in the study of the Gross-Pitaevskii hierarchy on $\mathbb{R}$ in \cite{CP2}. As a result, we can control the factorial number of Duhamel terms which we obtain in the expansion. Consequently, we can prove Theorem \ref{Smallness Bound} with a quantity $T>0$, which is independent of $n$.

A possible approach in the context of the hierarchy \eqref {RandomizedGPI} is to argue directly and prove a good spacetime estimate for higher-order Duhamel expansions without directly using \eqref{Theorem1bound2}. This reduces to a purely combinatorial problem of possible pairings of the frequencies.
However, as is shown in Subsection \ref{Difficulties arising from higher-order Duhamel expansions}, it is not possible to prove a good spatial estimate in the class of general density matrices by applying the combinatorial method used to prove \eqref{Theorem1bound2}. In Subsection \ref{Estimate in N}, an appropriate spatial estimate is shown if one imposes an additional condition of \emph{non-resonance}, similar to \cite{B3,COh,NS}. In fact, by working in the non-resonant classes, we are able to prove an estimate in a regime which allows us to go all the way down to the regularity of $L^2$, i.e. to $\alpha=0$. By studying the randomized Gross-Pitaevskii hierarchy in this, we show \eqref{Theorem3bound}.

\subsection{Organization of the paper}

In Section \ref{Notation}, we recall and define the relevant notation. In particular, in Subsection \ref{Notation1}, we recall the definition of the Fourier transform and differentiation of density matrices, whereas in Subsection \ref{Notation2}, we recall the definition of the collision operator and we give a precise definition of the randomized collision operator. The concept of a factorized solution is reviewed in Subsection \ref{Factorized solutions}.

In Section \ref{The randomized spacetime estimate}, we prove the main randomized spacetime estimate in the strong form given by \eqref{Theorem1bound2}. This is the content of Theorem \ref{Theorem 1}. In Subsection \ref{Alternative form of randomization}, we explore a different form of randomization, where the initial data is now random and we explain the difficulties of this approach in the context of the Gross-Pitaevskii hierarchy. 

Section \ref{Properties of the randomized Gross-Pitaevskii hierarchy} is devoted to the study of the randomized Gross-Pitaevskii hierarchy. More precisely, in Subsection \ref{Link with NLS}, we explore the connection between the randomized Gross-Pitaevskii hierarchy and the nonlinear Schr\"{o}dinger equation. In Remark \ref{OmegaBoardgame}, we comment on the boardgame argument in the context of the randomized Gross-Pitaevskii hierarchy.

The independently randomized Gross-Pitaevskii hierarchy is precisely defined in Section \ref{A new randomized hierarchy}. In Subsection \ref{Properties of randomizedGP2}, we explain some of the aspects in which this hierarchy is different from the (dependently) randomized Gross-Pitaevskii hierarchy as well as from the deterministic Gross-Pitaevskii hierarchy. In the study of the independently randomized Gross-Pitaevskii hierarchy, it is possible to apply the randomized spatial estimate given by Theorem \ref{Theorem 1}. This is done in Subsection \ref{Application of randomized estimate}. As a result, we obtain convergence to zero of a sequence of Duhamel terms in a low-regularity space containing a random component in Theorem \ref{Smallness Bound}. 

We revisit the (dependently) randomized Gross-Pitaevskii hierarchy in Section \ref{The randomized Gross-Pitaevskii hierarchy revisited}. In particular, in Subsection \ref{Difficulties arising from higher-order Duhamel expansions}, we give an example showing that the methods used to prove Theorem \ref{Theorem 1} do not apply to the study of higher-order Duhamel expansions. In Subsection \ref{The precise form of the Duhamel expansion term}, we explicitly write out a Duhamel expansion of order $\ell$. We also give a specific example of the expansion when $\ell=3$ for illustration purposes. In Subsection \ref{A special class of density matrices}, we define a class $\mathcal{N}$ of \emph{non-resonant} density matrices and, and we henceforth study Duhamel expansions corresponding to the dependently randomized Gross-Pitaevskii hierarchy of elements in this non-resonant class. More precisely, we prove a randomized spatial estimate for Duhamel expansions of arbitrary order starting from a density matrix in $\mathcal{N}$ in Subsection \ref{Estimate in N}. As a result, it is possible to deduce a convergence to zero of these Duhamel expansions in a low-regularity space containing a random component. This result is given in Theorem \ref{Smallness Bound 2} of Subsection \ref{The randomized Gross-Pitaevskii hierarchy in the class N}.

\subsection{Acknowledgements}
The authors would like to thank Philip Gressman and Antti Knowles for valuable discussions. They would also like to thank Hrvoje \v{S}iki\'{c} for several useful comments. They are grateful to Martin Zwierlein for directing their attention to the results \cite{Zwierlein1,Zwierlein2} in the experimental physics literature. The authors would like to thank the referee for their constructive suggestions. V. S. was partially supported by a Simons Postdoctoral Fellowship. G. S. was partially supported by NSF Grant DMS-1068815.

\section{Notation}
\label{Notation}
Let us first introduce some notation. Given two positive quantities $A$ and $B$, we write $A \lesssim B$ if there exists some constant $C>0$ such that $A \leq CB$. If the quantity $C$ depends on $q$, we write $C$ as $C(q)$. We write $A \leq C(q) B$ also as $A =\mathcal{O}_{q}(B)$ or as $A \lesssim_q B$. If $A \lesssim_q B$ and $B \lesssim_q A$, we write $A \sim_q B$. In our paper, we sometimes abbreviate the Gross-Pitaevskii hierarchy as the \emph{GP hierarchy}.

\subsection{Fourier transform and differentiation of density matrices}
\label{Notation1}

Throughout our paper, we fix the spatial domain $\Lambda$ to be the three-dimensional torus $\mathbb{T}^3$, unless it is otherwise specified.
Given $f \in L^2(\Lambda)$ and $\xi \in \mathbb{Z}^3$, we define the Fourier transform of $f$ evaluated at $\xi$ by:
$$\widehat{f}(\xi):=\int_{\Lambda} f(x) e^{- i \langle x, \xi\rangle} dx.$$
Here, the quantity $\langle \cdot, \cdot \rangle$ denotes the inner product on $\mathbb{R}^3$.
\\
When considering density matrices $\gamma_0^{(k)}: \Lambda^k \times \Lambda^k \rightarrow \mathbb{C}$, we use the same convention as in \cite{GSS} to define:
$$(\gamma_0^{(k)})\,\,\widehat{}\,\,(\vec{\xi}_k;\vec{\xi}'_k): = \int_{\Lambda^k \times \Lambda^k} \gamma_0^{(k)}(\vec{x}_k;\vec{x}'_k)
e^{-i \cdot \sum_{j=1}^{k} \langle x_j, \xi_j \rangle + i \cdot \sum_{j=1}^{k} \langle x'_j, \xi'_j \rangle} d\vec{x}_k \, d\vec{x}'_k$$
for $\vec{\xi}_k=(\xi_1,\ldots,\xi_k),\vec{\xi}'_k=(\xi'_1,\ldots,\xi'_k) \in (\mathbb{Z}^3)^k$.
In this way, the definition of the Fourier transform of density matrices is consistent with the definition of factorized solutions for the GP hierarchy \eqref{GrossPitaevskiiHierarchy}. We usually write the Fourier transform of the density matrix as $\widehat{\gamma_0}^{(k)}$.

The \emph{free evolution} corresponding to $i \partial_t + \big(\Delta_{\vec{x}_k}-\Delta_{\vec{x}'_k}\big)$, acting on density matrices of order $k$ is defined as:
$$\mathcal{U}^{(k)}(t)\gamma_0^{(k)}:=e^{it \sum_{j=1}^{k} \Delta_{x_j}} \gamma_0^{(k)} e^{-it \sum_{j=1}^{k} \Delta_{x_j'}}.$$ 
By construction, it is then the case that:
$$\Big(i \partial_t + (\Delta_{\vec{x}_k}-\Delta_{\vec{x}_k'})\Big)\,\mathcal{U}^{(k)}(t)\gamma_0^{(k)}=0.$$

Given a regularity parameter $\alpha$, we define the operator of \emph{differentiation of order $\alpha$} on density matrices of order $k$ by using the Fourier transform:
\begin{equation}
\label{FractionalDerivative}
\big(S^{(k,\alpha)} \gamma^{(k)}_0 \big)\,\,\widehat{}\,\,(\xi_1,\ldots,\xi_k;\xi'_1,\ldots,\xi'_k):=
\end{equation}
$$\langle \xi_1 \rangle^{\alpha} \cdots \langle \xi_k \rangle^{\alpha} \cdot \langle \xi'_1 \rangle^{\alpha} \cdots \langle \xi'_k \rangle^{\alpha} \cdot \widehat{\gamma_0}^{(k)} (\xi_1, \ldots, \xi_k;\xi'_1, \ldots, \xi'_k).$$ 
Here, we define the \emph{Japanese bracket} to be 
$$\langle x \rangle : = \sqrt{1+|x|^2}.$$
In other words, we can write:
\begin{equation}
\notag
S^{(k,\alpha)} \gamma_0^{(k)}= \prod_{j=1}^k (1-\Delta_{x_j})^{\frac{\alpha}{2}} (1-\Delta_{x'_j})^{\frac{\alpha}{2}} \gamma_0^{(k)}.
\end{equation}
\subsection{The collision operator and the randomized collision operator}
\label{Notation2}
Let us recall the definition of the \emph{collision operator} $B_{j,k+1}$ for $k \in \mathbb{N}$ and for $j \in \{1,\ldots,k\}$. The operator $B_{j,k+1}$ acts on density matrices of order $k+1$ and it is defined by:
$$B_{j,k+1}=B^{+}_{j,k+1}-B^{-}_{j,k+1},$$
where:
\begin{equation}
\notag
B^{+}_{j,k+1}\big(\gamma_0^{(k+1)}\big)(\vec{x}_k; \vec{x}_k'):=\int_{\Lambda} \delta(x_j-x_{k+1}) \gamma_0^{(k+1)}(\vec{x}_k,x_{k+1};\vec{x}_k',x_{k+1})\, dx_{k+1}
\end{equation}
and
\begin{equation}
\notag
B^{-}_{j,k+1}\big(\gamma_0^{(k+1)}\big)(\vec{x}_k; \vec{x}_k'):=\int_{\Lambda} \delta(x_j'-x_{k+1}) \gamma_0^{(k+1)}(\vec{x}_k,x_{k+1};\vec{x}_k',x_{k+1})\, dx_{k+1}.
\end{equation}
Hence, for example, when $j=1$:
$$B_{1,k+1}\big(\gamma^{(k+1)}_0\big)(\vec{x}_k;\vec{x}_k')= \int_{\Lambda} \gamma^{(k+1)}_0 (\vec{x}_k,x_1;\vec{x}_k',x_1)\,dx_1-\int_{\Lambda} \gamma^{(k+1)}_0 (\vec{x}_k,x'_1;\vec{x}_k',x'_1)\,dx'_1.$$
The calculation for general $j \in \{1,\ldots,k\}$ is similar.
In particular, we note that $B_{j,k+1}$ is a linear map which takes density matrices of order $k+1$ to density matrices of order $k$.

One can then compute that the Fourier transform of $B^{+}_{1,k+1} \big(\gamma_0^{(k+1)}\big)$ is given by:
\begin{equation}
\label{FourierTransform1}
(B^{+}_{1,k+1} \gamma_0^{(k+1)})\,\,\widehat{}\,\,(\vec{\xi}_k; \vec{\xi'}_k)= 
\end{equation}
$$\sum_{{\xi}_{k+1}, \, {\xi'}_{k+1} \in \mathbb{Z}^3} \, \widehat{\gamma_0}^{(k+1)}(\xi_1-\xi_{k+1}+\xi'_{k+1}, \xi_2, \ldots, \xi_k, \xi_{k+1}; \xi'_1,\ldots, \xi'_k,\xi'_{k+1}).$$
Furthermore, the Fourier transform of $B^{-}_{1,k+1} \gamma_0^{(k+1)}$ is given by:
\begin{equation}
\label{FourierTransform2}
(B^{-}_{1,k+1} \gamma_0^{(k+1)})\,\,\widehat{}\,\,(\vec{\xi}_k; \vec{\xi'}_k)= 
\end{equation} 
$$\sum_{{\xi}_{k+1}, \, {\xi'}_{k+1} \in \mathbb{Z}^3} 
\,\widehat{\gamma_0}^{(k+1)}(\xi_1, \ldots, \xi_k, \xi_{k+1}; \xi'_1-\xi'_{k+1}+\xi_{k+1}, \xi'_2, \ldots, \xi'_k,\xi'_{k+1}).$$
We note that in the calculation for $B^{-}_{1,k+1}$ the $\xi_{k+1}$ and $\xi'_{k+1}$ come with a different sign, which is a consequence of our definition of the Fourier transform of a density matrix. The calculation for general $B^{\pm}_{j,k+1}$ is similar.

We now \emph{randomize} the collision operator $B_{j,k+1}$ in order to obtain the \emph{randomized collision operator} $[B_{j,k+1}]^{\omega}$. In order to define the latter operator, we use the Fourier transform. 
Let us take $(h_{\xi})_{\xi \in \mathbb{Z}^3}$ to be a sequence of independent, identically distributed Bernoulli random variables with expectation zero and standard deviation $1$. 
We will use this notation from now on, unless we specify otherwise.
Throughout our paper, the probability space associated to this sequence of random variables is denoted by $(\Omega,\Sigma,p)$, where $\Sigma$  is the corresponding sigma-algebra and where $p$ is the probability measure. We will usually denote the probability space just by $\Omega$.


With the above notation, we define, for fixed $\omega \in \Omega$ and for a fixed density matrix $\gamma^{(k+1)}_0$ of order $k+1$:
\begin{equation}
\label{Bjkomega}
([B^{+}_{1,k+1}]^{\omega} \gamma_0^{(k+1)})\,\,\widehat{}\,\,(\vec{\xi}_k; \vec{\xi}_k'):= 
\end{equation}
$$ \sum_{\xi_{k+1},\,\xi'_{k+1} \in \mathbb{Z}^3} h_{\xi_1}(\omega) \cdot h_{\xi_1-\xi_{k+1}+\xi'_{k+1}}(\omega) \cdot h_{\xi_{k+1}}(\omega) \cdot h_{\xi'_{k+1}}(\omega) 
$$
$$\cdot \, \widehat{\gamma_0}^{(k+1)}(\xi_1-\xi_{k+1}+\xi'_{k+1}, \xi_2, \ldots, \xi_k, \xi_{k+1}; \xi'_1,\ldots, \xi'_k,\xi'_{k+1}).$$
In other words, we are just randomizing a subset of frequencies in \eqref{FourierTransform1}.
\\
Analogously, we use \eqref{FourierTransform2} and we define $[B^{-}_{1,k+1}]^{\omega}$ by:

\begin{equation}
\label{Bjkomega2}
([B^{-}_{1,k+1}]^{\omega} \gamma_0^{(k+1)})\,\,\widehat{}\,\,(\vec{\xi}_k; \vec{\xi}_k'):= 
\end{equation}
$$ \sum_{\xi_{k+1},\,\xi'_{k+1} \in \mathbb{Z}^3} h_{\xi_1}(\omega) \cdot h_{\xi_1-\xi'_{k+1}+\xi_{k+1}}(\omega) \cdot h_{\xi_{k+1}}(\omega) \cdot h_{\xi'_{k+1}}(\omega) 
$$
$$\cdot \, \widehat{\gamma_0}^{(k+1)}(\xi_1, \ldots, \xi_k, \xi_{k+1}; \xi'_1-\xi'_{k+1}+\xi_{k+1}, \xi'_2, \ldots, \xi'_k,\xi'_{k+1}).$$

The quantity $[B^{\pm}_{j,k+1}]^{\omega}$ is defined similarly, when $1 \leq j \leq k$. We can now define the \emph{randomized collision operator} as:
\begin{equation}
\label{Bjkrandomized}
[B_{j,k+1}]^{\omega}:=[B^{+}_{j,k+1}]^{\omega}-[B^{-}_{j,k+1}]^{\omega}.
\end{equation}
By construction, it follows that $[B_{j,k+1}]^{\omega}=B_{j,k+1}$, if $\omega \in \Omega$ is chosen such that $h_{\xi}(\omega)=1$ for all $\xi \in \mathbb{Z}^3$ or if $h_{\xi}(\omega)=-1$ for all $\xi \in \mathbb{Z}^3$.

We define the \emph{full randomized collision operator} by:
\begin{equation}
\label{Brandomized}
[B^{(k)}]^{\omega}:= \sum_{j=1}^{k} \, [B_{j,k+1}]^{\omega}.
\end{equation}

As a convention, we will extend the definition of the above collision operators which act on density matrices of order $k+1$ (i.e. $B_{j,k+1}$ and $\big[B_{j,k+1}\,\big]^{\omega}$) to density matrices $\sigma^{(\ell)}$ of order $\ell>k+1$. This is done by acting only in the variables $\vec{x}_{k+1}$ and $\vec{x}'_{k+1}$. We will use this convention in the discussion in Subsection \ref{The precise form of the Duhamel expansion term}.

We will sometimes write a density matrix $\gamma_0^{(k)}$ in terms of the standard Fourier basis of $L^2(\Lambda^{k} \times \Lambda^{k})$ as follows:

\begin{equation}
\label{Fourierbasis}
\gamma_0^{(k)}=\sum_{\ell_r, j_r \in \mathbb{Z}^3} a^{(k+1)}_{\ell_1,\ldots,\ell_{k}; \,j_1, \ldots, j_{k}} \cdot b^{(k+1)}_{\ell_1,\ldots,\ell_{k}; \,j_1, \ldots, j_{k}}
\end{equation}
Here, $a^{(k)}_{\ell_1,\ldots,\ell_{k};\, j_1, \ldots, j_{k}} = \widehat{\gamma_0}^{(k)}(\ell_1,\ldots,\ell_{k}; j_1, \ldots, j_{k})$ denotes the Fourier coefficient and $b^{(k)}_{\ell_1,\ldots,\ell_{k};\, j_1, \ldots, j_{k}}(\vec{x}_{k+1}; \vec{x}_{k+1}')=e^{i (x_1 \cdot \ell_1 + \cdots + x_{k} \cdot \ell_{k})-i(x'_1 \cdot j_1 + \cdots +  x'_{k} \cdot j_{k})}$ denotes the canonical Fourier basis element. 

\subsection{Factorized solutions}
\label{Factorized solutions}

Let us now precisely explain the notion of a \emph{factorized solution} to the Gross-Pitaevskii hierarchy \eqref{GrossPitaevskiiHierarchy}. Suppose that the function $\phi=\phi(x,t)$ is a solution to the defocusing cubic nonlinear Schr\"{o}dinger equation on $\Lambda$:

\begin{equation}
\notag
\begin{cases}
i\partial_t \phi + \Delta \phi = |\phi|^2 \phi \\
\phi \big|_{t=0}=\phi_0.
\end{cases}
\end{equation}
The \emph{Dirac bracket} $| \cdot \rangle \langle \cdot|$ is defined as $|f \rangle \langle g| (x,x'):=f(x) \cdot \overline{g(x')}$, for functions $f,g: \Lambda \rightarrow \mathbb{C}$.
Then, using the definition of the collision operator above, it can be shown that the sequence $(\gamma^{(k)})_k=(\gamma^{(k)}(t))_k$ of time-dependent density matrices given by:
\begin{equation}
\notag
\gamma^{(k)}(t,\vec{x}_k;\vec{x}_k'):=\prod_{j=1}^k \phi(t,x_j) \overline{\phi(t,x_j')}=|\phi \rangle \langle \phi|^{\otimes k}(t,\vec{x}_k;\vec{x}_k')
\end{equation}
solves \eqref{GrossPitaevskiiHierarchy} with the initial data $\gamma_0^{(k)}=|\phi_0 \rangle \langle \phi_0|^{\otimes k}$. We define these to be the \emph{factorized solutions} of the GP hierarchy. 

\section{The randomized spacetime estimate}
\label{The randomized spacetime estimate}
In this section, we prove the main estimate for the randomized collision operator defined in \eqref{Bjkrandomized}. As we will see, the randomization will let us extend the range of regularity exponent.
Let us recall, that in \cite{GSS}, it was proved that for the deterministic collision operator, for all $\alpha>1$, there exists $C>0$ such that for all $k \in \mathbb{N}$ and for all $1 \leq j \leq k$:

$$\big\|S^{(k,\alpha)}B_{j,k+1}\,\mathcal{U}^{(k+1)}(t)\,\gamma_{0}^{(k+1)}\big\|_{L^2([0,2\pi] \times \Lambda^k \times \Lambda^k)} \leq C\big\|S^{(k+1,\alpha)}\gamma_{0}^{(k+1)}\big\|_{L^2(\Lambda^{k+1} \times \Lambda^{k+1})}.$$
The $C$ depends on $\alpha$, but is independent of $j$ and $k$. The condition $\alpha>1$ is sharp for general density matrices. If we took factorized density matrices, we saw that we could take $\alpha=1$.

We now see how the range of $\alpha$ can be extended if we replace $B_{j,k+1}$ by $[B_{j,k+1}]^{\omega}$. In doing so, we also have to add an $\Omega$ component to the norm on the left-hand side. The main result of this section is the following stronger result:

\begin{theorem}
\label{Theorem 1} Let $\alpha>\frac{3}{4}$ be given. For all $k \in \mathbb{N}$ and for all $1 \leq j \leq k$, the following bound holds:
\begin{equation}
\notag
\big\|S^{(k,\alpha)}[B_{j,k+1}]^{\omega} \gamma_{0}^{(k+1)}\big\|_{L^2(\Omega \times \Lambda^k \times \Lambda^k)} \leq C_0 \big\|S^{(k+1,\alpha)} \gamma_0^{(k+1)}\big\|_{L^2(\Lambda^{k+1} \times \Lambda^{k+1})}.
\end{equation}
The constant $C_0>0$ depends on $\alpha$, but it is independent of $j,k$.
\\
In particular, for $[B^{(k+1)}]^{\omega}$ as defined in \eqref{Brandomized}, it follows that:
\begin{equation}
\notag
\big\|S^{(k,\alpha)} \, [B^{(k+1)}]^{\omega} \gamma_{0}^{(k+1)}\big\|_{L^2(\Omega \times \Lambda^k \times \Lambda^k)} \leq C_0 \,k \, \big\|S^{(k+1,\alpha)} \gamma_0^{(k+1)}\big\|_{L^2(\Lambda^{k+1} \times \Lambda^{k+1})}.
\end{equation}

\end{theorem}

From Theorem \ref{Theorem 1}, Fubini's Theorem and from the unitarity of $\mathcal{U}^{(k+1)}(t)$, we can deduce:

\begin{corollary}
\label{Corollary 1}
Let $\alpha>\frac{3}{4}$ be given.
For all $T>0$, $k \in \mathbb{N}$ and $1 \leq j \leq k$, the following bound holds:
\begin{equation}
\notag
\big\|S^{(k,\alpha)} \, [B_{j,k+1}]^{\omega} \, \mathcal{U}^{(k+1)}(t) \, \gamma_{0}^{(k+1)}\big\|_{L^2(\Omega \times [0,T] \times \Lambda^k \times \Lambda^k)} \leq C_0 \sqrt{T} \big\|S^{(k+1,\alpha)} \gamma_0^{(k+1)}\big\|_{L^2(\Lambda^{k+1} \times \Lambda^{k+1})}.
\end{equation}
Here, $C_0>0$ is the constant from Theorem \ref{Theorem 1}.
Hence, we can deduce that:
\begin{equation}
\notag
\big\|S^{(k,\alpha)} \, [B^{(k+1)}]^{\omega} \, \mathcal{U}^{(k+1)}(t) \, \gamma_{0}^{(k+1)}\big\|_{L^2(\Omega \times [0,T] \times \Lambda^k \times \Lambda^k)} \leq C_0 \, k \, \sqrt{T} \big\|S^{(k+1,\alpha)} \gamma_0^{(k+1)}\big\|_{L^2(\Lambda^{k+1} \times \Lambda^{k+1})}.
\end{equation}

\end{corollary}
Corollary \ref{Corollary 1} gives a direct extension of the spacetime bound from \cite{GSS} in the randomized setting. In this setting, we see that we can prove a bound when the regularity exponent satisfies $\alpha>\frac{3}{4}$. This is an improvement over the result in the deterministic setting. Let us remark, that in the proof of the deterministic spacetime bound in our previous joint work with Gressman, we encountered a sum of \emph{diagonal terms}, which was bounded precisely when $\alpha>\frac{3}{4}$. It was in the off-diagonal terms that the regularity was lost. We refer the interested reader to \cite{GSS} for more details.

By using Markov's inequality and Corollary \ref{Corollary 1}, we can deduce:

\begin{corollary}
\label{Corollary 2}
For $\alpha>\frac{3}{4}$, $T>0$, $k \in \mathbb{N}$, $1 \leq j \leq k$, and $\lambda>0$ the following bound holds:
\begin{equation}
\notag
\mathbb{P}\Big(\big\|S^{(k,\alpha)} \, [B_{j,k+1}]^{\omega} \, \mathcal{U}^{(k+1)}(t) \, \gamma_{0}^{(k+1)}\big\|_{L^2([0,T] \times \Lambda^k \times \Lambda^k)} \geq \lambda \Big) 
\end{equation}
$$\leq \frac{C_0^2 \, T \, \big\|S^{(k+1,\alpha)} \gamma_0^{(k+1)}\big\|_{L^2(\Lambda^{k+1} \times \Lambda^{k+1})}^2}{\lambda^2}\,\,.$$
Here, $C_0>0$ is the constant from Theorem \ref{Theorem 1}.
Moreover,
\begin{equation}
\notag
\mathbb{P}\Big(\big\|S^{(k,\alpha)} \, [B^{(k+1)}]^{\omega} \, \mathcal{U}^{(k+1)}(t) \, \gamma_{0}^{(k+1)}\big\|_{L^2([0,T] \times \Lambda^k \times \Lambda^k)} \geq \lambda\Big)
\end{equation}
$$\leq \frac{C_0^2 \, k^2 \, T \big\|S^{(k+1,\alpha)} \gamma_0^{(k+1)}\big\|_{L^2(\Lambda^{k+1} \times \Lambda^{k+1})}^2}{\lambda^2}\,\,.$$
\end{corollary}

\begin{remark}
\label{notexponential}
	Let us note that we are just considering the analogue of \eqref{GainOfIntegrability} when $p=2$ in the context of the Gross-Pitaevskii hierarchy. This already allows us to improve the range of $\alpha$ from the deterministic case. As we will see, the calculations when $p=2$ are already quite involved. This is due to the fact that there are many possible pairings in the proof of Theorem \ref{Theorem 1}. It is possible that one can extend this calculation to the case when $p$ is an arbitrary even number, which would give us an exponential upper bound on the right-hand side of the estimate in Corollary \ref{Corollary 2}, see \cite{BT1}.  The calculations for higher $p$ would become extremely difficult. We will not pursue this issue in our paper.
	
\end{remark}

We note that Corollary \ref{Corollary 2} gives a large deviation bound for the quantity that one is interested in looking at in the deterministic setting. As was noted in the introduction, this is one of the key steps in the use of randomization techniques in the analysis of nonlinear dispersive PDE. Since we would typically like to apply the spacetime estimate many times, we will see that the averaged estimate given in Corollary \ref{Corollary 1} will be easier to apply in the discussion that follows.

The proof of Theorem \ref{Theorem 1} is based on the pairing of the frequencies, similarly as in the proof of \eqref{GainOfIntegrability}, which was explained in the introduction. There are many more cases to consider in our setting. In some of the cases, we use shorthand matrix notation which is defined in Definition \ref{matrixdefinition} below.
\\
\\
We now prove Theorem \ref{Theorem 1}.

\begin{proof}
It suffices to consider the special case when $j=1$. The general case follows by analogous arguments. Moreover, we will prove the bound in the theorem for the operator $[B^{+}_{1,k+1}]^{\omega}$. The bound for $[B^{-}_{1,k+1}]^{\omega}$ is proved in the same way, and we thus obtain the claim for $[B_{1,k+1}]^{\omega}=[B^{+}_{1,k+1}]^{\omega}-[B^{-}_{1,k+1}]^{\omega}$.
\\
\\
We compute:

$$\Big(S^{(k,\alpha)}[B^{+}_{1,k+1}]^{\omega} \gamma_0^{(k+1)}\Big)\,\,\widehat{}\,\,(\vec{\xi}_k; \vec{\xi}_k')=$$ 
$$\prod_{j=1}^{k} \langle \xi_j \rangle^{\alpha} \cdot \langle \xi'_j \rangle^{\alpha}\,\cdot \Big(\sum_{\eta_1,\eta'_1 \in \mathbb{Z}^3}  h_{\xi_1}(\omega) \cdot h_{\xi_1-\eta_1+\eta'_1}(\omega) \cdot h_{\eta_1}(\omega) \cdot h_{\eta'_1}(\omega) 
$$
$$\cdot \, \widehat{\gamma_0}^{(k+1)}(\xi_1-\eta_1+\eta'_1, \xi_2, \ldots, \xi_k, \eta_1; \xi'_1,\ldots, \xi'_k,\eta'_1) \Big)$$
which by \eqref{Fourierbasis} equals
$$\prod_{j=1}^{k} \langle \xi_j \rangle^{\alpha} \cdot \langle \xi'_j \rangle^{\alpha}\,\cdot \Big(\sum_{\eta_1,\eta'_1 \in \mathbb{Z}^3}  h_{\xi_1}(\omega) \cdot h_{\xi_1-\eta_1+\eta'_1}(\omega) \cdot h_{\eta_1}(\omega) \cdot h_{\eta'_1}(\omega) 
$$
$$\cdot \, a^{(k+1)}_{\xi_1-\eta_1+\eta'_1, \xi_2, \ldots, \xi_k, \eta_1; \xi'_1,\ldots, \xi'_k,\eta'_1)} \Big).$$
Consequently: 
$$\big|(S^{(k,\alpha)}[B^{+}_{1,k+1}]^{\omega} \gamma_0^{(k+1)})\,\,\widehat{}\,\,(\vec{\xi}_k; \vec{\xi}_k')\big|^2=$$
$$=\prod_{j=1}^{k} \langle \xi_j \rangle^{2\alpha} \cdot \langle \xi_j' \rangle^{2\alpha} \cdot \Big(\sum_{\eta_1, \eta_1', \eta_2, \eta_2' \in \mathbb{Z}^3}
h_{\xi_1}(\omega) \cdot h_{\xi_1-\eta_1+\eta_1'}(\omega) \cdot h_{\eta_1}(\omega) \cdot h_{\eta_1'}(\omega)\cdot
$$
$$\overline{h}_{\xi_1}(\omega) \cdot \overline{h}_{\xi_1-\eta_2+\eta_2'}(\omega) \cdot \overline{h}_{\eta_2}(\omega) \cdot \overline{h}_{\eta_2'}(\omega) \cdot a^{(k+1)}_{\xi_1-\eta_1+\eta'_1, \xi_2, \ldots, \xi_k, \eta_1; \xi'_1,\ldots, \xi'_k,\eta'_1} \cdot$$
$$\overline{a}^{(k+1)}_{\xi_1-\eta_2+\eta'_2, \xi_2, \ldots, \xi_k, \eta_2; \xi'_1,\ldots, \xi'_k,\eta'_2}
\, \Big).$$
Since we are using Bernoulli random variables, we note that $\overline{h}=h$ and $h^2=1$. Hence, the previous expression equals:
$$\prod_{j=1}^{k} \langle \xi_j \rangle^{2\alpha} \cdot \langle \xi_j' \rangle^{2\alpha} \cdot \Big(\sum_{\eta_1, \eta_1', \eta_2, \eta_2' \in \mathbb{Z}^3}
h_{\xi_1-\eta_1+\eta_1'}(\omega) \cdot h_{\eta_1}(\omega) \cdot h_{\eta_1'}(\omega)\cdot h_{\xi_1-\eta_2+\eta_2'}(\omega) \cdot h_{\eta_2}(\omega) \cdot h_{\eta_2'}(\omega) \cdot 
$$
$$a^{(k+1)}_{\xi_1-\eta_1+\eta'_1, \xi_2, \ldots, \xi_k, \eta_1; \xi'_1,\ldots, \xi'_k,\eta'_1} \cdot \overline{a}^{(k+1)}_{\xi_1-\eta_2+\eta'_2, \xi_2, \ldots, \xi_k, \eta_2; \xi'_1,\ldots, \xi'_k,\eta'_2}
\, \Big)=$$
$$=\langle \xi_1 \rangle^{2\alpha} \cdot \Big(\sum_{\eta_1, \eta_1', \eta_2, \eta_2' \in \mathbb{Z}^3}
h_{\xi_1-\eta_1+\eta_1'}(\omega) \cdot h_{\eta_1}(\omega) \cdot h_{\eta_1'}(\omega)\cdot h_{\xi_1-\eta_2+\eta_2'}(\omega) \cdot h_{\eta_2}(\omega) \cdot h_{\eta_2'}(\omega) \cdot 
$$
$$\big[\langle \xi_2 \rangle^{\alpha} \cdot \cdots \cdot \langle \xi_k \rangle^{\alpha} \cdot \langle \xi'_1 \rangle^{\alpha} \cdot \cdots \cdot \langle \xi'_k \rangle^{\alpha} \cdot a^{(k+1)}_{\xi_1-\eta_1+\eta'_1, \xi_2, \ldots, \xi_k, \eta_1; \xi'_1,\ldots, \xi'_k,\eta'_1}\big]\cdot
$$
$$
\big[\langle \xi_2 \rangle^{\alpha} \cdot \cdots \cdot \langle \xi_k \rangle^{\alpha} \cdot \langle \xi'_1 \rangle^{\alpha} \cdot \cdots \cdot \langle \xi'_k \rangle^{\alpha} \cdot \overline{a}^{(k+1)}_{\xi_1-\eta_2+\eta'_2, \xi_2, \ldots, \xi_k, \eta_2; \xi'_1,\ldots, \xi'_k,\eta'_2} \big]
\, \Big).$$

By Plancherel's Theorem in the $x$-variables, it follows that:
\begin{equation}
\notag
I:=\big\|S^{(k,\alpha)}[B_{1,k+1}^{+}]^{\omega} \gamma_{0}^{(k+1)}\big\|_{L^2(\Omega \times \Lambda^k \times \Lambda^k)}^2=
\end{equation}
$$\sum_{\xi_1,\ldots,\xi_k, \xi'_1,\ldots,\xi'_k \in \mathbb{Z}^3}
\langle \xi_1 \rangle^{2\alpha} \cdot \Big(\sum_{\eta_1, \eta_1', \eta_2, \eta_2' \in \mathbb{Z}^3} \int_{\Omega}
h_{\xi_1-\eta_1+\eta_1'}(\omega) \cdot h_{\eta_1}(\omega) \cdot h_{\eta_1'}(\omega)\cdot h_{\xi_1-\eta_2+\eta_2'}(\omega) \cdot h_{\eta_2}(\omega) \cdot h_{\eta_2'}(\omega) \cdot 
$$
$$\big[\langle \xi_2 \rangle^{\alpha} \cdots \langle \xi_k \rangle^{\alpha} \cdot \langle \xi'_1 \rangle \cdots \langle \xi'_k \rangle^{\alpha} \cdot a^{(k+1)}_{\xi_1-\eta_1+\eta'_1, \xi_2, \ldots, \xi_k, \eta_1; \xi'_1,\ldots, \xi'_k,\eta'_1}\big]\cdot
$$
$$
\big[\langle \xi_2 \rangle^{\alpha} \cdots \langle \xi_k \rangle^{\alpha} \cdot \langle \xi'_1 \rangle \cdots \langle \xi'_k \rangle^{\alpha} \cdot \overline{a}^{(k+1)}_{\xi_1-\eta_2+\eta'_2, \xi_2, \ldots, \xi_k, \eta_2; \xi'_1,\ldots, \xi'_k,\eta'_2} \big]
\, dp(\omega) \Big).$$
\\
This is the expression that we want to estimate.
By the independence of the $(h_{\xi})$, and by the fact that they all have mean zero, it follows that each element in the set $\{\xi_1-\eta_1+\eta_1',\eta_1,\eta_1',\xi_1-\eta_2+\eta_2',\eta_2,\eta_2'\}$ occurs at least twice in the list $(\xi_1-\eta_1+\eta_1',\eta_1,\eta_1',\xi_1-\eta_2+\eta_2',\eta_2,\eta_2')$. 

Let's call this property $\textbf{(*)}$.

Since we are working with Bernoulli random variables, we note that, by the triangle inequality:

\begin{equation}
\label{Skalpha}
I \leq \mathop{\sum_{\xi_1,\ldots,\xi_k, \xi'_1,\ldots,\xi'_k, \eta_1, \eta_1', \eta_2, \eta_2' \in \mathbb{Z}^3}}^{\textbf{(*)}} 
\end{equation}
$$\Big[\langle \xi_1-\eta_1+\eta_1' \rangle^{\alpha} \cdot \langle \xi_2 \rangle^{\alpha} \cdots \langle \xi_k \rangle^{\alpha} \cdot \langle \eta_1 \rangle^{\alpha} \cdot \langle \xi'_1 \rangle^{\alpha} \cdots \langle \xi'_k \rangle^{\alpha} \cdot \langle \eta_1' \rangle^{\alpha} \cdot \big|a^{(k+1)}_{\xi_1-\eta_1+\eta'_1, \xi_2, \ldots, \xi_k, \eta_1; \xi'_1,\ldots, \xi'_k,\eta'_1}\big|\Big]\cdot
$$
$$
\Big[\langle \xi_1-\eta_2+\eta_2' \rangle^{\alpha} \cdot\langle \xi_2 \rangle^{\alpha} \cdots \langle \xi_k \rangle^{\alpha} \cdot \langle \eta_2 \rangle^{\alpha} \cdot \langle \xi'_1 \rangle^{\alpha} \cdots \langle \xi'_k \rangle^{\alpha} \cdot \langle \eta_2' \rangle^{\alpha} \cdot \big|\overline{a}^{(k+1)}_{\xi_1-\eta_2+\eta'_2, \xi_2, \ldots, \xi_k, \eta_2; \xi'_1,\ldots, \xi'_k,\eta'_2}\big| \Big]
$$
Let us denote the quantity on the right-hand side by $J$.

Here, 
$$``\mathop{\sum_{\xi_1,\ldots,\xi_k, \xi'_1,\ldots,\xi'_k, \eta_1, \eta_1', \eta_2, \eta_2' \in \mathbb{Z}^3}}^{\textbf{(*)}}"$$ 
denotes the sum over all $\xi_1,\ldots,\xi_k, \xi'_1,\ldots,\xi'_k, \eta_1, \eta_1', \eta_2, \eta_2' \in \mathbb{Z}^3$ satisfying the condition $\textbf{(*)}$ above.
In particular, we see that we need to bound the sum in the frequencies under the given constraint.

We now need to consider all of the possible cases separately. This requires a combinatorial analysis of all the possible pairings. In order to simplify the notation, we develop a matrix notation for the obtained sums. The notation is explained on an example in Definition \ref{matrixdefinition} and is used in order to analyze the subsequent cases.
\\
\\
\textbf{Big Case I)} $\mathbf{\xi_1-\eta_1+\eta_1'=\xi_1-\eta_2+\eta_2'}.$
\\
\\
In this big case, we obtain that: 
\begin{equation}
\label{BigCase1}
\eta_1'+\eta_2=\eta_1+\eta_2'.
\end{equation}
We now need to consider several possibilities:
\\
\\
\textbf{Case IA)} $\{\eta_1,\eta_1',\eta_2,\eta_2'\}$ has cardinality $2$; each element in the list $(\eta_1,\eta'_1,\eta_2,\eta'_2)$ occurs exactly twice.
\\
\\
Let us observe that:
$$\langle \xi_1 \rangle^{\alpha} \lesssim \langle \xi_1-\eta_1+\eta_1' \rangle^{\alpha} \cdot \langle \eta_1 \rangle^{\alpha} \cdot \langle \eta_1' \rangle^{\alpha}$$
and
$$\langle \xi_1 \rangle^{\alpha} \lesssim \langle \xi_1-\eta_2+\eta_2' \rangle^{\alpha} \cdot \langle \eta_2 \rangle^{\alpha} \cdot \langle \eta_2' \rangle^{\alpha}.$$
We henceforth refer to such inequalities as the \emph{fractional Leibniz Rule}.

It follows that the corresponding contribution to $J$ (defined after \eqref{Skalpha}), is:

$$\lesssim 
\mathop{\sum_{\xi_1,\ldots,\xi_k, \xi'_1,\ldots,\xi'_k, \eta_1, \eta_1', \eta_2, \eta_2' \in \mathbb{Z}^3}}_{Case \,\,IA)} 
$$
$$\Big[\langle \xi_1-\eta_1+\eta_1' \rangle^{\alpha} \cdot \langle \xi_2 \rangle^{\alpha} \cdots \langle \xi_k \rangle^{\alpha} \cdot \langle \eta_1 \rangle^{\alpha} \cdot \langle \xi'_1 \rangle^{\alpha} \cdots \langle \xi'_k \rangle^{\alpha} \cdot \langle \eta_1' \rangle^{\alpha} \cdot \big|a^{(k+1)}_{\xi_1-\eta_1+\eta'_1, \xi_2, \ldots, \xi_k, \eta_1; \xi'_1,\ldots, \xi'_k,\eta'_1}\big|\Big]\cdot
$$
$$
\Big[\langle \xi_1-\eta_2+\eta_2' \rangle^{\alpha} \cdot\langle \xi_2 \rangle^{\alpha} \cdots \langle \xi_k \rangle^{\alpha} \cdot \langle \eta_2 \rangle^{\alpha} \cdot \langle \xi'_1 \rangle^{\alpha} \cdots \langle \xi'_k \rangle^{\alpha} \cdot \langle \eta_2' \rangle^{\alpha} \cdot \big|\overline{a}^{(k+1)}_{\xi_1-\eta_2+\eta'_2, \xi_2, \ldots, \xi_k, \eta_2; \xi'_1,\ldots, \xi'_k,\eta'_2}\big| \Big]
.$$
Here, by: $$``\mathop{\sum_{\xi_1,\ldots,\xi_k, \xi'_1,\ldots,\xi'_k, \eta_1, \eta_1', \eta_2, \eta_2' \in \mathbb{Z}^3}}_{Case \,\,IA)}\,\,"$$ we denote the sum over all frequencies which occur in Case IA). 

We need to consider two possibilities:
\\
\\
\textbf{IA1)} $\eta_1=\eta_2, \,\eta_1'=\eta_2'$. 
\\
\\
We use the Cauchy-Schwarz inequality in $\eta_1, \eta_1'$, as well as in $\xi_2,\ldots,\xi_k, \xi_1',\ldots,\xi_k'$ and we use Plancherel's Theorem in order to deduce that the corresponding contribution to $J$ is:
$$\lesssim \big\|S^{(k+1,\alpha)}\gamma^{(k+1)}_0\big\|_{L^2(\Lambda^{k+1} \times \Lambda^{k+1})}^2.$$
\textbf{IA2)} $\eta_1=\eta_2',\,\eta_1'=\eta_2$.
\\
\\
We note that $\xi_1-\eta_1+\eta'_1=\xi_1-\eta_2+\eta'_2$ implies that $\eta_1=\eta'_1=\eta_2=\eta'_2$. Hence, $\{\eta_1,\eta_1',\eta_2,\eta_2'\}$ has cardinality $1$. This contradicts the assumption that this set has cardinality $2$ and hence this case does not occur.
\\
\\
\textbf{IA3)} $\eta_1=\eta_1',\,\eta_2=\eta_2'$.
\\
\\
In this case, we need to estimate:

$$\tilde{J}:= \mathop{\sum_{\xi_1,\ldots,\xi_k, \xi'_1,\ldots,\xi'_k, \eta_1,\eta_2 \in \mathbb{Z}^3}} \Big[\langle \xi_1\rangle^{\alpha} \cdot \langle \xi_2 \rangle^{\alpha} \cdots \langle \xi_k \rangle^{\alpha} \cdot \langle \xi'_1 \rangle^{\alpha} \cdots \langle \xi'_k \rangle^{\alpha} \cdot \big|a^{(k+1)}_{\xi_1, \xi_2, \ldots, \xi_k, \eta_1; \xi'_1,\ldots, \xi'_k,\eta_1}\big|\Big]\cdot
$$
$$
\Big[\langle \xi_1\rangle^{\alpha} \cdot\langle \xi_2 \rangle^{\alpha} \cdots \langle \xi_k \rangle^{\alpha} \cdot \langle \xi'_1 \rangle^{\alpha} \cdots \langle \xi'_k \rangle^{\alpha} \cdot \big|\overline{a}^{(k+1)}_{\xi_1, \xi_2, \ldots, \xi_k, \eta_2; \xi'_1,\ldots, \xi'_k,\eta_2}\big| \Big]
$$
Now, we need to argue differently, since the fractional Leibniz Rule
$$\langle \xi_1 \rangle^{\alpha} = \langle \xi_1-\eta_1+\eta_1' \rangle^{\alpha} \lesssim \langle \xi_1 \rangle^{\alpha} \cdot \langle \eta_1 \rangle^{\alpha} \cdot \langle \eta_1' \rangle^{\alpha}$$ does not give us a good upper bound. We rather leave the factor $\langle \xi_1 \rangle^{\alpha}$ by itself. 



We can then rewrite $\tilde{J}$ as:
$$ \mathop{\sum_{\xi_1,\ldots,\xi_k, \xi'_1,\ldots,\xi'_k, \eta_1,\eta_2 \in \mathbb{Z}^3}} \langle \eta_1 \rangle^{-2\alpha} \cdot \langle \eta_2 \rangle^{-2\alpha} \cdot $$
$$\Big[\langle \xi_1\rangle^{\alpha} \cdot \langle \xi_2 \rangle^{\alpha} \cdots \langle \xi_k \rangle^{\alpha} \cdot \langle \eta_1 \rangle^{\alpha} \cdot \langle \xi'_1 \rangle^{\alpha} \cdots \langle \xi'_k \rangle^{\alpha} \cdot \langle \eta_1 \rangle^{\alpha} \cdot \big|a^{(k+1)}_{\xi_1, \xi_2, \ldots, \xi_k, \eta_1; \xi'_1,\ldots, \xi'_k,\eta_1}\big|\Big]\cdot
$$
$$
\Big[\langle \xi_1\rangle^{\alpha} \cdot\langle \xi_2 \rangle^{\alpha} \cdots \langle \xi_k \rangle^{\alpha} \cdot \langle \eta_2 \rangle^{\alpha} \cdot \langle \xi'_1 \rangle^{\alpha} \cdots \langle \xi'_k \rangle^{\alpha}  \cdot \langle \eta_2 \rangle^{\alpha} \cdot \big|\overline{a}^{(k+1)}_{\xi_1, \xi_2, \ldots, \xi_k, \eta_2; \xi'_1,\ldots, \xi'_k,\eta_2}\big| \Big]
$$
Let us note that, since $\alpha>\frac{3}{4}$, the sequence $(\langle q \rangle^{-2\alpha})_{q \in \mathbb{Z}^3}$ belongs to $\ell^2(\mathbb{Z}^3)$.
Hence, we can use the Cauchy-Schwarz inequality in $\xi_1,\ldots \xi_k, \xi'_1, \ldots, \xi'_k, \eta_1, \eta_2$ and Plancherel's Theorem in order to deduce that:
\\
\begin{equation}
\label{SCBbound}
\tilde{J} \lesssim \big\|S^{(k+1,\alpha)}\gamma_0^{(k+1)}\big\|_{L^2(\Lambda^{k+1} \times \Lambda^{k+1})}^2.
\end{equation}
\\
\\
\textbf{Case IB)} $\{\eta_1,\eta_1',\eta_2,\eta_2'\}$ has cardinality $2$; one element in the list $(\eta_1,\eta_1',\eta_2,\eta_2')$ occurs three times.
\\
\\
We recall from \eqref{BigCase1} that $\eta'_1+\eta_2=\eta_1+\eta'_2$. From this identity, it follows that if one element in the list $(\eta_1,\eta'_1,\eta_2,\eta'_2)$ occurs three times, then they all have to be equal. In this way, we obtain a contradiction to the assumption that $\{\eta_1,\eta'_1,\eta_2,\eta'_2\}$ has cardinality $2$.
\\
\\
\textbf{Case IC)} $\{\eta_1,\eta'_1,\eta_2,\eta'_2\}$ has cardinality $3$.
\\
\\
We recall property $\textbf{(*)}$, and the fact that in Big Case $I$, one has: $\xi_1-\eta_1+\eta'_1=\xi_1-\eta_2+\eta'_2$, from where it follows that each element in the list: $$(\eta_1, \eta_2, \eta'_1,\eta'_2)$$ either occurs at least twice or it equals $\xi_1-\eta_1+\eta'_1$. In this way, we get a contradiction to the assumption that $\{\eta_1,\eta'_1,\eta_2,\eta'_2\}$ has cardinality $3$. Namely, if $\{\eta_1,\eta'_1,\eta_2,\eta'_2\}$ had cardinality $3$, the two distinct elements of the list $(\eta_1,\eta'_1,\eta_2,\eta'_2)$ which occur only once would both have to equal $\xi_1-\eta_1+\eta'_1$.
\\
\\
\textbf{Case ID)} $\{\eta_1,\eta'_1,\eta_2,\eta'_2\}$ has cardinality $4$.
\\
\\
In this case, we obtain a contradiction in the same way as in case IC).
\\
\\
\textbf{Case IE)} $\{\eta_1,\eta'_1,\eta_2,\eta'_2\}$ has cardinality $1$.
\\
\\
In this case, we note that:
$$\xi_1-\eta_1+\eta'_1=\xi_1-\eta_2+\eta'_2=\xi_1$$
and
$$\eta_1=\eta'_1=\eta_2=\eta'_2.$$
Hence, we need to estimate:
$$\sum_{\xi_1,\ldots,\xi_k, \xi'_1,\ldots,\xi'_k, \eta_1 \in \mathbb{Z}^3} 
$$
$$\Big[\langle \xi_1\rangle^{\alpha} \cdot \langle \xi_2 \rangle^{\alpha} \cdots \langle \xi_k \rangle^{\alpha} \cdot \langle \xi'_1 \rangle^{\alpha} \cdots \langle \xi'_k \rangle^{\alpha} \cdot \big|a^{(k+1)}_{\xi_1, \xi_2, \ldots, \xi_k, \eta_1; \xi'_1,\ldots, \xi'_k,\eta_1}\big|\Big]\cdot
$$
$$
\Big[\langle \xi_1\rangle^{\alpha} \cdot\langle \xi_2 \rangle^{\alpha} \cdots \langle \xi_k \rangle^{\alpha} \cdot \langle \xi'_1 \rangle^{\alpha} \cdots \langle \xi'_k \rangle^{\alpha}  \cdot \big|\overline{a}^{(k+1)}_{\xi_1, \xi_2, \ldots, \xi_k, \eta_1; \xi'_1,\ldots, \xi'_k,\eta_1}\big| \Big]
$$
We can use the Cauchy-Schwarz inequality in $\xi_1,\ldots,\xi_k, \xi'_1,\ldots,\xi'_k,\eta_1 $ and Plancherel's Theorem to deduce that the above expression is, in absolute value $$\lesssim 
\|S^{(k+1,\alpha)} \gamma^{(k+1)}_0\|_{L^2(\Lambda^{k+1} \times \Lambda^{k+1})}^2.$$
Here, we only used the fact that $\alpha \geq 0$.

Consequently, we obtain the wanted bound in Big Case I.

In what follows, we may assume that:

\begin{equation}
\label{BigCaseIbound}
\xi_1-\eta_1+\eta'_1 \neq \xi_1-\eta_2+\eta'_2.
\end{equation}
By condition \textbf{(*)}, it follows that $\xi_1-\eta_1+\eta'_1$ and $\xi_1-\eta_2+\eta'_2$ each get paired with distinct elements of $\{\eta_1,\eta'_1,\eta_2,\eta'_2\}$.
\\
\\
\textbf{Big Case II)} $\mathbf{\eta_1=\eta'_1}$.
\\
\\
In this Big Case, we obtain:
$$\xi_1-\eta_1+\eta'_1=\xi_1.$$
Furthermore, the condition \eqref{BigCaseIbound} implies that:
\begin{equation}
\label{eta2eta'2}
\eta_2 \neq \eta'_2.
\end{equation}
By using the condition \textbf{(*)}, it follows that:
$$\eta_2,\eta'_2 \in \{\xi_1,\eta_1,\xi_1-\eta_2+\eta'_2\}$$
We need to consider several cases:
\\
\\
\textbf{Case IIA)} $\eta_2,\eta'_2 \in \{\xi_1,\eta_1\}$ .
\\
\\
In case IIA), we need to consider several subcases:
\\
\\
\textbf{IIA1)} $\eta_2=\xi_1, \eta'_2=\eta_1$.
\\
\\
In this subcase, we know: $\xi_1-\eta_2+\eta'_2=\xi_1-\xi_1+\eta_1=\eta_1$.

In other words:
\begin{equation}
\notag
\begin{cases}
\xi_1-\eta_1+\eta'_1=\xi_1\\
\xi_1-\eta_2+\eta'_2=\eta_1\\
\eta_1=\eta'_1\\
\eta_2=\xi_1\\
\eta'_2=\eta_1
\end{cases}
\end{equation}
It follows that we have to bound:
\begin{equation}
\label{encodesum}
\sum_{\xi_1,\ldots,\xi_k,\xi'_1,\ldots,\xi'_k, \eta_1 \in \mathbb{Z}^3} 
\end{equation}
$$
\Big[\langle \xi_1 \rangle^{\alpha} \cdot \langle \xi_2 \rangle^{\alpha} \cdots \langle \xi_k \rangle^{\alpha} \cdot \langle \eta_1 \rangle^{0} \cdot \langle \xi'_1 \rangle^{\alpha} \cdots \langle \xi'_k \rangle^{\alpha} \cdot \langle \eta_1 \rangle^{0} \cdot \big|a^{(k+1)}_{\xi_1, \xi_2, \ldots, \xi_k, \eta_1; \xi'_1,\ldots, \xi'_k, \eta_1}\big|\Big]\cdot 
$$
$$
\Big[\langle \eta_1 \rangle^{0} \cdot\langle \xi_2 \rangle^{\alpha} \cdots \langle \xi_k \rangle^{\alpha} \cdot \langle \xi_1 \rangle^{\alpha} \cdot \langle \xi'_1 \rangle^{\alpha} \cdots \langle \xi'_k \rangle^{\alpha} \cdot \langle \eta_1 \rangle^{0} \cdot \big|\overline{a}^{(k+1)}_{\eta_1, \xi_2, \ldots, \xi_k, \xi_1; \xi'_1,\ldots, \xi'_k,\eta_1}\big| \Big].
$$
We bound the above expression in absolute value by using the Cauchy-Schwarz inequality in $\xi_1, \eta_1$ and in $\xi_2,\ldots,\xi_k, \xi'_1, \ldots, \xi'_k$ to obtain the bound $\lesssim \|S^{(k+1,\alpha)}\gamma^{(k+1)}_0\|_{L^2(\Lambda^{k+1} \times \Lambda^{k+1})}^2$. 
\\
\\
In what follows, we note that the application of the Cauchy-Schwarz inequality in $\xi_1,\ldots,\xi_k, \xi'_1, \ldots, \xi'_k$ is always straightforward. 
\\
\\
Let us now develop some shorthand notation:

\begin{definition}
\label{matrixdefinition}
Let $\mathcal{C}$ be a list of variables in $\{\xi_1,\eta_1,\eta'_1,\eta_2,\eta'_2\}$. Furthermore, let $\mu_j;j=1,\ldots,6$ be linear combinations of elements of the list $\mathcal{C}$ and let $\alpha_j; j=1,\ldots,6$ be elements of $\{0,\alpha\}$.
We define the quantity:
\begin{equation}
\notag
\mathcal{C}\,; \begin{bmatrix}
\langle \mu_1 \rangle^{\alpha_1} & \langle \mu_2 \rangle^{\alpha_2} & \langle \mu_3 \rangle^{\alpha_3} \\
\langle \mu_4 \rangle^{\alpha_4} & \langle \mu_5 \rangle^{\alpha_5} & \langle \mu_6 \rangle^{\alpha_6} 
\end{bmatrix}
\end{equation}
to be the sum over the elements of the list $\mathcal{A}$ as well as over $\xi_2,\xi_3,\ldots,\xi_k,\xi'_1,\xi'_2,\ldots,\xi'_k$ of:
$$
\Big[\langle \mu_1 \rangle^{\alpha_1} \cdot \langle \xi_2 \rangle^{\alpha} \cdots \langle \xi_k \rangle^{\alpha} \cdot \langle \mu_2 \rangle^{\alpha_2} \cdot \langle \xi'_1 \rangle^{\alpha} \cdots \langle \xi'_k \rangle^{\alpha} \cdot \langle \mu_3 \rangle^{\alpha_3} \cdot \big|a^{(k+1)}_{\mu_1, \xi_2, \ldots, \xi_k, \mu_2; \xi'_1,\ldots, \xi'_k, \mu_3}\big|\Big]\cdot 
$$
$$
\Big[\langle \mu_4 \rangle^{\alpha_4} \cdot\langle \xi_2 \rangle^{\alpha} \cdots \langle \xi_k \rangle^{\alpha} \cdot \langle \mu_5 \rangle^{\alpha_5} \cdot \langle \xi'_1 \rangle^{\alpha} \cdots \langle \xi'_k \rangle^{\alpha} \cdot \langle \mu_6 \rangle^{\alpha_6} \cdot \big|\overline{a}^{(k+1)}_{\mu_4, \xi_2, \ldots, \xi_k, \mu_5; \xi'_1,\ldots, \xi'_k,\mu_6}\big| \Big].
$$
\end{definition}

\begin{remark}
In other words, the entries of the first row of the matrix correspond to the derivatives on the Fourier side applied to the first, $(k+1)$-st, and $(2k+2)$-nd component of the factor of $\widehat{\gamma_0}^{(k+1)}$ which is given by: 
$$a^{(k+1)}_{\mu_1, \xi_2, \ldots, \xi_k, \mu_2; \xi'_1,\ldots, \xi'_k, \mu_3}=\widehat{\gamma_0}^{(k+1)}(\mu_1, \xi_2, \ldots, \xi_k, \mu_2; \xi'_1,\ldots, \xi'_k, \mu_3).$$ 
Hence, we apply derivatives on the Fourier side to the $\xi_1,\xi_{k+1},\xi'_{k+1}$ components of this factor. 

The second row of the matrix corresponds to the same derivatives of the factor of $\overline{\widehat{\gamma_0}}^{(k+1)}$, which is given by:

$$\overline{a}^{(k+1)}_{\mu_4, \xi_2, \ldots, \xi_k, \mu_5; \xi'_1,\ldots, \xi'_k,\mu_6}=\overline{\widehat{\gamma_0}}^{(k+1)}(\mu_4, \xi_2, \ldots, \xi_k, \mu_5; \xi'_1,\ldots, \xi'_k,\mu_6).$$
As before, we apply derivatives on the Fourier side to the $\xi_1,\xi_{k+1},\xi'_{k+1}$ components of this factor. 

In the definition, the variables which occur before the matrix denote the variables with respect to which we are summing. In this notation, we are excluding the variables $\xi_2, \ldots, \xi_k, \xi'_1, \xi_2, \ldots, \xi'_k$, since we noted above that the application of the Cauchy-Schwarz inequality in these variables is straightforward. 
\end{remark}

By using the the above definition, we can write the sum in \eqref{encodesum} as:

\begin{equation}
\label{matrix}
\xi_1,\eta_1; \begin{bmatrix}
\langle \xi_1 \rangle^{\alpha} & \langle \eta_1 \rangle^{0} & \langle \eta_1 \rangle^{0} \\
\langle \eta_1 \rangle^{0} & \langle \xi_1 \rangle^{\alpha} & \langle \eta_1 \rangle^{0} 
\end{bmatrix}
\end{equation}

We see that, in subcase IIA1), we can apply the Cauchy-Schwarz inequality in $\xi_1, \eta_1$ and obtain the wanted bound since these variables occur in the first and in the second row of the matrix. With the above shorthand notation and with these observations in mind, we continue the analysis of the possible cases.
\\
\\
\textbf{IIA2)} $\eta_2=\eta_1, \eta'_2=\xi_1$.
\\
\\
In this subcase, we know that:
\begin{equation}
\notag
\begin{cases}
\xi_1-\eta_1+\eta'_1=\xi_1\\
\xi_1-\eta_2+\eta'_2=2\xi_1-\eta_1\\
\eta'_1=\eta_1\\
\eta_2=\eta_1\\
\eta'_2=\xi_1
\end{cases}
\end{equation}
Moreover, we use the fractional Leibniz rule in the second row in order to deduce that:
$$\langle \xi_1 \rangle^{\alpha} \lesssim \langle \xi_1-\eta_2+\eta'_2 \rangle^{\alpha} \cdot \langle \eta_2 \rangle^{\alpha} \cdot \langle \eta'_2 \rangle^{\alpha} = \langle 2\xi_1-\eta_1 \rangle^{\alpha} \cdot \langle \eta_1 \rangle^{\alpha} \cdot \langle \xi_1 \rangle^{\alpha}$$
Consequently, arguing similarly as in Case IA), and taking into account the shorthand matrix notation, we need to consider:

$$\xi_1, \eta_1; \begin{bmatrix} \langle \xi_1 \rangle^{\alpha} & \langle \eta_1 \rangle^{0} & \langle \eta_1 \rangle^{0}\\
\langle 2\xi_1 - \eta_1 \rangle^{\alpha} & \langle \eta_1 \rangle^{\alpha} & \langle \xi_1 \rangle^{\alpha} 
\end{bmatrix}
$$
We can now apply the Cauchy-Schwarz inequality in $\xi_1$ and then in $\eta_1$ to obtain the wanted bound in this subcase.
\\
\\
\textbf{Case IIB)} $\eta_2=\xi_1-\eta_2+\eta'_2.$
\\
\\
We recall by \eqref{eta2eta'2} that $\eta'_2 \neq \eta_2$, so $\eta'_2 \neq \xi_1-\eta_2+\eta'_2$. 
\\
By the condition \textbf{(*)}, it follows that:
$$\eta'_2 \in \{\xi_1,\eta_1\}.$$
\\
We need to consider two possible subcases:
\\
\\
\textbf{IIB1)} $\eta'_2=\xi_1$. 
\\
\\
We obtain that: $\eta_2=\xi_1-\eta_2+\xi_1$, which implies: $\eta_2=\xi_1=\eta'_2$.
This is a contradiction, since we assumed that $\eta_2 \neq \eta'_2$.
\\
\\
\textbf{IIB2)} $\eta'_2=\eta_1$.
\\
\\
We recall that $\eta_2=\xi_1-\eta_2+\eta'_2$, hence it follows that:
$$\eta_2=\frac{\xi_1+\eta'_2}{2}=\frac{\xi_1+\eta_1}{2}.$$
In this subcase, we know that:

\begin{equation}
\notag
\begin{cases}
\xi_1-\eta_1+\eta'_1=\xi_1\\
\xi_1-\eta_2+\eta'_2=\eta_2=\frac{\xi_1+\eta_1}{2}\\
\eta'_1=\eta_1\\
\eta_2=\frac{\xi_1+\eta_1}{2}\\
\eta'_2=\eta_1
\end{cases}
\end{equation}

We use the fractional Leibniz rule in the second row in order to observe that:
$$\langle \xi_1 \rangle^{\alpha} \lesssim \langle \xi_1-\eta_2+\eta'_2 \rangle^{\alpha} \cdot \langle \eta_2 \rangle^{\alpha} \cdot \langle \eta'_2 \rangle^{\alpha} \lesssim \Big \langle \frac{\xi_1+\eta_1}{2} \Big \rangle^{\alpha} \cdot \Big \langle \frac{\xi_1+\eta_1}{2} \Big \rangle^{\alpha} \cdot \langle \eta_1 \rangle^{\alpha}$$
Hence, we need to consider:
$$\xi_1, \eta_1; \begin{bmatrix} \langle \xi_1 \rangle^{\alpha} & \langle \eta_1 \rangle^{0} & \langle \eta_1 \rangle^{0}\\
\Big \langle \frac{\xi_1+\eta_1}{2} \Big \rangle^{\alpha} & \Big \langle \frac{\xi_1+\eta_1}{2} \Big \rangle^{\alpha} & \langle \eta_1 \rangle^{\alpha} 
\end{bmatrix}$$
We now obtain the wanted bound by first applying the Cauchy-Schwarz inequality, first in $\xi_1$, and then in $\eta_1$.
\\
\\
\textbf{Case IIC)} $\eta'_2=\xi_1-\eta_2+\eta'_2$.
\\
\\
In this case, we obtain that $\xi_1=\eta_2$. 
More precisely, we know that:
\begin{equation}
\notag
\begin{cases}
\xi_1-\eta_1+\eta'_1=\xi_1\\
\xi_1-\eta_2+\eta'_2=\eta'_2\\
\eta'_1=\eta_1\\
\eta_2=\xi_1
\end{cases}
\end{equation}
We need to estimate:
$$\xi_1,\eta_1,\eta'_2; \begin{bmatrix} \langle \xi_1 \rangle^{\alpha} & \langle \eta_1 \rangle^{0} & \langle \eta_1 \rangle^{0}\\ 
\langle \eta'_2 \rangle^{0} & \langle \xi_1 \rangle^{\alpha} & \langle \eta'_2 \rangle^{0} \end{bmatrix}$$ 
In this case, we cannot directly apply the Cauchy-Schwarz in the variables $\eta_1,\eta'_2$ since the pairs of $\langle \eta_1 \rangle^0$ and $\langle \eta'_2 \rangle^0$ occur in the same rows of the matrix. Instead, we need to argue as in subcase IA3). We multiply and divide the $(1,2)$ and $(1,3)$ entry of the matrix with $\langle \eta_1 \rangle^{\alpha}$. Furthermore, we multiply and divide the $(2,1)$ and $(2,3)$ entry of the matrix by $\langle \eta'_2 \rangle^{\alpha}$ to deduce that we have to consider:

\begin{equation}
\label{newmatrix}
\xi_1,\eta_1,\eta'_2; \langle \eta_1 \rangle^{-2\alpha} \cdot \langle \eta'_2 \rangle^{-2\alpha} \cdot \begin{bmatrix} \langle \xi_1 \rangle^{\alpha} & \langle \eta_1 \rangle^{\alpha} & \langle \eta_1 \rangle^{\alpha}\\ 
\langle \eta'_2 \rangle^{\alpha} & \langle \xi_1 \rangle^{\alpha} & \langle \eta'_2 \rangle^{\alpha} \end{bmatrix}
\end{equation} 
In this notation, we mean that we multiply the result obtained from the matrix by $\langle \eta_1 \rangle^{-2\alpha} \cdot \langle \eta_2 \rangle^{-2\alpha}$ and we then sum in $\xi_1, \eta_1,\eta'_2$. (Let us emphasize that we do not multiply each row of the matrix by $\langle \eta_1 \rangle^{-2\alpha} \cdot \langle \eta_2 \rangle^{-2\alpha}.$)
We now obtain the wanted bound in this subcase after we apply the Cauchy-Schwarz inequality in $\xi_1,\eta_1,\eta'_2$. In order to bound the sum in $\eta_1, \eta'_2$, we need to use the assumption that $\alpha>\frac{3}{4}$ as in subcase IA3).
\\
\\
The bound in the big case when $\eta_1=\eta'_1$ now follows. The big case when $\eta_2=\eta'_2$ is analogous.
\\
\\
\textbf{Big Case III)} $\mathbf{\eta_1=\eta_2}$.
\\
\\
By the previous analysis, we may assume:
$$\xi_1-\eta_1+\eta'_1 \neq \xi_1-\eta_2+\eta'_2, \, \eta_1 \neq \eta'_1,\,\eta_2 \neq \eta'_2.$$
Since $\eta_1=\eta_2$, we note that $\xi_1-\eta_1+\eta'_1 \neq \xi_1-\eta_1+\eta'_2$, from where we deduce that $\eta'_1 \neq \eta'_2$.
Consequently, condition \textbf{(*)} implies that:
$$\eta'_1 \in \{\xi_1-\eta_1+\eta'_1,\xi_1-\eta_2+\eta'_2\}=\{\xi_1-\eta_1+\eta'_1,\xi_1-\eta_1+\eta'_2\}.$$
We consider the following cases:
\\
\\
\textbf{Case IIIA)} $\xi_1-\eta_1+\eta'_1=\eta'_1$.
\\
\\
In this case, we obtain that $\xi_1=\eta_1$. Hence:
$$\xi_1=\eta_1=\eta_2$$
and
$$\xi_1-\eta_2+\eta'_2=\xi_1-\xi_1+\eta'_2=\eta'_2.$$
In the matrix notation, we need to estimate:
$$\xi_1,\eta'_1,\eta'_2; \begin{bmatrix} \langle \eta'_1 \rangle^{0} & \langle \xi_1 \rangle^{\alpha} & \langle \eta'_1 \rangle^{0} \\
\langle \eta'_2 \rangle^{0} & \langle \xi_1 \rangle^{\alpha} & \langle \eta'_2 \rangle^{0} \end{bmatrix}
$$
Similarly as in Case IIC), we cannot estimate this expression directly by using the Cauchy-Schwarz inequality, but we need to write it as:
\begin{equation}
\label{newmatrixapplication}
\xi_1,\eta'_1,\eta'_2; \langle \eta'_1 \rangle^{-2\alpha} \cdot \langle \eta'_2 \rangle^{-2\alpha} \cdot \begin{bmatrix} \langle \eta'_1 \rangle^{\alpha} & \langle \xi_1 \rangle^{\alpha} & \langle \eta'_1 \rangle^{\alpha}\\
\langle \eta'_2 \rangle^{\alpha} & \langle \xi_1 \rangle^{\alpha} & \langle \eta'_2 \rangle^{\alpha} \end{bmatrix}
\end{equation}
which we can estimate by using the Cauchy-Schwarz inequality in $\xi_1,\eta'_1,\eta'_2$ since $\alpha>\frac{3}{4}$. We note that in \eqref{newmatrixapplication} we are using the analogous notational convention as was defined in \eqref{newmatrix}. 
\\
\\
\textbf{Case IIIB)} $\xi_1-\eta_2+\eta'_2=\eta'_1$.
\\
\\
In this case, we obtain:
$$\eta'_2=\eta_2+\eta'_1-\xi_1=\eta_1+\eta'_1-\xi_1.$$
We use the fractional Leibniz rule as:
$$\langle \xi_1 \rangle^{\alpha} \lesssim \langle \xi_1-\eta_1+\eta'_1 \rangle^{\alpha} \cdot \langle \eta_1 \rangle^{\alpha} \cdot \langle \eta'_1 \rangle^{\alpha}$$
in the first row and as:
$$\langle \xi_1 \rangle^{\alpha} \lesssim \langle \xi_1-\eta_2+\eta'_2 \rangle^{\alpha} \cdot \langle \eta_2 \rangle^{\alpha} \cdot \langle \eta'_2 \rangle^{\alpha}=\langle \eta'_1 \rangle^{\alpha} \cdot \langle \eta_1 \rangle^{\alpha} \cdot \langle \eta_1+\eta'_1-\xi_1 \rangle^{\alpha}$$
in the second row.
Consequently, we need to estimate:
$$\xi_1,\eta_1,\eta'_1; \begin{bmatrix} \langle \xi_1-\eta_1+\eta'_1 \rangle^{\alpha} & \langle \eta_1 \rangle^{\alpha} & \langle \eta'_1 \rangle^{\alpha} \\
\langle \eta'_1 \rangle^{\alpha} & \langle \eta_1 \rangle^{\alpha} & \langle \eta_1+\eta'_1-\xi_1 \rangle^{\alpha} \end{bmatrix}$$
We estimate this expression by first applying the Cauchy-Schwarz inequality in $\xi_1$ and then applying the Cauchy-Schwarz inequality in $\eta_1,\eta'_1$ to obtain the wanted bound.
\\
\\
\textbf{Big Case IV)} $\mathbf{\eta'_1=\eta'_2}$.
\\
\\
From the previous analysis, we may assume that:
$$\xi_1-\eta_1+\eta'_1 \neq \xi_1-\eta_2+\eta'_2,\,\eta_1 \neq \eta_2, \, \eta_1 \neq \eta'_1, \, \eta_2 \neq \eta'_2.$$
Hence, from condition \textbf{(*)}, we obtain:
$$\{\eta_1, \eta_2\} = \{\xi_1-\eta_1+\eta'_1, \xi_1-\eta_2+\eta'_2\}.$$
We need to consider two cases:
\\
\\
\textbf{Case IVA)} $\xi_1-\eta_1+\eta'_1=\eta_1$.
\\
\\
In this case, we obtain that:
$$\eta_1=\frac{\xi_1+\eta'_1}{2}.$$
We also know that $\eta_2=\xi_1-\eta_2+\eta'_2$ and $\eta'_1=\eta'_2$ and hence:
$$\eta_2=\frac{\xi_1+\eta'_2}{2}=\frac{\xi_1+\eta'_1}{2}.$$
It follows that:
$$\eta_1=\eta_2=\frac{\xi_1+\eta'_1}{2}.$$
This gives us a contradiction since we assumed that $\eta_1 \neq \eta_2$.
\\
\\
\textbf{Case IVB)} $\xi_1-\eta_2+\eta'_2=\eta_1$.
\\
\\
Since $\eta'_1=\eta'_2$, it follows that $\xi_1-\eta_2+\eta'_1=\xi_1-\eta_2+\eta'_2=\eta_1$.
Hence:
$$\eta_2=\xi_1-\eta_1+\eta'_1.$$
We use the fractional Leibniz rule as:
$$\langle \xi_1 \rangle^{\alpha} \lesssim \langle \xi_1-\eta_1+\eta'_1 \rangle^{\alpha} \cdot \langle \eta_1 \rangle^{\alpha} \cdot \langle \eta'_1 \rangle^{\alpha}$$
in the first row and as:
$$\langle \xi_1 \rangle^{\alpha} \lesssim \langle \xi_1-\eta_2+\eta'_2 \rangle^{\alpha} \cdot \langle \eta_2 \rangle^{\alpha} \cdot \langle \eta'_2 \rangle^{\alpha}=
\langle \eta_1 \rangle^{\alpha} \cdot \langle \xi_1-\eta_1+\eta'_1 \rangle^{\alpha} \cdot \langle \eta'_1 \rangle^{\alpha}$$
in the second row.
Consequently, we need to estimate:
$$\xi_1, \eta_1, \eta'_1; \begin{bmatrix} \langle \xi_1-\eta_1+\eta'_1 \rangle^{\alpha} & \langle \eta_1 \rangle^{\alpha} & \langle \eta'_1 \rangle^{\alpha}\\
\langle \eta_1 \rangle^{\alpha} & \langle \xi_1-\eta_1+\eta'_1 \rangle^{\alpha} & \langle \eta'_1 \rangle^{\alpha} \end{bmatrix}
$$
We estimate this expression by first applying the Cauchy-Schwarz inequality in $\xi_1$ and then in $\eta_1,\eta'_1$ and we obtain the wanted bound.
\\
\\
\textbf{Big Case V)} $\mathbf{\eta_1=\eta'_2}$.
\\
\\
By the previous analysis, we may assume that:
\begin{equation}
\label{smalltriangle}
\xi_1-\eta_1+\eta'_1 \neq \xi_1-\eta_2+\eta'_2,\,\eta_1 \neq \eta_2,\,\eta_1 \neq \eta'_1,\,\eta_2 \neq \eta'_2,\,\eta'_1 \neq \eta'_2.
\end{equation}
We need to consider several cases:
\\
\\
\textbf{Case VA)} $\eta_2=\eta'_1$.
\\
\\
In this case, we can use the fractional Leibniz rule as:
$$\langle \xi_1 \rangle^{\alpha} \lesssim \langle \xi_1-\eta_1+\eta'_1 \rangle^{\alpha} \cdot \langle \eta_1 \rangle^{\alpha} \cdot \langle \eta'_1 \rangle^{\alpha} = \langle \xi_1-\eta_1+\eta_2 \rangle^{\alpha} \cdot \langle \eta_1 \rangle^{\alpha} \cdot \langle \eta_2 \rangle^{\alpha}$$
in the first row and as:
$$\langle \xi_1 \rangle^{\alpha} \lesssim \langle \xi_1-\eta_2+\eta'_2 \rangle^{\alpha} \cdot \langle \eta_2 \rangle^{\alpha} \cdot \langle \eta'_2 \rangle^{\alpha}= \langle \xi_1-\eta_2+\eta_1 \rangle^{\alpha} \cdot \langle \eta_2 \rangle^{\alpha} \cdot \langle \eta_1 \rangle^{\alpha}$$
in the second row.
Consequently, we need to estimate:
$$\xi_1,\eta_1,\eta_2; \begin{bmatrix} \langle \xi_1-\eta_1+\eta_2 \rangle^{\alpha} & \langle \eta_1 \rangle^{\alpha} & \langle \eta_2 \rangle^{\alpha} \\
\langle \xi_1-\eta_2+\eta_1 \rangle^{\alpha} & \langle \eta_2 \rangle^{\alpha} & \langle \eta_1 \rangle^{\alpha} \end{bmatrix}$$
We can estimate this expression by first applying the Cauchy-Schwarz inequality in $\xi_1$ and then in $\eta_1,\eta_2$ to obtain the wanted bound.
\\
The next case we need to consider is:
\\
\\
\textbf{Case VB)} $\eta_2 \neq \eta'_1$.
\\
\\
From condition \textbf{(*)}, the fact that $\eta'_2=\eta_1$, and from \eqref{smalltriangle}, it follows that in this case, one has:
$$\{\eta'_1,\eta_2\}=\{\xi_1-\eta_1+\eta'_1,\xi_1-\eta_2+\eta'_2\}=\{\xi_1-\eta_1+\eta'_1,\xi_1-\eta_2+\eta_1\}.$$
We need to consider two subcases:
\\
\\
\textbf{VB1)} $\eta'_1=\xi_1-\eta_1+\eta'_1,\,\eta_2=\xi_1-\eta_2+\eta_1$.
\\
\\
From the first equality, it follows that $\xi_1=\eta_1$.
From the second equality, it follows that $\eta_2=\frac{\xi_1+\eta_1}{2}$.
Consequently:
$$\eta_2=\xi_1=\eta_1.$$
This gives us a contradiction since we are assuming that $\eta_1 \neq \eta_2$.
\\
\\
\textbf{VB2)} $\eta'_1=\xi_1-\eta_2+\eta_1,\,\eta_2=\xi_1-\eta_1+\eta'_1$.
\\
\\
We combine the two equalities to deduce that:
$$\eta'_1=\xi_1-(\xi_1-\eta_1+\eta'_1)+\eta_1=2\eta_1-\eta'_1$$
from where it follows that:
$$\eta'_1=\eta_1.$$
This gives us a contradiction since we are assuming that $\eta'_1 \neq \eta_1$.
\\
\\
The big case $\eta'_1=\eta_2$ is analogous.
\\
\\
\textbf{Big Case VI)} $\mathbf{\xi_1-\eta_1+\eta'_1 \neq \xi_1-\eta_2+\eta'_2,\,\eta_1 \neq \eta'_1,\,\eta_2 \neq \eta'_2,\,\eta_1 \neq \eta_2,\,\eta'_1 \neq \eta'_2,\,\eta_1 \neq \eta'_2,\,\eta'_1 \neq \eta_2}$.
\\
\\
This big case cannot satisfy the condition \textbf{(*)}. Namely, $\xi_1-\eta_1+\eta'_1$ and $\xi_1-\eta_2+\eta'_2$ can pair up with at most two distinct elements of $\{\eta_1,\eta'_1,\eta_2,\eta'_2\}$.
\\
\\
The claim now follows.
\end{proof}

Let us now note:

\begin{remark}
\label{Remark2}
If we consider the domain $\Lambda=\mathbb{T}^d$, the proof of Theorem \ref{Theorem 1} carries over and the condition on $\alpha$ becomes $\alpha>\frac{d}{4}$, since for such $\alpha$, it is the case that $\frac{1}{\langle n \rangle^{2\alpha}} \in \ell^2(\mathbb{Z}^d)$. We will omit the details.
\end{remark}

\begin{remark}
\label{Remark1}
The above proof can be modified to apply in the case when the $(h_{\xi})_{\xi \in \mathbb{Z}^3}$ are independent, identically distributed Gaussian random variables centered at zero. Since then it is no longer the case that $h^2=1$, it is better to define the randomized collision operator in a different way. The difference would be not to include the factor of $h_{\xi_1}(\omega)$ in \eqref{Bjkomega} and \eqref{Bjkomega2}. The proof of the estimate is analogous now and we use the fact the quantity $\int_{\Omega} |h_{\xi}(\omega)|^6 \,dp(\omega)$ is uniformly bounded in $\xi$. The latter condition is used when we apply H\"{o}lder's inequality in $\omega$. In the Gaussian case, it is necessary to apply H\"{o}lder's inequality in $\omega$ since it is no longer the case that $|h_{\xi}(\omega)|$ is uniformly bounded in $\xi$ and $\omega$. In this way, we can reduce the estimate to bounding a sum in frequencies satisfying a constraint analogous to $(\textbf{*\,})$, as was the case for Bernoulli random variables. We will omit the details.


We have chosen to use the Bernoulli randomization in order relate the randomization to the nonlinear Schr\"{o}dinger equation. The link with the nonlinear Schr\"{o}dinger equation is explained in more detail in Subsection \ref{Link with NLS}. In this context, we see that it is essential to take Bernoulli random variables in the randomization. Moreover, the explicit computation of the Duhamel iterates in Subsection \ref{The precise form of the Duhamel expansion term} is based on the fact that we are taking a Bernoulli randomization.
\end{remark}

\subsection{An alternative form of randomization}
\label{Alternative form of randomization}
It is possible to randomize in the initial data $\gamma^{(k+1)}_0$ instead of in the collision operator $B_{j,k+1}$ as we did in \eqref{Bjkomega}. The former approach was shown to be useful in the study of nonlinear dispersive equations in the references noted in the introduction. In the context of the GP hierarchy, this approach still applies, but it leads to additional difficulties and it does not seem to be well suited to the problem.

Given $\gamma^{(k+1)}_0$ as in \eqref{Fourierbasis}, one defines a \emph{randomization map}:
\begin{equation}
\label{randomizeRemark2}
\gamma_0^{(k+1)} \mapsto \gamma_{0,\omega}^{(k+1)}.
\end{equation}
This map takes a deterministic object and gives us a random object. $\gamma^{(k+1)}_0$ is defined by:
$$\gamma_{0,\omega}^{(k+1)}=\sum_{\ell_r, j_r} h^{(k+1)}_{\ell_1,\ldots,\ell_{k+1}; j_1, \ldots, j_{k+1}}(\omega) \cdot a^{(k+1)}_{\ell_1,\ldots,\ell_{k+1}; j_1, \ldots, j_{k+1}} \cdot b^{(k+1)}_{\ell_1,\ldots,\ell_{k+1}; j_1, \ldots, j_{k+1}}$$
for $h^{(k+1)}_{\ell_1,\ldots,\ell_{k+1}; j_1, \ldots, j_{k+1}}(\omega):=\prod_{r=1}^{k+1}h_{\ell_r}(\omega) \cdot \prod_{r=1}^{k+1} \overline{h_{j_r}(\omega)}.$
Here, we assume that $(h_j)$ is a sequence of \emph{independent, identically distributed Bernoulli random variables centered at zero}. Since the random variables are then all real, we can write $h^{(k+1)}_{\ell_1,\ldots,\ell_{k+1}; j_1, \ldots, j_{k+1}}(\omega):=\prod_{r=1}^{k+1}h_{\ell_r}(\omega) \cdot \prod_{r=1}^{k+1} h_{j_r}(\omega).$

Having defined randomization this way, it is possible to prove:

\begin{theorem}
\label{Proposition 2}  For the randomization defined in \eqref{randomizeRemark2}, the following bound holds:
\begin{equation}
\notag
\big\|S^{(k,\alpha)}B_{j,k+1} \gamma_{0,\omega}^{(k+1)}\big\|_{L^2(\Omega \times \Lambda^k \times \Lambda^k)} \leq C(\alpha,k) \big\|S^{(k+1,\alpha)} \gamma_0^{(k+1)}\big\|_{L^2(\Lambda^{k+1} \times \Lambda^{k+1})},
\end{equation}
whenever $\alpha>\frac{3}{4}$. The constant $C=C(\alpha,k)>0$ depends on both $\alpha$ and $k$.
\end{theorem}
The $k$-dependence in the constant $C$ is of the order $\sim_{\alpha} \sqrt{\frac{(2k)!}{2^k k!}}$, which is $\sim \sqrt{k!}$ by Stirling's formula. The reason for this large constant is that the various frequencies can pair up in many different ways, which was not the case in the proof of Theorem \ref{Theorem 1}. More precisely, when we take squares and integrate in $\omega$, there are $2k$ frequencies which pair up in $\sim \frac{(2k)!}{2^k k!}$ different ways. A similar phenomenon is encountered in Subsection \ref{Estimate in N}, where we keep precise track of the implied constant in the estimate.

The proof of Theorem \ref{Proposition 2} is similar to the proof of Theorem \ref{Theorem 1}, but it is significantly more involved since there are many more cases. As we noted above, the reason why there are more cases is the fact that the frequencies in the analogue of condition \textbf{(*)} (in the proof of Theorem \ref{Theorem 1}) can pair up in more ways. There are approximately three times as many cases to check. In the end, the condition that one needs is again $\alpha>\frac{3}{4}$. In the general case of $\Lambda=\mathbb{T}^d$, one needs to again assume that $\alpha>\frac{d}{4}$. We will omit the details.

If, instead of Bernoulli random variables, we take standard Gaussian random variables centered at zero, then a version of Theorem \ref{Proposition 2} still holds. In this case, the constant $C$ has additional dependence on $k$, in addition to the $\sqrt{k!}$ factor which was noted earlier. This is due to the fact that we need to bound the quantity 
$\int_{\Omega} |h_{\xi}(\omega)|^{2(k+1)} \,dp(\omega)$, which is not uniformly bounded in $k$ for $(h_{\xi})$ a standard Gaussian random variable centered at zero.

	For a further discussion of this method of randomization, we refer the reader to Remark \ref{Remark3}. In particular, we note that the non-commutativity relation given by \eqref{noncommutativity} in the mentioned remark implies that the boardgame argument is not applicable in the model which one obtains by using this method of randomization. The precise details are given in the remark.

\section{Properties of the randomized Gross-Pitaevskii hierarchy}
\label{Properties of the randomized Gross-Pitaevskii hierarchy}

In this section, we study the properties of the \emph{randomized Gross-Pitaevskii hierarchy}:
\begin{equation}
\label{randomizedGP}
\begin{cases}
i \partial_t \gamma^{(k)} + (\Delta_{\vec{x}_k}-\Delta_{\vec{x'}_k})\gamma^{(k)}=\sum_{j=1}^{k} [B_{j,k+1}]^{\omega} (\gamma^{(k+1)})\\
\gamma^{(k)}\big|_{t=0}=\gamma_0^{(k)}.
\end{cases}
\end{equation}
We see that, in this hierarchy, there is only one random parameter $\omega \in \Omega$. In other words, we see that all of the collision operators are randomized according to the same  parameter, according to which all of the collisions are randomized. As a result, we also sometimes call \eqref{randomizedGP} the \emph{dependently randomized Gross-Pitaevskii hierarchy} in the discussion that follows.
We note below that this hierarchy has a lot of properties which carry over from the deterministic GP hierarchy.

\subsection{Link with the nonlinear Schr\"{o}dinger equation}
\label{Link with NLS}
In this subsection, we construct \emph{factorized solutions} to \eqref{randomizedGP}. As we will see, they will be obtained as tensor products of solutions to a \emph{nonlinear Schr\"{o}dinger equation} with a \emph{random nonlinearity}. In this way, we obtain a probabilistic analogue of the discussion in Subsection \ref{Factorized solutions}. As a result, we can better understand the connection between the deterministic and the randomized GP hierarchy.

In order to avoid notational difficulties of dependence on the randomization parameter $\omega$, we will give a more precise notation for the randomization map given in \eqref{randomization}.
Namely, given $\omega \in \Omega$, we define the map $T^{\omega}:L^2(\Lambda) \rightarrow L^2(\Lambda)$ as follows: 

\begin{equation}
\label{TOmega}
(T^{\omega}f)\,\widehat{\,}\,(\xi):=h_{\xi}(\omega) \cdot \widehat{f}(\xi),\,\xi \in \mathbb{Z}^3.
\end{equation}
We note that $T^{\omega}$ is a linear operator. It corresponds to the map defined in \eqref{randomization}.
Since we are using standard Bernoulli random variables, it follows that $h_{\xi}^2=1$. Hence, we obtain that $T^{\omega}$ is an involution:
\begin{equation}
\label{T2}
T^{\omega} \circ T^{\omega}=Id.
\end{equation}
Suppose that $\phi=\phi(x,t)$ solves the Cauchy problem associated to the nonlinear Schr\"{o}dinger equation:
\begin{equation}
\label{NLS}
\begin{cases}
i \partial_t \phi + \Delta \phi = |\phi|^2 \phi\\
\phi\big|_{t=0}=\phi_0.
\end{cases}
\end{equation}
Applying the operator $T^{\omega}$ to the above equation, it follows that:
$$i \partial_t \,\Big(T^{\omega} \phi \Big) + \Delta \,\Big(T^{\omega} \phi \Big) = T^{\omega} \,\Big(|\phi|^2 \phi \Big)$$
By \eqref{T2}, it follows that:
$$i \partial_t \,\Big(T^{\omega} \phi \Big) + \Delta \,\Big(T^{\omega} \phi \Big) = T^{\omega} \,\Big(|T^{\omega}(T^{\omega}\phi)|^2 \cdot T^{\omega}(T^{\omega}\phi) \Big)$$
Consequently, the function $\psi=\psi(\omega):=T^{\omega} \phi$ solves the equation:
\begin{equation}
\label{NLS_random}
\begin{cases}
i \partial_t \psi + \Delta \psi = T^{\omega} \Big(|T^{\omega} \psi|^2 \cdot T^{\omega} \psi \Big)\\
\psi|_{t=0}=T^{\omega}\phi_0.
\end{cases}
\end{equation}
In other words, the function $\psi$, which is a randomization of the solution of the deterministic NLS \eqref{NLS}, solves a randomized NLS-type equation.
We can now use the function $\psi^{\omega}$ to find a \emph{factorized solution} to the 
randomized Gross-Pitaevskii hierarchy \eqref{randomizedGP}.
In particular, if we take initial data $\gamma^{(k)}_0:=|T^{\omega} \phi_0 \rangle \langle T^{\omega} \phi_0|^{\otimes k}$, we note that a solution to \eqref{randomizedGP} is given by $\gamma^{(k)}:= |T^{\omega} \phi \rangle \langle T^{\omega} \phi|^{\otimes k}$. This follows from the definition of the operator $T^{\omega}$ and from the definition of the Fourier transform of $[B_{j,k+1}]^{\omega} \gamma^{(k+1)}_0$ given in \eqref{Bjkrandomized}. 
\begin{remark}
\label{OmegaBoardgame}
It can be shown that the boardgame argument in general applies in the context of \eqref{randomizedGP}, as long as one works in a class of permutation symmetric density matrices. This gives another connection of the randomized GP hierarchy with the nonlinear Schr\"{o}dinger equation. However, in our further analysis, we will study density matrices in the class of $\mathcal{N}$ given in Definition \ref{Class_N} below, in which the objects no longer have permutation symmetry. This symmetry is crucial in applications the boardgame argument, so the combinatorial reduction does not apply in the class in which we will be working. It is important to note that this is not a structural feature of the hierarchy itself. We will henceforth not pursue the issue of applying the boardgame argument in the context of \eqref{randomizedGP}.
\end{remark}

\section{A new randomized hierarchy}
\label{A new randomized hierarchy}
As was noted in the previous section, the randomized Gross-Pitaevskii hierarchy \eqref{randomizedGP} shares a lot of properties with the deterministic Gross-Pitaevskii hierarchy. However, due to the dependent randomization, it is not possible to directly apply the randomized spacetime estimate from Theorem \ref{Theorem 1}. In this section, we circumvent this difficulty by modifying the problem. In particular, we consider the following
\emph{independently randomized Gross-Pitaevskii hierarchy}:
\begin{equation}
\label{randomizedGP2}
\begin{cases}
i \partial_t \gamma^{(k)} + (\Delta_{\vec{x}_k}-\Delta_{\vec{x'}_k})\gamma^{(k)}=\sum_{j=1}^{k} [B_{j,k+1}]^{\omega_{k+1}} (\gamma^{(k+1)})\\
\gamma^{(k)}|_{t=0}=\gamma_0^{(k)}.
\end{cases}
\end{equation}
The difference between the randomized GP hierarchy \eqref{randomizedGP} and the hierarchy \eqref{randomizedGP2} is the fact that, in the latter case, the randomization parameters $\omega_k \in \Omega$ are not assumed to be equal. We can view this as a way of independently randomizing the full collision operator $B^{(k+1)}=\sum_{j=1}^{k} B_{j,k+1}$ according to \eqref{Bjkomega} and \eqref{Bjkomega2}.

More precisely, for each $k$, we perform a sequence of independent Bernoulli trials. In other words, for every frequency $\xi \in \mathbb{Z}^3$, we choose $h_{\xi}(\omega_{k+1})$, which is $1$ or $-1$, each with probability $\frac{1}{2}$. We then define $[B_{j,k+1}]^{\omega_{k+1}}$ according to \eqref{Bjkomega} and \eqref{Bjkomega2}, where each $h_{\xi}(\omega)$ is replaced by $h_{\xi}(\omega_{k+1})$. We apply the above procedure at each level of the hierarchy. This is in contrast to the (dependently) randomized GP hierarchy \eqref{randomizedGP}, where we perform a sequence of independent Bernoulli trials only once to determine $h_{\xi}(\omega)$ for $\xi \in \mathbb{Z}^3$, and we use the obtained random coefficients at each level of the hierarchy.

In the context of \eqref{randomizedGP2}, it is possible to apply the averaged spacetime estimate from Theorem \ref{Theorem 1} to estimate the Duhamel expansion of the solution evolving from zero initial data. We have to estimate these terms in a norm which involves an average in all of the randomization parameters which occur in the expansion. The precise estimate is given in Theorem \ref{Smallness Bound}.

\subsection{Properties of the hierarchy \eqref{randomizedGP2}}
\label{Properties of randomizedGP2}
Before we state and give an application of Theorem \ref{Theorem 1} in Subsection \ref{Application of randomized estimate}, let us briefly summarize some of the relevant properties of the hierarchy \eqref{randomizedGP2}. Unlike the (dependently) randomized GP hierarchy, the independently randomized GP hierarchy has some properties which are different from the deterministic case. We observe that the following holds:

\subsubsection{No obvious factorized solutions}
The hierarchy \eqref{randomizedGP2} no longer admits obvious factorized solutions. We cannot in general find solutions $\gamma^{(k)}:=|T^{\omega_k} \phi \rangle \langle T^{\omega_k} \phi|^{\otimes k}$, where $i\partial_t \phi + \Delta \phi = |\phi|^2 \phi$.
The relation:
$$\Big(i \partial_t + (\Delta_{\vec{x}_k}-\Delta_{\vec{x}'_k})\Big)|T^{\omega_k} \phi \rangle \langle T^{\omega_k} \phi|^{\otimes k}=\sum_{j=1}^{k}[B_{j,k+1}]^{\omega_{k+1}}
\Big(|T^{\,\omega_{k+1}}\phi \rangle \langle T^{\,\omega_{k+1}} \phi|^{\otimes (k+1)}\Big)
$$
only holds when $\omega_k=\omega_{k+1}$.

\subsubsection{The boardgame argument does not apply}
\label{No boardgame argument randomizedGP2}
Given $k,\ell,m,n \in \mathbb{N}$ mutually distinct such that $1 \leq k \leq \ell-1$, $1 \leq m \leq n-1$, and given $\omega_1,\omega_2 \in \Omega$, the following commutation relation can be shown for the randomized collision operators:
\begin{equation}
\label{commutation2}
[B_{k,\ell}]^{\omega_1} \, [B_{m,n}]^{\omega_2} = [B_{m,n}]^{\omega_2} \, [B_{k,\ell}]^{\omega_1}.
\end{equation}

Even though the commutation relation \eqref{commutation2} holds for \eqref{randomizedGP2}, the boardgame argument still does not apply. The main reason is the fact that in the derivation of the boardgame argument \cite{KM}, one can interchange the $t_k$ and $x_k$ variables inside of the integrand. However, in the context of \eqref{randomizedGP2}, it is not possible to interchange the $\omega_k$ parameters. Thus, integrands which were equal in the deterministic case (or in the case that all of the $\omega_k$ were equal), now become mutually distinct.

Let us examine this problem on a concrete example. We use the same example given in the explanation of the boardgame argument in equations (26) and (27) of \cite{KM}. In the following discussion, we write the operator $\mathcal{U}^{(k)}(t)$ explicitly as $e^{\,it\Delta_{\pm}^{(k)}}$, with notation as in \cite{KM}.
Let us start with the time-dependent density matrix $\gamma^{(5)}=\gamma^{(5)}(t,\vec{x}_5;\vec{x}'_5)$ which satisfies the permutation symmetry:
\begin{equation}
\label{symmetry}
\gamma^{(5)}(t,x_{\sigma(1)},\ldots,x_{\sigma(5)};x_{\sigma(1)}',\ldots,x_{\sigma(5)}')=\gamma^{(5)}(t,x_1,\ldots,x_5;x_1',\ldots,x_5')
\end{equation}
for all permutations $\sigma \in S^5$. In the discussion that follows, we write $\gamma^{(5)}(t,\vec{x}_5;\vec{x}'_5)$ just as $\gamma^{(5)}(t)$. We write $\Delta_{\pm,x_j}$ for the Laplace operator only in the $x_j$ variable. We also use the following shorthand $\Delta_{\pm}^{(k)}:=\Delta_{\vec{x}_k}-\Delta_{\vec{x}'_k}$.

With this notation in mind, we consider:

\begin{equation}
\label{I1}
I_1:=\int_{t_1 \geq t_2 \geq t_5 \geq t_4 \geq t_3} e^{i(t_1-t_2)\Delta_{\pm}^{(1)}} \, [B_{1,2}]^{\omega_2} \, e^{i(t_2-t_3)\Delta_{\pm}^{(2)}} \, [B_{2,3}]^{\omega_3} \, e^{i(t_3-t_4) \Delta_{\pm}^{(3)}} 
\end{equation}
$$[B_{1,4}]^{\omega_4} \, e^{i(t_4-t_5)\Delta_{\pm}^{(4)}} \, [B_{4,5}]^{\omega_5} \gamma^{(5)}(t_5)\,dt_2\,dt_3\,dt_4\,dt_5.
$$
and
\begin{equation}
\label{I2}
I_2:=\int_{t_1 \geq t_2 \geq t_5 \geq t_3 \geq t_4} e^{i(t_1-t_2)\Delta_{\pm}^{(1)}} \, [B_{1,2}]^{\omega_2} \, e^{i(t_2-t_3)\Delta_{\pm}^{(2)}} \, [B_{1,3}]^{\omega_3} \, e^{i(t_3-t_4) \Delta_{\pm}^{(3)}} 
\end{equation}
$$
 [B_{3,4}]^{\omega_4} \, e^{i(t_4-t_5)\Delta_{\pm}^{(4)}} [B_{3,5}]^{\omega_5} \,\gamma^{(5)}(t_5) \,dt_2\,dt_3\,dt_4\,dt_5.$$
We note that $I_1$ and $I_2$ are functions of $t_1,x_1,x'_1$. The collision operators in the integrand act by integrating out the dependence on $x_2,x'_2,x_3,x'_3,x_4,x'_4,x_5,x'_5$.

In the deterministic setting \cite{KM}, it was shown that the analogues of the expressions $I_1$ and $I_2$ are mutually equal. As was noted above, the same is true in the setting of the (dependently) randomized GP hierarchy. We now show that, in the context of the independently randomized GP hierarchy, it is no longer the case that $I_1$ equals to $I_2$.

As in \cite{KM}, an identity which follows from the commutation \eqref{commutation2} is:
\begin{equation}
\label{commutation3}
e^{i(t_2-t_3)\Delta_{\pm}^{(2)}}\,[B_{2,3}]^{\omega_3}\,e^{i(t_3-t_4)\Delta_{\pm}^{(3)}}\,
[B_{1,4}]^{\omega_4}\,e^{i(t_4-t_5)\Delta_{\pm}^{(4)}}=
\end{equation}
$$e^{i(t_2-t_4)\Delta_{\pm}^{(2)}} \, [B_{1,4}]^{\omega_4}\,e^{-i(t_3-t_4)(\Delta_{\pm}^{(3)}-\Delta_{\pm,x_3}+\Delta_{\pm,x_4})} \, [B_{2,3}]^{\omega_3} \, e^{i(t_3-t_4)\Delta_{\pm}^{(4)}}.$$
In the following discussion, let $k_{1,2}^{\omega_2}$ formally denote the kernel of $[B_{1,2}]^{\omega_2}$. We use similar notation for the other collision operators.
Let:
$$\gamma_{3,4}:=\gamma^{(5)}(t_5,x_1,x_2,x_3,x_4,x_5;x'_1,x'_2,x'_3,x'_4,x'_5)$$
and
$$\gamma_{4,3}:=\gamma^{(5)}(t_5,x_1,x_2,x_4,x_3,x_5;x'_1,x'_2,x'_4,x'_3,x'_5).$$
The symmetry assumption on $\gamma^{(5)}$ given by \eqref{symmetry} guarantees that $\gamma_{3,4}=\gamma_{4,3}$.
\\
\\
If we use \eqref{commutation3} and if we write out $I_1$ in terms of the integral kernels, it follows that:

$$I_1=\int_{t_1 \geq t_2 \geq t_5 \geq t_4 \geq t_3} \int_{\Lambda^4 \times \Lambda^4} e^{i(t_1-t_2)\Delta_{\pm}^{(1)}} \, k_{1,2}^{\omega_2} \, e^{i(t_2-t_4)\Delta_{\pm}^{(2)}} \, k_{1,4}^{\omega_4}\,e^{-i(t_3-t_4)(\Delta_{\pm}^{(3)}-\Delta_{\pm,x_3}+\Delta_{\pm,x_4})} $$
$$ k_{2,3}^{\omega_3} \, e^{i(t_3-t_5)\Delta_{\pm}^{(4)}} \, k_{4,5}^{\omega_5} \, \gamma_{4,3} \, dx_2 \ldots dx_5\, dx'_2 \ldots dx'_5 \, dt_2 \, dt_3 \, dt_4 \, dt_5 .$$
We exchange $(t_3,x_3,x'_3)$ and $(t_4,x_4,x'_4)$ to deduce that this expression equals:
$$
\int_{t_1 \geq t_2 \geq t_5 \geq t_3 \geq t_4} \int_{\Lambda^4 \times \Lambda^4} e^{i(t_1-t_2)\Delta_{\pm}^{(1)}} \, k_{1,2}^{\omega_2} \, e^{i(t_2-t_3)\Delta_{\pm}^{(2)}} \, k_{1,3}^{\omega_4}\,e^{i(t_3-t_4)\Delta_{\pm}^{(3)}}$$
$$ k_{2,4}^{\omega_3} \, e^{i(t_4-t_5)\Delta_{\pm}^{(4)}} \, k_{3,5}^{\omega_5} \, \gamma_{3,4} \, dx_2 \ldots dx_5\, dx'_2 \ldots dx'_5  \, dt_2 \, dt_3 \, dt_4 \, dt_5$$
which since $\gamma_{3,4}=\gamma_{4,3}$ equals:
$$
\int_{t_1 \geq t_2 \geq t_5 \geq t_3 \geq t_4} \int_{\Lambda^4 \times \Lambda^4} e^{i(t_1-t_2)\Delta_{\pm}^{(1)}} \, k_{1,2}^{\omega_2} \, e^{i(t_2-t_3)\Delta_{\pm}^{(2)}} \, k_{1,3}^{\omega_4}\,e^{i(t_3-t_4)\Delta_{\pm}^{(3)}}$$
$$ k_{2,4}^{\omega_3} \, e^{i(t_4-t_5)\Delta_{\pm}^{(4)}} \, k_{3,5}^{\omega_5} \, \gamma_{4,3}  \, dx_2 \ldots dx_5\, dx'_2 \ldots dx'_5 \, dt_2 \, dt_3 \, dt_4 \, dt_5$$
By definition, this is:
$$
\int_{t_1 \geq t_2 \geq t_5 \geq t_3 \geq t_4}  e^{i(t_1-t_2)\Delta_{\pm}^{(1)}} \,[B_{1,2}]^{\omega_2} \, e^{i(t_2-t_3)\Delta_{\pm}^{(2)}} \, [B_{1,3}]^{\omega_4} \, e^{i(t_3-t_4)\Delta_{\pm}^{(3)}}$$
$$[B_{2,3}]^{\omega_3} \, e^{i(t_4-t_5)\Delta_{\pm}^{(4)}} \, [B_{3,5}]^{\omega_5} \gamma^{(5)}(t_5) \, dt_2 \, dt_3 \, dt_4 \, dt_5$$
which does not equal to:
$$I_2= \int_{t_1 \geq t_2 \geq t_5 \geq t_3 \geq t_4} e^{i(t_1-t_2)\Delta_{\pm}^{(1)}} \, [B_{1,2}]^{\omega_2} \, e^{i(t_2-t_3)\Delta_{\pm}^{(2)}} \, [B_{1,3}]^{\omega_3} \, e^{i(t_3-t_4) \Delta_{\pm}^{(3)}}$$
$$[B_{3,4}]^{\omega_4} \, e^{i(t_4-t_5)\Delta_{\pm}^{(4)}} [B_{3,5}]^{\omega_5} \,\gamma^{(5)}(t_5) \,dt_2\,dt_3\,dt_4\,dt_5$$
unless $\omega_3=\omega_4$.

\begin{remark}
\label{Remark3}
If we were to use the randomization mentioned in Subsection \ref{Alternative form of randomization}, we would be led to the study of the following hierarchy:
\begin{equation}
\label{randomizedGP1}
\begin{cases}
i \partial_t \gamma^{(k)}+(\Delta_{x_k}-\Delta_{x'_k}) \gamma^{(k)}=\sum_{j=1}^{k}B^{\omega}_{j,k+1}(\gamma^{(k+1)})\\
\gamma^{(k)}\big|_{t=0}=\gamma^{(0)}_k.
\end{cases}
\end{equation}
for $\omega \in \Omega$.
Let us explain the notation.
The operator $B^{\omega}_{j,k+1}$ is defined as first applying the operator $B_{j,k+1}$ and then applying the randomization given by \eqref{randomizeRemark2} in Remark \ref{Remark2} to the result. In other words, for a fixed density matrix $\sigma_0^{(k+1)}$ :
$$\sigma_0^{(k+1)} \mapsto B_{j,k+1}(\sigma_0^{(k+1)}) \mapsto B^{\omega}_{j,k+1}(\sigma_0^{(k+1)})=\,\mbox{randomization of}\,B_{j,k+1}(\sigma_0^{(k+1)}).$$ 
It can be shown that, for $k,\ell,m,n \in \mathbb{N}$ mutually distinct such that $1 \leq k \leq \ell-1$, $1 \leq m \leq n-1$, the analogue of the commutation relation \eqref{commutation2} does not hold, i.e.
\begin{equation}
\label{noncommutativity}
B_{k,\ell}^{\omega} \, B_{m,n}^{\omega} \neq B_{m,n}^{\omega} \, B_{k, \ell}^{\omega}.
\end{equation}
Hence, the boardgame argument does not apply in the study of \eqref{randomizedGP1}.
\end{remark}
\subsection{An application of the randomized spatial estimate}
\label{Application of randomized estimate}

In this subsection, we apply the estimate given by Theorem \ref{Theorem 1} to the \emph{independently randomized Gross-Pitaevskii hierarchy} \eqref{randomizedGP2}. The main result is given in Theorem \ref{Smallness Bound} below.  Throughout this subsection we assume that $\alpha>\frac{3}{4}$.

Let us fix a \textbf{\emph{deterministic}} sequence of time-dependent density matrices $(\gamma^{(k)}(t))_k=(\gamma^{(k)}(t, \vec{x}_k; \vec{x}'_k))_k$ satisfying the following a priori bound:
\begin{equation}
\label{gammakcondition}
\big\|S^{(k,\alpha)}\gamma^{(k)}(t)\big\|_{L^2(\Lambda^k \times \Lambda^k)} \leq C_1^k
\end{equation}
uniformly in time for some $C_1>0$ independent of $k$. 
We note that, for factorized objects $\gamma^{(k)}=|\psi \rangle \langle \psi|^{\otimes k}$, the condition \eqref{gammakcondition} reduces to 
\begin{equation}
\label{psi}
\|\psi(t)\|_{H^{\alpha}} \leq \sqrt{C_1}.
\end{equation} 

In the case $\alpha=1$, the condition \eqref{psi} holds globally in time if $\psi=\phi$ for $\phi$ solving the NLS equation:

\begin{equation}
\notag
\begin{cases}
i \partial_t \phi + \Delta \phi=|\phi|^2 \phi\\
\phi|_{t=0} = \phi_0 \in H^{\alpha}.
\end{cases}
\end{equation}
Since $T^{\omega}$ is an isometry on $H^{\alpha}$, it follows that the condition \eqref{psi} also holds for $\psi=T^{\omega} \phi$, whenever $\omega \in \Omega$.
If $\alpha>1$, then we obtain that the condition \eqref{psi} holds only locally in time for $\psi=\phi$ and $\psi=\phi^{\omega}$. The latter fact can be deduced from the arguments given in \cite{1B}. As we will see, this also will be enough for our application.

Since we will not be applying the boardgame argument, we do not need to assume that $\gamma^{(k)}$ satisfies the permutation symmetry property as in \eqref{symmetry}. In this section, we write $\omega_k \in \Omega_k$ instead of $\omega \in \Omega$ in order to emphasize that the indices are distinct. In other words, the different randomizations will be independent.
\\
\\
Given $k, n \in \mathbb{N}, t_k >0$ and $(\omega_{k+1},\omega_{k+2}, \ldots, \omega_{n+k}) \in \Omega_{k+1} \times \Omega_{k+2} \times \cdots \times \Omega_{n+k}$, we define $\sigma^{(k)}_{n;\,\omega_{k+1},\ldots,\omega_{n+k}}$ by:
\begin{equation}
\label{precisedefinition2}
\sigma^{(k)}_{n;\,\omega_{k+1},\ldots,\omega_{n+k}}(t_k):=
\end{equation}
$$(-i)^n \int_{0}^{t_k} \int_{0}^{t_{k+1}} \cdots \int_{0}^{t_{n+k-1}} \mathcal{U}^{(k)}(t_k-t_{k+1}) \, [B^{(k+1)}]^{\omega_{k+1}} \,\mathcal{U}^{(k+1)}(t_{k+1}-t_{k+2})$$
$$ [B^{(k+2)}]^{\omega_{k+2}} \cdots \,\mathcal{U}^{(n+k-1)}(t_{n+k-1}-t_{n+k}) \, [B^{(n+k)}]^{\omega_{n+k}} \, \gamma^{(n+k)}(t_{n+k}) \,dt_{n+k} \cdots dt_{k+2} \, dt_{k+1}.$$ 
Here, we recall that by \eqref{Bjkrandomized}: 
$$[B^{(\ell+1)}]^{\omega}=\sum_{j=1}^{\ell} [B^{+}_{j,\ell+1}]^{\omega}-\sum_{j=1}^{\ell} [B^{-}_{j,\ell+1}]^{\omega}=\sum_{j=1}^{\ell}[B_{j,\ell+1}]^{\omega}.$$
\\
\\
This is the density matrix of order $k$ which we obtain after $n$ Duhamel iterations in the hierarchy \eqref{randomizedGP2}. 
In this notation, the superscript $k$ denotes the order of the density matrix and the subscript $n$ denotes the length of the Duhamel expansion. The $\omega_j$ are the (fixed) randomization parameters in the probability space $\Omega$.
\\
\\
More precisely, let us fix $n \in \mathbb{N}$ and $(\omega_2, \omega_3, \ldots, \omega_{n+1}) \in \Omega_2  \times \Omega_3 \times \cdots \times \Omega_{n+1}$. 
Let us look at:
\begin{eqnarray*}\tilde{\gamma}^{(1)}&:=&\sigma^{(1)}_{n;\,\omega_2,\omega_3, \omega_4, \omega_5, \ldots, \,\omega_{n+1}}\\
\tilde{\gamma}^{(2)}&:=&\sigma^{(2)}_{n-1;\,\omega_3, \omega_4, \omega_5, \ldots,\, \omega_{n+1}}\\
\tilde{\gamma}^{(3)}&:=&\sigma^{(3)}_{n-2;\,\omega_4, \omega_5, \ldots, \,\omega_{n+1}}\\
&&\vdots\\
\tilde{\gamma}^{(n)}&:=&\sigma^{(n)}_{1;\,\omega_{n+1}}.\end{eqnarray*}

Then, by construction, we obtain that:

\begin{equation}
\notag
\begin{cases}
i \partial_t \tilde{\gamma}^{(k)} + (\Delta_{\vec{x}_k}-\Delta_{\vec{x'}_k})\tilde{\gamma}^{(k)}=\sum_{j=1}^{k} [B_{j,k+1}]^{\omega_{k+1}} (\tilde{\gamma}^{(k+1)})\\
\tilde{\gamma}^{(k)}\big|_{t=0}=0.
\end{cases}
\end{equation}
for all $k \in \{1,2,\ldots,n-1\}$. In other words, we obtain an \emph{arbitrarily long subset of solutions} to the full hierarchy \eqref{randomizedGP2} with zero initial data.
\\
\\
\begin{theorem}
\label{Smallness Bound}
Suppose that $\alpha>\frac{3}{4}$ and $k \in \mathbb{N}$.
Consider $\sigma^{(k)}_{n;\,\omega_{k+1},\omega_{k+2},\ldots,\omega_{n+k}}$, defined as in \eqref{precisedefinition2}. 
\\
\\
There exists $T>0$ depending only on the constant $C_1$ in \eqref{gammakcondition} and on $\alpha$ such that:
\begin{equation}
\label{firstclaim}
\,\, \sup_{t \in [0,T]} \, \big\|S^{(k,\alpha)} \sigma^{(k)}_{n;\,\omega_{k+1},\omega_{k+2},\ldots,\omega_{n+k}}(t)\big\|_{L^2 \big(\Omega_{k+1} \times \Omega_{k+2} \times \cdots \times \Omega_{n+k}; \,L^2(\Lambda^k \times \Lambda^k)\big)} \rightarrow 0
\end{equation} 
as $n \rightarrow \infty.$
\\
\\
Moreover,
\begin{equation}
\label{secondclaim}
\,\,\sup_{t \in [0,T]}\, \big\|S^{(k,\alpha)}\sigma^{(k)}_{n;\, \omega_{k+1},\omega_{k+2}, \ldots, \omega_{n+k}}(t)\big\|_{L^2\big(\prod_{m \geq 2} \Omega_m; L^2(\Lambda^k \times \Lambda^k)\big)} \rightarrow 0
\end{equation}
as $n \rightarrow \infty.$
\end{theorem}

\begin{proof}
We first prove \eqref{firstclaim}.
Throughout the proof, let us take $T>0$ small which will be precisely determined later and let us consider a fixed $t=t_k \in [0,T]$. We compute:

$$S^{(k,\alpha)}\sigma^{(k)}_{n;\,\omega_{k+1},\ldots,\omega_{n+k}}(t_k)=$$
$$(-i)^n \int_{0}^{t_k} \int_{0}^{t_{k+1}} \cdots \int_{0}^{t_{n+k-1}} S^{(k,\alpha)}\,\mathcal{U}^{(k)}(t_k-t_{k+1}) \, [B^{(k+1)}]^{\omega_{k+1}} \,\mathcal{U}^{(k+1)}(t_{k+1}-t_{k+2})$$
$$ [B^{(k+2)}]^{\omega_{k+2}} \cdots \,\mathcal{U}^{(n+k-1)}(t_{n+k-1}-t_{n+k}) \, [B^{(n+k)}]^{\omega_{n+k}} \, \gamma^{(n+k)}(t_{n+k}) \,dt_{n+k} \cdots dt_{k+2} \, dt_{k+1}.$$ 
Hence, we obtain by Minkowski's inequality:
\begin{equation}
\label{Skalphasigmakbound}
\Big\|S^{(k,\alpha)} \sigma^{(k)}_{n;\,\omega_{k+1},\omega_{k+2},\ldots,\omega_{n+k}}(t_k)\Big\|_{L^2 \big(\Omega_{k+1} \times \Omega_{k+2} \times \cdots \times \Omega_{n+k}; \,L^2(\Lambda^k \times \Lambda^k)\big)}
\end{equation}
$$\leq \int_{0}^{t_k} \int_{0}^{t_{k+1}} \cdots \int_{0}^{t_{n+k-1}} \Big\|S^{(k,\alpha)}\,\mathcal{U}^{(k)}(t_k-t_{k+1}) \, [B^{(k+1)}]^{\omega_{k+1}} \,\mathcal{U}^{(k+1)}(t_{k+1}-t_{k+2})$$
$$ [B^{(k+2)}]^{\omega_{k+2}} \cdots \,\mathcal{U}^{(n+k-1)}(t_{n+k-1}-t_{n+k}) \, [B^{(n+k)}]^{\omega_{n+k}} \, \gamma^{(n+k)}(t_{n+k})\Big\|_{L^2 \big(\Omega_{k+1} \times \Omega_{k+2} \times \cdots \times \Omega_{n+k}; \,L^2(\Lambda^k \times \Lambda^k)\big)}$$
$$ \,dt_{n+k} \cdots dt_{k+2} \, dt_{k+1}.$$ 
We look at the integrand in the $t$-variables, i.e. at:
$$\Big\|S^{(k,\alpha)}\,\mathcal{U}^{(k)}(t_k-t_{k+1}) \, [B^{(k+1)}]^{\omega_{k+1}} \,\mathcal{U}^{(k+1)}(t_{k+1}-t_{k+2})$$
$$ [B^{(k+2)}]^{\omega_{k+2}} \cdots \,\mathcal{U}^{(n+k-1)}(t_{n+k-1}-t_{n+k}) \, [B^{(n+k)}]^{\omega_{n+k}} \, \gamma^{(n+k)}(t_{n+k})\Big\|_{L^2 \big(\Omega_{k+1} \times \Omega_{k+2} \times \cdots \times \Omega_{n+k}; \,L^2(\Lambda^k \times \Lambda^k)\big)}$$

By using the unitarity of $\mathcal{U}^{(k)}(t_{k}-t_{k+1})$, and the fact that this operator commutes with $S^{(k,\alpha)}$, the above expression equals:
$$\big\|S^{(k,\alpha)} \, [B^{(k+1)}]^{\omega_{k+1}} \,\mathcal{U}^{(k+1)}(t_{k+1}-t_{k+2}) \,[B^{(k+2)}]^{\omega_{k+2}} \cdots \, \mathcal{U}^{(n+k-1)}(t_{n+k-1}-t_{n+k})$$
$$[B^{(n+k)}]^{\omega_{n+k}} \gamma^{(n+k)}(t_{n+k})\big\|_{L^2 \big(\Omega_{k+1} \times \Omega_{k+2} \times \cdots \times \Omega_{n+k}; \,L^2(\Lambda^k \times \Lambda^k)\big)}$$
We use the second bound from Theorem \ref{Theorem 1} to deduce that this expression is:
\begin{equation}
\label{kbound}
\leq C_0\,k \, \big\|S^{(k+1,\alpha)} [B^{(k+2)}]^{\omega_{k+2}} \,\mathcal{U}^{(k+2)}(t_{k+2}-t_{k+3})\cdots \, \mathcal{U}^{(n+k-1)}(t_{n+k-1}-t_{n+k})
\end{equation}
$$[B^{(n+k)}]^{\omega_{n+k}} \gamma^{(n+k)}(t_{n+k})\big\|_{L^2 \big(\Omega_{k+2} \times \Omega_{k+3} \times \cdots \times \Omega_{n+k}; \,L^2(\Lambda^{k+1} \times \Lambda^{k+1})\big)}$$
Here, $C_0$ denotes the constant $C$ from Corollary \ref{Corollary 1}. We note that the last step is justified since we are taking an $L^2$ norm in all of the variables and hence we can apply the randomized spacetime estimate in a ``vector-valued'' setting. We can see this by squaring both sides, integrating and using Fubini's theorem. Let us note that, in order to deduce the bound \eqref{kbound}, we also used the unitarity of $\mathcal{U}^{(k+1)}(t_k-t_{k+1})$ and the fact that this operator commutes with $S^{(k+1,\alpha)}$.
\\
\\
Iterating this procedure, we obtain:
$$\Big\|S^{(k,\alpha)}\,\mathcal{U}^{(k)}(t_k-t_{k+1}) \, [B^{(k+1)}]^{\omega_{k+1}} \,\mathcal{U}^{(k+1)}(t_{k+1}-t_{k+2})$$
$$ [B^{(k+2)}]^{\omega_{k+2}} \cdots \,\mathcal{U}^{(n+k-1)}(t_{n+k-1}-t_{n+k}) \, [B^{(n+k)}]^{\omega_{n+k}} \, \gamma^{(n+k)}(t_{n+k})\Big\|_{L^2 \big(\Omega_{k+1} \times \Omega_{k+2} \times \cdots \times \Omega_{n+k}; \,L^2(\Lambda^k \times \Lambda^k)\big)}$$
$$\leq \Big(C_0 k \Big) \, \cdot \, \Big(C_0  (k+1) \Big) \, \cdots \, \Big(C_0  (n+k)\Big) \cdot \Big\|S^{(n+k,\alpha)}\gamma^{(n+k)}(t_{n+k})\Big\|_{L^2(\Lambda^{n+k} \times \Lambda^{n+k})}.$$ 
By using the a priori bound on $\gamma^{(n+k)}$ (given by the condition \eqref{gammakcondition}), this is:
\begin{equation}
\label{Factoriel}
\leq \frac{(n+k)!}{(k-1)!} \, \cdot \, C_0^{n+1} \, \cdot \, C_1^{n+k}.
\end{equation}
We can assume, without loss of generality that $C_0 \geq 1$.
In particular, we can assume that $C_0^{n+1} \leq C_0^{n+k}$.
Hence, the expression in \eqref{Factoriel} is:
\begin{equation}
\label{Factoriel2}
\leq \frac{(n+k)!}{(k-1)!} \, \cdot M^{n+k}
\end{equation}
for some $M \sim C_0 \cdot C_1$, which is independent of $n$ and $k$.
The following identity is deduced by induction on $n$:
\begin{equation}
\label{factorielintegral}
\int_{0}^{t_k} \int_{0}^{t_{k+1}} \cdots \int_{0}^{t_{n+k-1}} \,dt_{n+k} \cdots dt_{k+2} \, dt_{k+1}=\frac{t_k^n}{n!}\,.
\end{equation}
The gain of $n!$ in the denominator will be useful in the analysis which follows. Namely, we can combine \eqref{Skalphasigmakbound}, \eqref{Factoriel2} and \eqref{factorielintegral} in order to deduce that:

\begin{equation}
\notag
\sup_{t \in [0,T]} \, \big\|S^{(k,\alpha)} \sigma^{(k)}_{n;\,\omega_{k+1},\omega_{k+2},\ldots,\omega_{n+k}}(t)\big\|_{L^2 \big(\Omega_{k+1} \times \Omega_{k+2} \times \cdots \times \Omega_{n+k}; \,L^2(\Lambda^k \times \Lambda^k)\big)}
\end{equation}
$$
\leq \frac{T^n}{n!} \, \cdot \frac{(n+k)!}{(k-1)!} \, \cdot M^{n+k} \leq  \frac{T^n}{n!} \, \cdot (n+k)! \cdot M^{n+k}=$$ 
$$= T^n \cdot (n+1) \cdot (n+2) \cdots (n+k) \cdot M^{n+k} \leq T^n \cdot (n+k)^k \cdot M^{n+k}
$$
$$\lesssim_{k} T^n \cdot n^k \cdot M^n \cdot M^k \lesssim_{k} T^n \cdot (2M)^n \cdot M^k=\big(2MT\big)^n \cdot M^k.$$
Here, we used the fact that $n^k = \mathcal{O}_k(2^n)$.
We now choose $T>0$ sufficiently small such that $2MT<1$.

It follows that, for $T$ chosen in this way:
$$\sup_{t \in [0,T]} \, \big\|S^{(k,\alpha)} \sigma^{(k)}_{n;\,\omega_{k+1},\omega_{k+2},\ldots,\omega_{n+k}}(t)\big\|_{L^2 \big(\Omega_{k+1} \times \Omega_{k+2} \times \cdots \times \Omega_{n+k}; \,L^2(\Lambda^k \times \Lambda^k)\big)} \rightarrow 0\,\,\mbox{as}\,\,n \rightarrow \infty.$$
By construction $T$ depends only on $M$. M in turn depends only on $C_1$ and $\alpha$. Hence, $T$ depends only on $C_1$ and $\alpha$. This proves the claim given in \eqref{firstclaim}.

In order to prove \eqref{secondclaim}, let us first state a non-trivial extension result from Probability theory that allows us to make sense of the norm in which we are studying the convergence. 

\begin{theorem} (\cite{Kakutani}, \cite{Kolmogorov}) 
\label{ProbabilityTheorem}
Suppose that $(\Omega_i,\Sigma_i,p_i)_{i \in \mathcal{I}}$ is an arbitrary, non-empty family of probability spaces. Then, there exists a probability measure $p$ on $\prod_{i \in \mathcal{I}} \Omega_i$ such that, for all finite subsets $F \subseteq I$:
\begin{equation}
\label{rectangle}
p \Big(\bigcap_{i \in F} \pi_i^{-1}(A_i)\Big)=\prod_{i \in F} p_i (A_i)
\end{equation}
whenever $A_i \in \Sigma_i$ for all $i \in F$. Here, for $j \in \mathcal{I}$, $\pi_j: \prod_{i \in \mathcal{I}} \Omega_i \rightarrow \Omega_j$ denotes the projection map onto $\Omega_j$.
\\
The obtained probability space is given by $(\prod_{i \in \mathcal{I}} \Omega_i,\mathcal{B},p)$, where $\mathcal{B}$ is the Borel field generated by sets of the form:
$$\bigcap_{i \in F} \pi_i^{-1}(A_i)$$
for some finite subset $F \subseteq \mathcal{I}$ and $A_i \in \Sigma_i$.
\end{theorem} 

The first result of this type was proved in the work of Kolmogorov \cite{Kolmogorov} in the case in which each probability space is a $[0,1]$ with Lebesgue measure. Kolmogorov's proof uses the fact that, in this context, the obtained product is topologically compact.  The result in full generality was proved in the work of Kakutani \cite{Kakutani}. This is the version stated in Theorem \ref{ProbabilityTheorem}. There are also related results by Doob \cite{Doob}. We use Theorem \ref{ProbabilityTheorem} in the case when $\mathcal{I}$ is a countable family. 


By \eqref{rectangle}, it follows that:
$$\big\|S^{(k,\alpha)}\sigma^{(k)}_{n;\, \omega_{k+1},\omega_{k+2}, \ldots, \omega_{n+k}}(t_k)\big\|_{L^2\big(\Omega_{k+1} \times \Omega_{k+2} \times \cdots \times \Omega_{n+k};\, L^2(\Lambda^k \times \Lambda^k)\big)}=$$
$$=\big\|S^{(k,\alpha)}\sigma^{(k)}_{n;\, \omega_{k+1},\omega_{k+2}, \ldots, \omega_{n+k}}(t_k)\big\|_{L^2\big(\mathbf{\Omega_{2}} \times \mathbf{\Omega_3} \times \cdots \times \mathbf{\Omega_{k}} \times \Omega_{k+1} \times \cdots \times \Omega_{n+k} \mathbf{\times \Omega_{n+k+1} \times \Omega_{n+k+2} \times \cdots};\, L^2(\Lambda^k \times \Lambda^k)\big)}=$$
$$=\big\|S^{(k,\alpha)}\sigma^{(k)}_{n;\, \omega_{k+1},\omega_{k+2}, \ldots, \omega_{n+k}}(t_k)\big\|_{L^2\big(\prod_{m \geq 2} \Omega_m; L^2(\Lambda^k \times \Lambda^k)\big)}.$$
Here, we used the fact that there is no dependence on $\omega_2, \omega_3, \ldots, \omega_{k-1}$ and on $\omega_{n+2},\omega_{n+3},...$ in $S^{(k,\alpha)}\sigma^{(k)}_{n;\, \omega_{k+1},\omega_{k+2}, \ldots, \omega_{n+k}}(t_k)$.
In particular, \eqref{firstclaim} implies that, for $T=T(C_1,\alpha)>0$,
$$\sup_{t \in [0,T]} \big\|S^{(k,\alpha)}\sigma^{(k)}_{n;\, \omega_{k+1},\omega_{k+2}, \ldots, \omega_{n+k}}(t)\big\|_{L^2\big(\prod_{m \geq 2} \Omega_m; L^2(\Lambda^k \times \Lambda^k)\big)} \rightarrow 0$$
as $n \rightarrow \infty$. The second claim of the theorem now follows.
\end{proof}

\begin{remark}
The analysis is the same, as well as the bound, if one considers negative times. In the discussion that follows, we consider non-negative times, for simplicity of notation.
\end{remark}

\begin{remark}
The above proof gives us that $T \sim_{\alpha} \frac{1}{C_1}$, where $C_1$ is the constant in \eqref{gammakcondition}.
\end{remark}
\begin{remark}
The above discussion can be applied on $\mathbb{T}^d$, whenever $d \geq 1$, provided that $\alpha>\frac{d}{4}$.
\end{remark}
\section{The randomized Gross-Pitaevskii hierarchy revisited}
\label{The randomized Gross-Pitaevskii hierarchy revisited}

In this section, we study the randomized Gross-Pitaevskii hierarchy \eqref{randomizedGP} in more detail. In the discussion that follows, we assume that $\alpha \geq 0$.  We note that all of the results of this section apply to the spatial domain $\Lambda=\mathbb{T}^d$ for $d \geq 1$.

The main result of this section is Theorem \ref{Smallness Bound 2}, which is the analogue of Theorem \ref{Smallness Bound} in the dependently randomized setting. In this context, the level of regularity is $\alpha \geq 0$. This should be compared to the assumption that $\alpha>\frac{3}{4}$, which we had to make when we were applying the spacetime estimate from Theorem \ref{Theorem 1} in the study of the independently randomized GP hierarchy \eqref{randomizedGP2}. Furthermore, since there is only one random parameter involved, the random component is simpler than the infinite product that we had to use in the independently randomized setting, and hence the spaces in which we will work will be simpler than those used in Section \ref{A new randomized hierarchy}.
	
Let us recall that the (depdendently) randomized Gross-Pitaevskii hierarchy is given by:

\begin{equation}
\notag
\begin{cases}
i \partial_t \gamma^{(k)} + (\Delta_{\vec{x}_k}-\Delta_{\vec{x'}_k})\gamma^{(k)}=\sum_{j=1}^{k} [B_{j,k+1}]^{\omega} (\gamma^{(k+1)})\\
\gamma^{(k)}\big|_{t=0}=\gamma_0^{(k)}.
\end{cases}
\end{equation}
As before, we take homogeneous initial data $\gamma_0^{(k)}=0$ and we study Duhamel terms of the type:
$$(-i)^n\int_{0}^{t_k} \int_{0}^{t_{k+1}} \cdots \int_{0}^{t_{n+k-1}} \mathcal{U}^{(k)}(t_k-t_{k+1}) \, [B^{(k+1)}]^{\omega} \,\mathcal{U}^{(k+1)}(t_{k+1}-t_{k+2})$$
$$ [B^{(k+2)}]^{\omega} \cdots \,\mathcal{U}^{(n+k-1)}(t_{n+k-1}-t_{n+k}) \, [B^{(n+k)}]^{\omega} \, \gamma^{(n+k)}(t_{n+k}) \,dt_{n+k} \cdots dt_{k+2} \, dt_{k+1}.$$ 
In this case, since the randomizations are no longer mutually independent, we cannot apply the estimate from Theorem \ref{Theorem 1} except in the special case that $n=1$. We explain this difficulty in more detail now.

\subsection{Difficulties arising from higher-order Duhamel expansions}
\label{Difficulties arising from higher-order Duhamel expansions}

If we try to extend the ideas from the proof of Theorem \ref{Theorem 1} to higher Duhamel expansions, we see that the pairings of the frequencies are not enough in order to deduce the needed estimate. In the following section, we will explain this point on an example. Let us note that these arguments can be extended to the case of a Duhamel expansion of length $n \geq 3$. For simplicity, we will construct an example when the length of the expansion equals $2$.

We will show that it is, in general, not possible to use the pairings due to the randomization in order to deduce that, for $\alpha>\frac{3}{4}$:
$$\big\|S^{(1,\alpha)} \, \mathcal{U}^{(1)}(t_1-t_2)\,[B^{+}_{1,2}]^{\omega} \,\mathcal{U}^{(2)}(t_2-t_3)\,[B^{+}_{2,3}]^{\omega} \gamma^{(3)}\big\|_{L^2\big(\Omega \times \Lambda \times \Lambda \big)} \lesssim \big\|S^{(1,\alpha)}\gamma^{(3)}\big\|_{L^2(\Lambda^3 \times \Lambda^3)}.$$
In other words:

\begin{equation}
\notag
\big\|S^{(1,\alpha)} \,[B^{+}_{1,2}]^{\omega} \,\mathcal{U}^{(2)}(t_2-t_3)\,[B^{+}_{2,3}]^{\omega} \gamma^{(3)}\big\|_{L^2\big(\Omega \times \Lambda \times \Lambda \big)} \lesssim \big\|S^{(1,\alpha)}\gamma^{(3)}\big\|_{L^2(\Lambda^3 \times \Lambda^3)}.
\end{equation}
The construction is as follows:
\\
\\
\begin{example}
Let us consider the case $t_2-t_3=0$. We show that the pairings due to the randomization do not allow us to prove:
\begin{equation}
\label{notabound}
\big\|S^{(1,\alpha)} \,[B^{+}_{1,2}]^{\omega}\,[B^{+}_{2,3}]^{\omega} \gamma^{(3)}\big\|_{L^2\big(\Omega \times \Lambda \times \Lambda \big)} \lesssim \big\|S^{(1,\alpha)}\gamma^{(3)}\big\|_{L^2(\Lambda^3 \times \Lambda^3)}.
\end{equation}

We compute:

$$\big([B_{1,2}^{+}]^{\omega} \, [B_{2,3}^{+}]^{\omega} \gamma^{(3)}]\big)\,\,\widehat{}\,\,(\xi_1; \xi'_1)=$$
$$=\sum_{\eta_2, \eta'_2}  \big([B_{2,3}^{+}]^{\omega} \gamma^{(3)}]\big)\,\,\widehat{}\,\,(\xi_1-\eta_2+\eta'_2, \eta_2; \xi'_1, \eta'_2) \cdot h_{\xi_1}(\omega) \cdot h_{\xi_1-\eta_2+\eta'_2}(\omega) \cdot h_{\eta_2}(\omega) \cdot h_{\eta'_2}(\omega)=$$
$$=\sum_{\eta_2,\eta'_2,\eta_3,\eta'_3} \widehat{\gamma}^{(3)}(\xi_1-\eta_2+\eta'_2, \eta_2-\eta_3+\eta'_3, \eta_3; \xi'_1,\eta'_2,\eta'_3) \cdot $$ 
\begin{equation}
\label{Duhamel2Sum}
h_{\xi_1}(\omega) \cdot h_{\xi_1-\eta_2+\eta'_2}(\omega) \cdot h_{\eta'_2}(\omega) \cdot h_{\eta_2-\eta_3+\eta'_3}(\omega) \cdot h_{\eta_3}(\omega) \cdot h_{\eta'_3}(\omega).
\end{equation}
In the above sum, we used the fact that $h_{\eta_2}^2(\omega)=1$.

Hence, we need to consider:
$$\sum_{\xi_1,\xi'_1} \mathop{\sum_{\eta_2,\eta'_2,\eta_3,\eta'_3}}_{\tilde{\eta}_2,\tilde{\eta}'_2,\tilde{\eta}_3,\tilde{\eta}'_3}
\langle \xi_1 \rangle^{\alpha} \cdot \langle \xi'_1 \rangle^{\alpha} \cdot
\widehat{\gamma}^{(3)}(\xi_1-\eta_2+\eta'_2, \eta_2-\eta_3+\eta'_3,\eta_3; \xi'_1,\eta'_2,\eta'_3) \cdot$$
$$\langle \xi_1 \rangle^{\alpha} \cdot \langle \xi'_1 \rangle^{\alpha} \cdot\widehat{\gamma}^{(3)}(\xi_1-\tilde{\eta}_2+\tilde{\eta}'_2, \tilde{\eta}_2-\tilde{\eta}_3+\tilde{\eta}'_3,\tilde{\eta}_3; \xi'_1,\tilde{\eta}'_2,\tilde{\eta}'_3) \cdot
$$
$$h_{\xi_1-\eta_2+\eta'_2}(\omega) \cdot h_{\eta_2-\eta_3+\eta'_3}(\omega) \cdot h_{\eta_3}(\omega) \cdot h_{\eta'_2}(\omega) \cdot h_{\eta'_3}(\omega)\cdot$$
$$
h_{\xi_1-\tilde{\eta}_2+\tilde{\eta}'_2}(\omega) \cdot h_{\tilde{\eta}_2-\tilde{\eta}_3+\tilde{\eta}'_3}(\omega) \cdot h_{\tilde{\eta}_3}(\omega) \cdot h_{\tilde{\eta}'_2}(\omega) \cdot h_{\tilde{\eta}'_3}(\omega).
$$
In the above formula, we used the fact that $h_{\xi_1}^2(\omega)=1$.

It is simple to apply the Cauchy-Schwarz inequality in $\xi'_1$.
We will henceforth consider the sum in the other variables. In particular, with the notation as in Definition \ref{matrixdefinition}, we need to consider:

\begin{equation}
\notag
\xi_1,\eta_2,\eta'_2,\eta_3, \eta'_3, \tilde{\eta}_2,\tilde{\eta}'_2,\tilde{\eta}_3,\tilde{\eta}'_3; \langle \xi_1 \rangle^{2\alpha} \cdot
\begin{bmatrix}
\langle \xi_1-\eta_2+\eta'_2 \rangle^{0} & \langle \eta_2-\eta_3+\eta'_3 \rangle^{0} & \langle \eta_3 \rangle^{0} & \langle \eta'_2 \rangle^{0} & \langle \eta'_3 \rangle^{0} \\
\langle \xi_1-\tilde{\eta}_2+\tilde{\eta}'_2 \rangle^{0} & \langle \eta_2-\tilde{\eta}_3+\tilde{\eta}'_3 \rangle^{0} & \langle \tilde{\eta}_3 \rangle^{0} & \langle \tilde{\eta}'_2 \rangle^{0} & \langle \tilde{\eta}'_3 \rangle^{0}
\end{bmatrix}
\end{equation}
with the additional condition that we have a pairing of frequencies that comes from the randomization.
By the randomization, we can obtain the analogue of property $\textbf{(*)}$ from the proof of Theorem \ref{Theorem 1}, i.e. we can deduce that each element in the set: 
$$\{\xi_1-\eta_2+\eta_2', \eta_2-\eta_3+\eta'_3, \eta_3, \eta'_2, \eta'_3,\xi_1-\tilde{\eta}_2+\tilde{\eta}_2', \tilde{\eta}_2-\tilde{\eta}_3+\tilde{\eta}'_3, \tilde{\eta}_3, \tilde{\eta}'_2, \tilde{\eta}'_3 \}$$ 
occurs at least twice in the list: 
$$(\xi_1-\eta_2+\eta_2', \eta_2-\eta_3+\eta'_3, \eta_3, \eta'_2, \eta'_3,\xi_1-\tilde{\eta}_2+\tilde{\eta}_2', \tilde{\eta}_2-\tilde{\eta}_3+\tilde{\eta}'_3, \tilde{\eta}_3, \tilde{\eta}'_2, \tilde{\eta}'_3).$$ 
The goal would be to distribute a factor of $\langle \xi_1 \rangle^{\alpha}$ over the first and second row of the matrix.
Let us consider one pairing in which we see that this is not possible:
\begin{equation}
\notag
\begin{cases}
\xi_1-\eta_2+\eta'_2=\eta_2-\eta_3+\eta'_3\\
\xi_1-\tilde{\eta}_2+\tilde{\eta}'_2=\tilde{\eta}_2-\tilde{\eta}_3+\tilde{\eta}'_3\\
\eta_3=\tilde{\eta}_3\\
\eta'_2=\eta'_3\\
\tilde{\eta}'_2=\tilde{\eta}'_3.
\end{cases}
\end{equation}
In this case, we note that:
$$\xi_1-\eta_2+\eta'_2=\eta_2-\eta_3+\eta'_3=\eta_2-\eta_3+\eta'_2.$$
Hence, it follows that:
$$\xi_1=2\eta_2-\eta_3.$$
We analogously obtain that:
$$\xi_1=2\tilde{\eta}_2-\tilde{\eta}_3=2\tilde{\eta}_2-\eta_3.$$
Consequently:
$$\eta_2=\tilde{\eta}_2.$$
Moreover, we can eliminate the sum in $\xi_1$ since $\xi_1=2\eta_2-\eta_3$.  Hence, the expression we need to bound is:

\begin{equation}
\notag
\eta_2,\eta'_2,\tilde{\eta}'_2,\eta_3; \langle 2\eta_2-\eta_3 \rangle^{2\alpha} \cdot
\begin{bmatrix}
\langle \eta_2-\eta_3+\eta'_2 \rangle^{0} & \langle \eta_2-\eta_3+\eta'_2  \rangle^{0} & \langle \eta_3 \rangle^{0} & \langle \eta'_2 \rangle^{0} & \langle \eta'_2 \rangle^{0} \\
\langle \eta_2-\eta_3+\tilde{\eta}'_2 \rangle^{0} & \langle \eta_2-\eta_3+\tilde{\eta}'_2  \rangle^{0} & \langle \eta_3 \rangle^{0} & \langle \tilde{\eta}'_2 \rangle^{0} & \langle \tilde{\eta}'_2 \rangle^{0}
\end{bmatrix}.
\end{equation}
We change variables  from $\eta_2-\eta_3$ to $\eta_2$ and the above sum becomes:
\begin{equation}
\notag
\eta_2,\eta'_2,\tilde{\eta}'_2,\eta_3; \langle 2\eta_2+\eta_3 \rangle^{2\alpha} \cdot
\begin{bmatrix}
\langle \eta_2+\eta'_2 \rangle^{0} & \langle \eta_2+\eta'_2  \rangle^{0} & \langle \eta_3 \rangle^{0} & \langle \eta'_2 \rangle^{0} & \langle \eta'_2 \rangle^{0} \\
\langle \eta_2+\tilde{\eta}'_2 \rangle^{0} & \langle \eta_2+\tilde{\eta}'_2  \rangle^{0} & \langle \eta_3 \rangle^{0} & \langle \tilde{\eta}'_2 \rangle^{0} & \langle \tilde{\eta}'_2 \rangle^{0}
\end{bmatrix}.
\end{equation}
At this point, we are supposed to distribute the two factors of $\langle 2\eta_2 + \eta_3 \rangle^{\alpha}$.
We can do this by using the fractional Leibniz rule in the first row as:
$$\langle 2\eta_2+\eta_3 \rangle^{\alpha} \lesssim \langle \eta_2+\eta'_2 \rangle^{\alpha} \cdot \langle \eta_2+\eta'_2 \rangle^{\alpha} \cdot \langle \eta_3 \rangle^{\alpha} \cdot \langle \eta'_2 \rangle^{\alpha} \cdot \langle \eta'_2 \rangle^{\alpha},$$
which follows from the fact that:
$$2\eta_2+\eta_3=(\eta_2+\eta'_2)+(\eta_2+\eta'_2)+\eta_3-\eta'_2-\eta'_2.$$
Similarly, we use the fractional Leibniz rule in the second row as:
$$\langle 2\eta_2+\eta_3 \rangle^{\alpha} \lesssim \langle \eta_2+\tilde{\eta}'_2 \rangle^{\alpha} \cdot \langle \eta_2+\tilde{\eta}'_2 \rangle^{\alpha} \cdot \langle \eta_3 \rangle^{\alpha} \cdot \langle \tilde{\eta}'_2 \rangle^{\alpha} \cdot \langle \tilde{\eta}'_2 \rangle^{\alpha}.$$
In particular, it suffices to estimate:
\begin{equation}
\notag
\eta_2,\eta'_2,\tilde{\eta}'_2,\eta_3;
\begin{bmatrix}
\langle \eta_2+\eta'_2 \rangle^{\alpha} & \langle \eta_2+\eta'_2  \rangle^{\alpha} & \langle \eta_3 \rangle^{\alpha} & \langle \eta'_2 \rangle^{\alpha} & \langle \eta'_2 \rangle^{\alpha} \\
\langle \eta_2+\tilde{\eta}'_2 \rangle^{\alpha} & \langle \eta_2+\tilde{\eta}'_2  \rangle^{\alpha} & \langle \eta_3 \rangle^{\alpha} & \langle \tilde{\eta}'_2 \rangle^{\alpha} & \langle \tilde{\eta}'_2 \rangle^{\alpha}
\end{bmatrix}.
\end{equation}
It is possible to estimate the sum in $\eta_2$ and in $\eta_3$ first by using the Cauchy-Schwarz inequality in these variables. If we want to 
estimate the sum in $\eta'_2$ or in $\tilde{\eta}'_2$, we need to lose derivatives in these variables and it is not possible to sum otherwise. This is due to the fact that the variables $\eta'_2$ and $\tilde{\eta}'_2$ occur in the same row of the matrix. 
\\
\\
In particular, we can use the previous method  to estimate the sum from above by:
$$\lesssim \big\|S^{(3,\alpha+\frac{3}{4}+\epsilon)}\gamma^{(3)}\big\|_{L^2(\Lambda^3 \times \Lambda^3)}^2.$$
Hence, in the case of the second Duhamel iteration, we cannot argue as in the Case IA3) in the proof of Theorem \ref{Theorem 1}. The reason why this is the case is that the frequency $\eta_2$ does not appear in the sum \eqref{Duhamel2Sum}. We are thus forced to use the fractional Leibniz rule in all of the variables.
\end{example}
Consequently, it is not possible to prove an estimate of the form \eqref{notabound} for the second Duhamel iteration in the class of general density matrices by using the pairings which come from the randomization. In what follows, we will explicitly write out the Duhamel expansion of arbitrary length. We will then use this formula to prove a randomized spacetime bound for a specific class of density matrices.

\subsection{The precise form of the Duhamel expansion term}
\label{The precise form of the Duhamel expansion term}
In this subsection, we explicitly write out the full form of the Duhamel expansion which occurs in the study of the randomized Gross-Pitaevskii hierarchy \eqref{randomizedGP}. In this expansion, we will again use the fact that the randomization is obtained by using standard Bernoulli random variables. We will use this explicit formula 
in the proof of the randomized spacetime estimates in non-resonant classes given in Proposition \ref{non-resonant1}
and Proposition \ref{non-resonant1} in the following two subsections.

The explicit expansion is obtained as follows.
Let us fix $n,\ell \in \mathbb{N}$ and $j_1,\ldots j_{\ell}, k_1, \ldots, k_{\ell} \in \mathbb{N}$, with $j_1<k_1 \leq n+\ell, \ldots, j_{\ell}<k_{\ell} \leq n+\ell$ and $t_1,t_2, \ldots, t_{\ell+1} \in \mathbb{R}$. We show that, for all sequences $(\gamma^{(m)})_m$ of density matrices, the following identity holds:
\begin{equation}
\label{Duhameln}
\Big(\mathcal{U}^{(n)}(t_1-t_2)\,[B^{\pm}_{j_1,k_1}]^{\omega}\,\mathcal{U}^{(n+1)}(t_2-t_3)\,[B^{\pm}_{j_2,k_2}]^{\omega} \cdots
\end{equation}
$$\cdots \mathcal{U}^{(n+\ell-1)}(t_{\ell}-t_{\ell+1}) \,[B^{\pm}_{j_{\ell},k_{\ell}}]^{\omega} \,\gamma^{(n+\ell)}\Big)\,\,\widehat{}\,\,(\xi_1,\ldots,\xi_n; \xi'_1,\ldots,\xi'_n)=$$

$$\mathop{\sum_{\eta_1,\ldots,\eta_{n+\ell}}}_{\eta'_1,\ldots,\eta'_{n+\ell}}^{*} e^{i(t_1-t_2)(\cdots)} \cdot e^{i(t_2-t_3)(\cdots)} \cdots e^{i(t_{\ell}-t_{\ell+1})(\cdots)} \cdot \widehat{\gamma}^{(n+\ell)}(\eta_1,\ldots,\eta_{n+\ell}; \eta'_1, \ldots, \eta'_{n+\ell}) \cdot$$
$$
\cdot \mathop{\prod_{1 \leq j \leq n}}_{\xi_j \in \mathcal{A}}
\Big\{h_{\xi_j}(\omega) \cdot h_{\eta_{1}^{j,1}} (\omega) \cdot h_{\eta_{2}^{j,1}} (\omega) \cdots h_{\eta_{N_j}^{j,1}}(\omega) \Big\} \cdot \mathop{\prod_{1 \leq j \leq n}}_{\xi'_j \in \mathcal{B}}
\Big\{h_{\xi'_j}(\omega) \cdot h_{\eta_{1}^{j,2}} (\omega) \cdot h_{\eta_{2}^{j,2}} (\omega) \cdots h_{\eta_{M_j}^{j,2}}(\omega) \Big\}.$$
Here, we need to explain the notation:

We say that $\xi_j \in \mathcal{A}$ for some $1 \leq j \leq n$ if a collision operator $[B_{j,k}^{\pm}]^{\omega}$ was applied to this frequency. In other words, $\xi_j \in \mathcal{A}$ if the frequency $\xi_j$ does not appear in the list $(\eta_1,\ldots,\eta_{n+\ell};\eta'_1,\ldots,\eta'_{n+\ell})$. 
Analogously, we say that $\xi'_j \in \mathcal{B}$ if a collision operator $[B_{j,k}^{\pm}]^{\omega}$ was applied to this frequency. The sets $\mathcal{A}$ and $\mathcal{B}$ are uniquely determined by the collision operators that we are applying.

Moreover, $$\mathop{\sum_{\eta_1,\ldots,\eta_{n+\ell}}}_{\eta'_1,\ldots,\eta'_{n+\ell}}^{*}(\cdots)$$ denotes the sum over $\eta_1,\ldots, \eta_{n+\ell}, \eta'_1,\ldots,\eta'_{n+\ell}$ which satisfy the following constraints:
\begin{itemize}
\item[$i)$] The set $\big\{\eta_1,\ldots,\eta_{n+\ell},\eta'_1,\ldots,\eta'_{n+\ell}\big\}$ can be written as the union of:
$$\Big(\mathop{\bigcup_{1 \leq j \leq n}}_{\xi_j \in \mathcal{A}}\big\{\eta^{j,1}_1,\ldots,\eta^{j,1}_{N_j}\big\}\Big) \cup \Big(\mathop{\bigcup_{1 \leq j \leq n}}_{\xi'_j \in \mathcal{B}}\big\{\eta^{j,2}_1,\ldots,\eta^{j,2}_{M_j}\big\}\Big)$$
(these are the new frequencies which we obtain after applying the collision operators) and:

$$\big\{\xi_1,\ldots,\xi_n,\xi'_1,\ldots,\xi'_n\big\} \setminus \big(\mathcal{A}\cup \mathcal{B}\big).$$
(these are the frequencies on which we do not apply the collision operators and which remain the same).
\item[$ii)$] Each $\xi_j \in \mathcal{A}$ can be written as:
$$\xi_j=\pm \eta_{1}^{j,1} \pm \eta_{2}^{j,1} \pm \cdots \pm \eta_{N_j}^{j,1}.$$

\item[$iii)$] Each $\xi'_j \in \mathcal{B}$ can be written as:
$$\xi'_j=\pm \eta_{1}^{j,2} \pm \eta_{2}^{j,2} \pm \cdots \pm \eta_{M_j}^{j,2}.$$
\end{itemize}
The choice of $\pm$ above is uniquely determined by the collision operators that we are applying. The choice of collision operators also uniquely determines the quantities $N_j$ and $M_j$ for $j=1,\ldots,n$.

To summarize, in $\stackrel{*}{\sum}$, we are summing over all the frequency configurations which satisfy the above conditions.
Let us note that the $(\cdots)$ in $e^{i(t_j-t_{j+1})(\cdots)}$ denotes a real-valued function in the frequencies. Consequently, the exponential has modulus one and hence does not affect the subsequent estimates. Finally, in the above calculations and all of its subsequent applications, we will always fix a choice of signs for each $[B^{\pm}_{j,k}]^{\omega}$.

Before we proceed, let us give an explicit example of the expansion \eqref{Duhameln}:
\begin{example}
We suppose that $n=2, \ell=3$. For simplicity of notation, let us analyze the case when $t_1=t_2=t_3=t_4=0$. As was noted earlier, in the general case, we would multiply all the terms by factors of modulus $1$. Let us consider the collision operators $[B_{1,2}^{+}]^{\omega}, [B_{2,3}^{-}]^{\omega}, [B_{4,5}^{-}]^{\omega}$.
\\
\\
In particular, we consider:
$$\Big([B_{1,2}^{+}]^{\omega} [B_{2,3}^{-}]^{\omega} [B_{4,5}^{-}]^{\omega} \gamma^{(5)} \Big)\,\,\widehat{}\,\,(\xi_1,\xi_2;\xi'_1,\xi'_2)=$$
$$=\sum_{\xi_3,\xi'_3} \Big([B_{2,3}^{-}]^{\omega} [B_{4,5}^{-}]^{\omega} \gamma^{(5)}\Big)\,\,\widehat{}\,\,(\xi_1-\xi_3+\xi'_3,\xi_3,\xi_2; \xi'_1,\xi'_3,\xi'_2) \cdot h_{\xi_1}(\omega) \cdot h_{\xi_1-\xi_3+\xi'_3}(\omega) \cdot h_{\xi_3}(\omega) \cdot h_{\xi'_3}(\omega)=$$
$$=\sum_{\xi_3,\xi'_3,\xi_4,\xi'_4} \Big([B_{4,5}^{-}]^{\omega} \gamma^{(5)} \Big)\,\,\widehat{}\,\,(\xi_1-\xi_3+\xi'_3,\xi_3,\xi_4,\xi_2; \xi'_1,\xi'_3-\xi'_4+\xi_4,\xi'_4,\xi'_2) \cdot$$
$$h_{\xi_1}(\omega) \cdot h_{\xi_1-\xi_3+\xi'_3}(\omega) \cdot h_{\xi_3}(\omega) \cdot h_{\xi'_3}(\omega) \cdot h_{\xi'_3}(\omega) \cdot h_{\xi'_3-\xi'_4+\xi_4}(\omega) \cdot h_{\xi_4}(\omega) \cdot h_{\xi'_4}(\omega)$$
\\
Since $h_{\xi'_3}^2(\omega)=1$, this expression equals:
$$\sum_{\xi_3,\xi'_3,\xi_4,\xi'_4} \Big([B_{4,5}^{-}]^{\omega} \gamma^{(5)} \Big)\,\,\widehat{}\,\,(\xi_1-\xi_3+\xi'_3,\xi_3,\xi_4,\xi_2; \xi'_1,\xi'_3-\xi'_4+\xi_4,\xi'_4,\xi'_2) \cdot$$
$$h_{\xi_1}(\omega) \cdot h_{\xi_1-\xi_3+\xi'_3}(\omega) \cdot h_{\xi_3}(\omega) \cdot h_{\xi'_3-\xi'_4+\xi_4}(\omega) \cdot h_{\xi_4}(\omega) \cdot h_{\xi'_4}(\omega)=$$
$$=\sum_{\xi_3,\xi'_3,\xi_4,\xi'_4,\xi_5,\xi'_5} \widehat{\gamma}^{(5)}(\xi_1-\xi_3+\xi'_3,\xi_3,\xi_4,\xi_2,\xi_5; \xi'_1,\xi'_3-\xi'_4+\xi_4,\xi'_4,\xi'_2-\xi'_5+\xi_5,\xi'_5) \cdot$$
$$h_{\xi_1}(\omega) \cdot h_{\xi_1-\xi_3+\xi'_3}(\omega) \cdot h_{\xi_3}(\omega) \cdot h_{\xi'_3-\xi'_4+\xi_4}(\omega) \cdot h_{\xi_4}(\omega) \cdot h_{\xi'_4}(\omega) \cdot$$
$$h_{\xi'_2}(\omega) \cdot h_{\xi'_2-\xi'_5+\xi_5}(\omega) \cdot h_{\xi_5}(\omega) \cdot h_{\xi'_5}(\omega).$$
Hence, we take $\mathcal{A}:=\{\xi_1\}$ and $\mathcal{B}:=\{\xi'_2\}$ and we note that:
$$\xi_1=(\xi_1-\xi_3+\xi'_3)+\xi_3+\xi_4-(\xi'_3-\xi'_4+\xi_4)-\xi'_4$$
$$\xi'_2=-\xi_5+(\xi'_2-\xi'_5+\xi_5)+\xi'_5.$$ 
In the notation from \eqref{Duhameln}, we hence take:
$$\eta_{1}^{1,1}:=\xi_1-\xi_3+\xi'_3, \eta_{2}^{1,1}:=\xi_3, \eta_{3}^{1,1}:=\xi_4, \eta_{4}^{1,1}:=\xi'_3-\xi'_4+\xi_4, \eta_{5}^{1,1}:=\xi'_4$$
$$\eta_{1}^{2,2}:=\xi_5,\eta_{2}^{2,2}:=\xi'_2-\xi'_5+\xi_5,\eta_{3}^{2,2}:=\xi'_5.$$
Consequently, we can write:
$$\Big([B_{1,2}^{+}]^{\omega} [B_{2,3}^{-}]^{\omega} [B_{4,5}^{-}]^{\omega} \gamma^{(5)} \Big)\,\,\widehat{}\,\,(\xi_1,\xi_2;\xi'_1,\xi'_2)=$$
$$=\mathop{\sum_{\eta^{1,1}_1+\eta^{1,1}_2+\eta^{1,1}_3-\eta^{1,1}_4-\eta^{1,1}_5=\,\xi_1}}_{-\eta^{2,2}_1+\eta^{2,2}_2+\eta^{2,2}_3=\,\xi'_2} \widehat{\gamma}^{(5)}(\eta^{1,1}_2,\eta^{1,1}_2,\eta^{1,1}_3,\xi_2,\eta^{2,2}_1;
\xi'_1,\eta^{1,1}_4,\eta^{1,1}_5,\eta^{2,2}_2,\eta^{2,2}_3) \cdot$$
$$h_{\xi_1}(\omega) \cdot h_{\eta^{1,1}_1}(\omega) \cdot h_{\eta^{1,1}_2}(\omega) \cdot h_{\eta^{1,1}_3}(\omega) \cdot h_{\eta^{1,1}_4}(\omega) \cdot h_{\eta^{1,1}_5}(\omega) \cdot$$
$$h_{\xi'_2}(\omega) \cdot h_{\eta^{2,2}_1}(\omega) \cdot h_{\eta^{2,2}_2}(\omega) \cdot h_{\eta^{2,2}_3}(\omega).$$
This is the expansion given in \eqref{Duhameln}.
\end{example}

We now prove the identity \eqref{Duhameln} in the general case by induction on $\ell$.
Both the base case $\ell=1$ and the inductive step follow from the form of the randomized collision operator. In particular, we note the  following two identities, for $r \in \mathbb{N}$ and for $j<k \leq r+1$:

\begin{itemize}
\item[1)] \textbf{Identity for $[B_{j,k}^{+}]^{\omega}$.}
\begin{equation}
\label{Duhameln1}
\Big(\mathcal{U}^{(r)}(t)[B^{+}_{j,k}]^{\omega}\gamma^{(r+1)}_{\omega}\Big)\,\,\widehat{}\,\,(\vec{\mu}_r; \vec{\mu}'_r)=\sum_{\lambda,\lambda'} e^{it(-|\mu_j-\lambda+\lambda'|^2-|\lambda|^2-|\vec{\mu}_r|^2+|\mu_j|^2+|\vec{\mu}'_r|^2+|\lambda'|^2)} \cdot 
\end{equation}
$$\widehat{\gamma}^{(r+1)}_{\omega}(\mu_1,\ldots,\mu_{j-1},\mu_j-\lambda+\lambda',\mu_{j+1},\ldots,\mu_{k-1},\lambda,\mu_{k+1},\ldots,\mu_r;\mu'_1,\ldots,\mu'_{k-1},\lambda',\mu'_{k+1},\ldots,\mu'_r) \cdot$$
$$h_{\mu_j}(\omega) \cdot h_{\mu_j-\lambda+\lambda'}(\omega) \cdot h_{\lambda}(\omega) \cdot h_{\lambda'}(\omega).$$
Here, the subscript $\omega$ in $\gamma^{(r+1)}_{\omega}$ denotes $\omega$ dependence in the density matrix (through factors of $h_{\xi}(\omega)$).
We note in the above formula, that $\mu_j$ is no longer taken to be a frequency in $\widehat{\gamma_{\omega}}^{(r+1)}$, but the random factor $h_{\mu_j}(\omega)$ corresponding to $\mu_j$ still occurs.
It is important to analyze this factor.
\\
\\
We observe that the following two cases can occur:
\begin{itemize}
\item[a)] $h_{\mu_j}(\omega)$ appears in $\widehat{\gamma_{\omega}}^{(r+1)}$; In this case we use the identity $h_{\mu_j}^2(\omega)=1$ and we cancel the factor of $h_{\mu_j}(\omega)$ in $\widehat{\gamma_{\omega}}^{(r+1)}$ with the factor of $h_{\mu_j}(\omega)$ in \eqref{Duhameln1}. Consequently, the term $h_{\mu_j}(\omega)$ corresponding to $\mu_j$ does not appear in the random part at this stage.
\item[b)] $h_{\mu_j}(\omega)$ does not appear in $\widehat{\gamma_{\omega}}^{(r+1)}$; In this case $\mu_j=\xi_{\ell}$ for some $1 \leq \ell \leq r$ and the factor $h_{\xi_{\ell}}(\omega)$ corresponding to $\xi_{\ell}$ appears in the random part at this stage. We say that $\xi_{\ell} \in \mathcal{A}$.
\end{itemize}
In either case, we have the following relation between the frequencies of the left and right-hand side of \eqref{Duhameln1}
\begin{equation}
\label{Duhameln1sum}
\mu_j=(\mu_j-\lambda+\lambda')+\lambda-\lambda'.
\end{equation}
\item[2)] \textbf{Identity for $[B_{j,k}^{-}]^{\omega}$.}
\\
For $\gamma^{(r+1)}_{\omega}$ as before, we note that:
\begin{equation}
\label{Duhameln2}
\Big(\mathcal{U}^{(r)}(t)[B^{-}_{j,k}]^{\omega}\gamma^{(r+1)}_{\omega}\Big)\,\,\widehat{}\,\,(\vec{\mu}_r;\vec{\mu}'_r)=\sum_{\lambda,\lambda'} e^{it(-|\vec{\mu}_r|^2-|\lambda|^2+|\mu'_j+\lambda-\lambda'|^2+|\lambda'|^2+|\vec{\mu}'_r|^2-|\mu_j|^2)} \cdot
\end{equation}
$$
\widehat{\gamma}^{(r+1)}_{\omega}(\mu_1,\ldots,\mu_{k-1},\lambda,\mu_{k+1},\ldots,\mu_n;\mu'_1,\ldots,\mu'_{j-1},\mu'_j+\lambda-\lambda',\mu'_{j+1},\ldots,\mu'_{k-1},\lambda',\mu'_{k+1},\ldots,\mu'_n) \cdot
$$
$$h_{\mu'_j}(\omega) \cdot h_{\mu'_j+\lambda-\lambda'}(\omega) \cdot h_{\lambda}(\omega) \cdot h_{\lambda'}(\omega).$$
As before, two cases can occur:
\begin{itemize}
\item[a)] $h_{\mu'_j}(\omega)$ appears in $\widehat{\gamma}^{(r+1)}_{\omega}$; In this case, we use the fact that $h_{\mu'_j}^2(\omega)=1$ and we deduce as before that the term $h_{\mu'_j}(\omega)$ corresponding to $\mu'_j$ does not appear in the random part at this stage.  

\item[b)] $h_{\mu'_j}(\omega)$ does not appear $\widehat{\gamma_{\omega}}^{(r+1)}$; In this case, we note that $\mu'_j=\xi'_{\ell}$ for some $1 \leq \ell \leq n$. Consequently, the factor $h_{\xi'_{\ell}}(\omega)$ corresponding to $\xi'_{\ell}$ appears in the random part at this stage. We say that $\xi'_{\ell} \in \mathcal{B}$.
\end{itemize}
We furthermore note the following relation between the frequencies occuring on the left and right-hand side of \eqref{Duhameln2}:
\begin{equation}
\label{Duhameln2sum}
\mu'_j=(\mu'_j-\lambda+\lambda')-\lambda+\lambda'.
\end{equation}
\end{itemize}
The identity \eqref{Duhameln} now follows by induction and by the use of \eqref{Duhameln1} and \eqref{Duhameln2}. The construction of the $\eta^{j,1}_{\ell}$ and $\eta^{j,2}_{\ell}$ is based on the relations \eqref{Duhameln1sum} and \eqref{Duhameln2sum}.

Similarly as before, we denote by:
$$\mathop{\sum_{\tilde{\eta}_1,\ldots,\tilde{\eta}_{n+\ell}}}_{\tilde{\eta}'_1,\ldots,\tilde{\eta}'_{n+\ell}}^{*}$$
the sum over $\tilde{\eta}_1,\ldots,\tilde{\eta}_{n+\ell}, \tilde{\eta}'_1,\ldots,\tilde{\eta}'_{n+\ell}$ with the analogous constraints satisfied when the $\eta$ and $\eta'$ variables are replaced by the $\tilde{\eta}$ and $\tilde{\eta}'$ variables.
In other words, for the same choice of collision operators and for the same definition of sets $\mathcal{A}$ and $\mathcal{B}$ and numbers $N_j$ and $M_j$ as before, we are summing over  $\tilde{\eta}_1,\ldots,\tilde{\eta}_{n+\ell}, \tilde{\eta}'_1,\ldots,\tilde{\eta}'_{n+\ell}$ such that:

\begin{itemize}
\item[$i)$] The set $\big\{\tilde{\eta}_1,\ldots,\tilde{\eta}_{n+\ell},\tilde{\eta}'_1,\ldots,\tilde{\eta}'_{n+\ell}\big\}$ can be written as the union of:
$$\Big(\mathop{\bigcup_{1 \leq j \leq n}}_{\xi_j \in \mathcal{A}}\big\{\tilde{\eta}^{j,1}_1,\ldots,\tilde{\eta}^{j,1}_{N_j}\big\}\Big) \cup \Big(\mathop{\bigcup_{1 \leq j \leq n}}_{\xi'_j \in \mathcal{B}}\big\{\tilde{\eta}^{j,2}_1,\ldots,\tilde{\eta}^{j,2}_{M_j}\big\}\Big)$$
(these are the new frequencies which we obtain after applying the collision operators) and:

$$\big\{\xi_1,\ldots,\xi_n,\xi'_1,\ldots,\xi'_n\} \setminus \big(\mathcal{A}\cup \mathcal{B}\big).$$
(these are the frequencies on which we do not apply the collision operators and which remain the same).
\item[$ii)$] Each $\xi_j \in \mathcal{A}$ can be written as:
$$\xi_j=\pm \tilde{\eta}_{1}^{j,1} \pm \tilde{\eta}_{2}^{j,1} \pm \cdots \pm \tilde{\eta}_{N_j}^{j,1}.$$

\item[$iii)$] Each $\xi'_j \in \mathcal{B}$ can be written as:
$$\xi'_j=\pm \tilde{\eta}_{1}^{j,2} \pm \tilde{\eta}_{2}^{j,2} \pm \cdots \pm \tilde{\eta}_{M_j}^{j,2}.$$
\end{itemize}

From \eqref{Duhameln}, we deduce that:

\begin{equation}
\label{Duhameln1A}
\Big|\Big(S^{(n,\alpha)}\,\mathcal{U}^{(n)}(t_1-t_2)\,[B^{\pm}_{j_1,k_1}]^{\omega}\,\mathcal{U}^{(n+1)}(t_2-t_3)\,[B^{\pm}_{j_2,k_2}]^{\omega} \cdots
\end{equation}
$$\cdots \,\mathcal{U}^{(n+\ell-1)}(t_{\ell}-t_{\ell+1}) \,[B^{\pm}_{j_{\ell},k_{\ell}}]^{\omega} \,\gamma^{(n+\ell)}\Big)\,\,\widehat{}\,\,(\xi_1,\ldots,\xi_n; \xi'_1,\ldots,\xi'_n)\Big|^2=$$
$$=\langle \xi_1 \rangle^{2\alpha} \cdots \langle \xi_n \rangle^{2\alpha} \cdot \langle \xi'_1 \rangle^{2\alpha} \cdots \langle \xi'_n \rangle^{2\alpha} \cdot \mathop{\sum_{\eta_1,\ldots,\eta_{n+\ell}}}_{\eta'_1,\ldots,\eta'_{n+\ell}}^{*}
\mathop{\sum_{\tilde{\eta}_1,\ldots,\tilde{\eta}_{n+\ell}}}_{\tilde{\eta}'_1,\ldots,\tilde{\eta}'_{n+\ell}}^{*}e^{i(t_1-t_2)(\cdots)} \cdot e^{i(t_2-t_3)(\cdots)} \cdots e^{i(t_{\ell}-t_{\ell+1})(\cdots)} \cdot $$
$$ \widehat{\gamma}^{(n+\ell)}(\eta_1,\ldots,\eta_{n+\ell}; \eta'_1,\ldots,\eta'_{n+\ell}) \cdot \overline{\widehat{\gamma}^{(n+\ell)}}(\tilde{\eta}_1,\ldots,\tilde{\eta}_{n+\ell}; \tilde{\eta}'_1,\ldots,\tilde{\eta}'_{n+\ell}) \cdot
$$
$$
\cdot \mathop{\prod_{1 \leq j \leq n}}_{\xi_j \in \mathcal{A}}
\Big\{h_{\xi_j}(\omega) \cdot h_{\eta_{1}^{j,1}} (\omega) \cdot h_{\eta_{2}^{j,1}} (\omega) \cdots h_{\eta_{N_j}^{j,1}}(\omega) \Big\} \cdot \mathop{\prod_{1 \leq j \leq n}}_{\xi'_j \in \mathcal{B}}
\Big\{h_{\xi'_j}(\omega) \cdot h_{\eta_{1}^{j,2}} (\omega) \cdot h_{\eta_{2}^{j,2}} (\omega) \cdots h_{\eta_{M_j}^{j,2}}(\omega) \Big\} \cdot$$
$$
\cdot \mathop{\prod_{1 \leq j \leq n}}_{\xi_j \in \mathcal{A}}
\Big\{h_{\xi_j}(\omega) \cdot h_{\tilde{\eta}_{1}^{j,1}} (\omega) \cdot h_{\tilde{\eta}_{2}^{j,1}} (\omega) \cdots h_{\tilde{\eta}_{N_j}^{j,1}}(\omega) \Big\} \cdot \mathop{\prod_{1 \leq j \leq n}}_{\xi'_j \in \mathcal{B}}
\Big\{h_{\xi'_j}(\omega) \cdot h_{\tilde{\eta}_{1}^{j,2}} (\omega) \cdot h_{\tilde{\eta}_{2}^{j,2}} (\omega) \cdots h_{\tilde{\eta}_{M_j}^{j,2}}(\omega) \Big\}=$$
$$=\langle \xi_1 \rangle^{2\alpha} \cdots \langle \xi_n \rangle^{2\alpha} \cdot \langle \xi'_1 \rangle^{2\alpha} \cdots \langle \xi'_n \rangle^{2\alpha} \cdot \mathop{\sum_{\eta_1,\ldots,\eta_{n+\ell}}}_{\eta'_1,\ldots,\eta'_{n+\ell}}^{*}
\mathop{\sum_{\tilde{\eta}_1,\ldots,\tilde{\eta}_{n+\ell}}}_{\tilde{\eta}'_1,\ldots,\tilde{\eta}'_{n+\ell}}^{*}e^{i(t_1-t_2)(\cdots)} \cdot e^{i(t_2-t_3)(\cdots)} \cdots e^{i(t_{\ell}-t_{\ell+1})(\cdots)} \cdot $$
$$ \widehat{\gamma}^{(n+\ell)}(\eta_1,\ldots,\eta_{n+\ell}; \eta'_1,\ldots,\eta'_{n+\ell}) \cdot \overline{\widehat{\gamma}^{(n+\ell)}}(\tilde{\eta}_1,\ldots,\tilde{\eta}_{n+\ell}; \tilde{\eta}'_1,\ldots,\tilde{\eta}'_{n+\ell}) \cdot
$$
$$
\cdot \mathop{\prod_{1 \leq j \leq n}}_{\xi_j \in \mathcal{A}}
\Big\{h_{\eta_{1}^{j,1}} (\omega) \cdot h_{\eta_{2}^{j,1}} (\omega) \cdots h_{\eta_{N_j}^{j,1}}(\omega) \Big\} \cdot \mathop{\prod_{1 \leq j \leq n}}_{\xi'_j \in \mathcal{B}}
\Big\{h_{\eta_{1}^{j,2}} (\omega) \cdot h_{\eta_{2}^{j,2}} (\omega) \cdots h_{\eta_{M_j}^{j,2}}(\omega) \Big\} \cdot$$
$$
\cdot \mathop{\prod_{1 \leq j \leq n}}_{\xi_j \in \mathcal{A}}
\Big\{h_{\tilde{\eta}_{1}^{j,1}} (\omega) \cdot h_{\tilde{\eta}_{2}^{j,1}} (\omega) \cdots h_{\tilde{\eta}_{N_j}^{j,1}}(\omega) \Big\} \cdot \mathop{\prod_{1 \leq j \leq n}}_{\xi'_j \in \mathcal{B}}
\Big\{h_{\tilde{\eta}_{1}^{j,2}} (\omega) \cdot h_{\tilde{\eta}_{2}^{j,2}} (\omega) \cdots h_{\tilde{\eta}_{M_j}^{j,2}}(\omega) \Big\}.$$
Here, we used the fact that $h_{\xi_j}^2(\omega)=h_{\xi'_j}^2(\omega)=1$. 
Let us note that the factors $e^{i(t_j-t_{j+1})(\cdots)}$ in \eqref{Duhameln} and in \eqref{Duhameln1A} are not necessarily the same. In what follows, we will only use the fact that these factors have modulus equal to $1$.

From \eqref{Duhameln1A} and Plancherel's theorem, it follows that:
\begin{equation}
\label{Duhamelnsquared}
\Big\|S^{(n,\alpha)}\mathcal{U}^{(n)}(t_1-t_2)\,[B^{\pm}_{j_1,k_1}]^{\omega}\,\mathcal{U}^{(n+1)}(t_2-t_3)\,[B^{\pm}_{j_2,k_2}]^{\omega} \cdots
\end{equation}
$$\cdots \mathcal{U}^{(n+\ell-1)}(t_{\ell}-t_{\ell+1}) \,[B^{\pm}_{j_{\ell},k_{\ell}}]^{\omega} \,\gamma^{(n+\ell)}\Big\|_{L^2(\Omega \times \Lambda^n \times \Lambda^n)}^2=$$
$$=\int_{\Omega} \Big\{ \mathop{\sum_{\xi_1,\ldots,\xi_n}}_{\xi'_1,\ldots,\xi'_n} \langle \xi_1 \rangle^{2\alpha} \cdots \langle \xi_n \rangle^{2\alpha} \cdot \langle \xi'_1 \rangle^{2\alpha} \cdots \langle \xi'_n \rangle^{2\alpha} \cdot \mathop{\sum_{\eta_1,\ldots,\eta_{n+\ell}}}_{\eta'_1,\ldots,\eta'_{n+\ell}}^{*}
\mathop{\sum_{\tilde{\eta}_1,\ldots,\tilde{\eta}_{n+\ell}}}_{\tilde{\eta}'_1,\ldots,\tilde{\eta}'_{n+\ell}}^{*}  e^{i(t_1-t_2)(\cdots)} \cdot e^{i(t_2-t_3)(\cdots)} \cdots e^{i(t_{\ell}-t_{\ell+1})(\cdots)}  \cdot$$
$$\widehat{\gamma}^{(n+\ell)}(\eta_1,\ldots,\eta_{n+\ell}; \eta'_1,\ldots,\eta'_{n+\ell}) \cdot \overline{\widehat{\gamma}^{(n+\ell)}}(\tilde{\eta}_1,\ldots,\tilde{\eta}_{n+\ell}; \tilde{\eta}'_1,\ldots,\tilde{\eta}'_{n+\ell}) \cdot
$$
$$\cdot \mathop{\prod_{1 \leq j \leq n}}_{\xi_j \in \mathcal{A}}
\Big\{h_{\eta_{1}^{j,1}} (\omega) \cdot h_{\eta_{2}^{j,1}} (\omega) \cdots h_{\eta_{N_j}^{j,1}}(\omega) \cdot h_{\tilde{\eta}_{1}^{j,1}} (\omega) \cdot h_{\tilde{\eta}_{2}^{j,1}} (\omega) \cdots h_{\tilde{\eta}_{N_j}^{j,1}}(\omega) \Big\} \cdot$$
$$ \cdot \mathop{\prod_{1 \leq j \leq n}}_{\xi'_j \in \mathcal{B}}
\Big\{h_{\eta_{1}^{j,2}} (\omega) \cdot h_{\eta_{2}^{j,2}} (\omega) \cdots h_{\eta_{M_j}^{j,2}}(\omega) \cdot h_{\tilde{\eta}_{1}^{j,2}} (\omega) \cdot h_{\tilde{\eta}_{2}^{j,2}} (\omega) \cdots h_{\tilde{\eta}_{M_j}^{j,2}}(\omega) \Big\} \Big\} \,dp(\omega).$$


This is the quantity that we would like to estimate.

\subsection{A special class of density matrices}
\label{A special class of density matrices}

In this subsection, we study a special class of density matrices $(\gamma^{(m)})_m$ in which we can prove an upper bound for the expression in \eqref{Duhamelnsquared}. We note that by the results in Subsection \ref{Difficulties arising from higher-order Duhamel expansions}, we cannot estimate this quantity using pairing of frequencies due to randomization in the general class of density matrices, even in the case when $\ell=2$. As we will see below, we will be able to prove an upper bound on \eqref{Duhamelnsquared} if we impose additional \emph{non-resonance} conditions. This is reminiscent of the ideas in \cite{B2,COh,NS}. The precise bound is given in Proposition \ref{non-resonant1} below.

More precisely, we consider:

\begin{definition}
\label{Class_N}
Let $\mathcal{N}$ denote the class of all density matrices $(\gamma^{(m)})_m$ such that:
\begin{itemize}
\item[$i)$] For all $m \in \mathbb{N}$ and for all $(\xi_1,\ldots,\xi_m,\xi'_1,\ldots,\xi'_m) \in (\mathbb{Z}^3)^m \times  (\mathbb{Z}^3)^m$, $$\widehat{\gamma}^{(m)}(\xi_1,\ldots,\xi_m;\xi'_1,\ldots,\xi'_m)=0$$ unless 
$$|\xi_1|>|\xi_2|>\cdots>|\xi_m|>|\xi'_1|>|\xi'_2|>\cdots |\xi'_m|.$$
\item[$ii)$] There exists $C_1>0$ independent of $m$ such that for all $m \in \mathbb{N}$:
$$\big\|S^{(m,\alpha)}\gamma^{(m)}\big\|_{L^2(\Lambda^m \times \Lambda^m)} \leq C_1^m.$$
\end{itemize}
\end{definition}

We call the class $\mathcal{N}$ the class of \emph{non-resonant density matrices.}
Let us note that the elements in $\mathcal{N}$ are \emph{not symmetric} in the sense that it is not in general true that for all $m \in \mathbb{N}$, $(x_1,\ldots,x_m,x'_1,\ldots,x'_m) \in \Lambda^m \times  \Lambda^m$, $\sigma \in S_m$:
$$\gamma^{(m)}(x_{\sigma(1)},\ldots,x_{\sigma(m)}; x'_{\sigma(1)}, \ldots, x'_{\sigma(m)})=\gamma^{(m)}(x_1,\ldots,x_m;x'_1,\ldots,x_m).$$
Furthermore, we observe that the class $\mathcal{N}$ is non-empty. We can construct an element of $\mathcal{N}$ by taking a density matrix satisfying the condition $ii)$ and by then setting all of the Fourier coefficients not satisfying the constraint from $i)$ to equal zero. 
Finally, let us note that the class $\mathcal{N}$ does not contain the factorized objects studied earlier. In the analysis that follows, we see that the ordering in condition $i)$ on the sizes of the frequencies $\xi_1, \ldots, \xi_m, \xi'_1, \ldots \xi'_m$ could be replaced by any other (fixed) ordering of the elements of $\{1,2,\ldots,2m\}$.
\\
\\
The rest of this section is devoted to the study of the randomized Gross-Pitaevskii hierarchy in the non-resonant class $\mathcal{N}$.


\subsection{The randomized estimate in the class $\mathcal{N}$}
\label{Estimate in N}

We use the randomization in \eqref{Duhamelnsquared} in order to deduce that each element of:

$$\mathcal{S}:=\Big(\mathop{\bigcup_{1 \leq j \leq n}}_{\xi_j \in \mathcal{A}}\big\{\eta^{j,1}_1,\ldots,\eta^{j,1}_{N_j}\big\}\Big) \cup \Big(\mathop{\bigcup_{1 \leq j \leq n}}_{\xi'_j \in \mathcal{B}}\big\{\eta^{j,2}_1,\ldots,\eta^{j,2}_{M_j}\big\}\Big) $$
gets paired with at least one element of:
$$\tilde{\mathcal{S}}:=\Big(\mathop{\bigcup_{1 \leq j \leq n}}_{\xi_j \in \mathcal{A}}\big\{\tilde{\eta}^{j,1}_1,\ldots,\tilde{\eta}^{j,1}_{N_j}\big\}\Big) \cup \Big(\mathop{\bigcup_{1 \leq j \leq n}}_{\xi'_j \in \mathcal{B}}\big\{\tilde{\eta}^{j,2}_1,\ldots,\tilde{\eta}^{j,2}_{M_j}\big\}\Big). $$
We recall that, by definition of the class $\mathcal{N}$, the elements of each of the sets $\mathcal{S}$ and $\tilde{\mathcal{S}}$ are mutually distinct. Since $\# \mathcal{S}= \# \tilde{\mathcal{S}}$, it follows that each element of $\mathcal{S}$ gets paired with exactly one element of $\tilde{\mathcal{S}}$ and vice versa.
Let us recall the \emph{fractional Leibniz rule}, which states that for all $\alpha \geq 0$, there exists $C_3>0$ such that, for all $x,y \in \mathbb{R}^3$:
\begin{equation}
\label{3DFractionalLeibnizRule}
\langle x+y \rangle^{\alpha} \leq C_3 \langle x \rangle^{\alpha} \cdot \langle y \rangle^{\alpha}.
\end{equation}
The constant $C_3$ depends on $\alpha$.
We iteratively use the fractional Leibniz rule and the fact that there is a bijection between $\mathcal{S}$ and $\tilde{\mathcal{S}}$ to deduce that, for $\alpha \geq 0$:
\begin{equation}
\label{Duhamelnsquared2}
\big\|S^{(n,\alpha)}\,\mathcal{U}^{(n)}(t_1-t_2)\,[B^{\pm}_{j_1,k_1}]^{\omega}\,\mathcal{U}^{(n+1)}(t_2-t_3)\,[B^{\pm}_{j_2,k_2}]^{\omega} \cdots
\end{equation}
$$\cdots \,\mathcal{U}^{(n+\ell-1)}(t_{\ell}-t_{\ell+1}) \,[B^{\pm}_{j_{\ell},k_{\ell}}]^{\omega} \,\gamma^{(n+\ell)}\big\|_{L^2(\Omega \times \Lambda^n \times \Lambda^n)}^2 \leq$$
$$\leq C_3^{\,2(n+\ell)} \mathop{\sum_{\eta_1,\ldots,\eta_{n+\ell}, \eta'_1, \ldots, \eta'_{n+\ell}}}_{\tilde{\eta}_1,\ldots,\tilde{\eta}_{n+\ell},\tilde{\eta}'_1,\ldots,\tilde{\eta}'_{n+\ell}}^{**} \langle \eta_1 \rangle^{\alpha} \cdots \langle \eta_{n+\ell} \rangle^{\alpha} \cdot \langle \eta'_1 \rangle^{\alpha} \cdots \langle \eta'_{n+\ell} \rangle^{\alpha} \cdot \langle \tilde{\eta}_1 \rangle^{\alpha} \cdots \langle \tilde{\eta}_{n+\ell} \rangle^{\alpha} \cdot \langle \tilde{\eta}'_1 \rangle^{\alpha} \cdots \langle \tilde{\eta}'_{n+\ell} \rangle^{\alpha}.$$
$$\cdot \big|\widehat{\gamma}^{(n+\ell)}(\eta_1,\ldots,\eta_{n+\ell};\eta'_1,\ldots,\eta'_{n+\ell})\big| \cdot \big|\widehat{\gamma}^{(n+\ell)}(\tilde{\eta}_1,\ldots,\tilde{\eta}_{n+\ell};\tilde{\eta}'_1,\ldots,\tilde{\eta}'_{n+\ell})\big|.$$
More precisely, we use conditions $ii)$ and $iii)$ 
from the definitions of the sets $\{\eta_1,\ldots, \eta_{n+\ell},\eta'_1,\ldots,\eta'_{n+\ell}\}$ and $\{\tilde{\eta}_1,\ldots,\tilde{\eta}_{n+\ell},\tilde{\eta}'_1,\ldots,\tilde{\eta}'_{n+\ell}\}$
in the sums: $$\mathop{\sum_{\eta_1,\ldots,\eta_{n+\ell}}}_{\eta'_1,\ldots,\eta'_{n+\ell}}^{*}(\cdots)$$ 
and
$$\mathop{\sum_{\tilde{\eta}_1,\ldots,\tilde{\eta}_{n+\ell}}}_{\tilde{\eta}'_1,\ldots,\tilde{\eta}'_{n+\ell}}^{*}(\cdots)$$ 
and we apply the estimate \eqref{3DFractionalLeibnizRule} at most $2(n+\ell)$ times to obtain the factor of $C_3^{\,2(n+\ell)}$ in the bound. Here,
$$\mathop{\sum_{\eta_1,\ldots,\eta_{n+\ell}, \eta'_1, \ldots, \eta'_{n+\ell}}}_{\tilde{\eta}_1,\ldots,\tilde{\eta}_{n+\ell},\tilde{\eta}'_1,\ldots,\tilde{\eta}'_{n+\ell}}^{**}$$
denotes the sum $$\mathop{\sum_{\eta_1,\ldots,\eta_{n+\ell}}}_{\eta'_1,\ldots,\eta'_{n+\ell}}^{*}
\mathop{\sum_{\tilde{\eta}_1,\ldots,\tilde{\eta}_{n+\ell}}}_{\tilde{\eta}'_1,\ldots,\tilde{\eta}'_{n+\ell}}^{*}$$  with the additional constraint that there exists a bijective pairing between the sets $\mathcal{S}$ and $\tilde{\mathcal{S}},$ which was noted earlier. 
We also use the fact that $|e^{i(t_j-t_{j+1})(\cdots)}|=1$ for all $j=1,\ldots,\ell.$ Finally, for frequencies $\xi_j \notin \mathcal{A}$ and $\xi'_k \notin \mathcal{B}$, we use the fact that
$\langle \xi_j \rangle^{\alpha} \geq 1$ and $\langle \xi'_k \rangle^{\alpha} \geq 1$.
\\
\\
If $(\gamma^{(m)})_m \in \mathcal{N}$, we can use condition $i)$ in the definition $\mathcal{N}$ in order to deduce that the only possible pairing is given by:
$$\eta_1=\tilde{\eta}_1, \eta_2=\tilde{\eta}_2,\ldots, \eta_{n+\ell}=\tilde{\eta}_{n+\ell}$$
$$\eta'_1=\tilde{\eta}'_1, \eta'_2=\tilde{\eta}'_2, \ldots, \eta'_{n+\ell}=\tilde{\eta}'_{n+\ell}.$$
Consequently, the right-hand side of the expression in \eqref{Duhamelnsquared2} equals:

$$C_3^{\,2(n+\ell)} \mathop{\sum_{\eta_1,\ldots,\eta_{n+\ell}}}_{\eta'_1,\ldots,\eta'_{n+\ell}} \langle \eta_1 \rangle^{2\alpha} \cdots \langle \eta_{n+\ell} \rangle^{2\alpha} \cdot \langle \eta'_1 \rangle^{2\alpha} \cdots \langle \eta'_{n+\ell} \rangle^{2\alpha} \cdot$$
$$\big| \widehat{\gamma}^{(n+\ell)}(\eta_1,\ldots,\eta_{n+\ell};\eta'_1, \ldots, \eta'_{n+\ell})\big|^2 = C_3^{\,2(n+\ell)} \|S^{(n+\ell,\alpha)}\gamma^{(n+\ell)}\|_{L^2(\Lambda^{n+\ell} \times \Lambda^{n+\ell})}^2.$$
We hence deduce the following result:

\begin{proposition}
\label{non-resonant1}
Suppose that $\alpha \geq 0$. Then, there exists $C_2>0$, depending only on $\alpha$ such that for all $(\gamma^{(m)})_m \in \mathcal{N}$, for all $n,\ell \in \mathbb{N};j_1,\ldots j_{\ell}, k_1, \ldots, k_{\ell} \in \mathbb{N}$, with $j_1<k_1 \leq n+\ell, \ldots, j_{\ell}<k_{\ell} \leq n+\ell$ and $t_1,t_2, \ldots, t_{\ell+1} \in \mathbb{R}$, the following bound holds:
$$\big\|S^{(n,\alpha)}\mathcal{U}^{(n)}(t_1-t_2)\,[B^{\pm}_{j_1,k_1}]^{\omega}\,\mathcal{U}^{(n+1)}(t_2-t_3)\,[B^{\pm}_{j_2,k_2}]^{\omega} \cdots$$
$$\cdots \mathcal{U}^{(n+\ell-1)}(t_{\ell}-t_{\ell+1}) \,[B^{\pm}_{j_{\ell},k_{\ell}}]^{\omega} \,\gamma^{(n+\ell)}\big\|_{L^2(\Omega \times \Lambda^n \times \Lambda^n)} \leq C_2^{\,n+\ell} \|S^{(n+\ell,\alpha)} \gamma^{(n+\ell)}\|_{L^2(\Lambda^{n+\ell} \times \Lambda^{n+\ell})}.$$
\end{proposition}
\begin{proof}
The result immediately follows from the above discussion if we take  $C_2$ to equal the constant $C_3$ from the estimate \eqref{3DFractionalLeibnizRule}.
\end{proof}

\begin{remark}
We observe that we are only using condition $i)$ in the definition of the class $\mathcal{N}$.
\end{remark}

\subsection{An application of the randomized estimate to the study of the dependently randomized GP hierarchy}
\label{The randomized Gross-Pitaevskii hierarchy in the class N}


We now apply the randomized estimate from Proposition \ref{non-resonant1} to the study of the randomized Gross-Pitaevskii hierarchy \eqref{randomizedGP}. We argue similarly as before and we fix $(\gamma^{(m)}(t))_m$ to be a time-dependent sequence in  $\mathcal{N}$. Here, we assume that the constant $C_1$ from part $ii)$ of the definition of $\mathcal{N}$ is uniform in $t$.  In this case, we say that $(\gamma^{(m)}(t))_m$ belongs to $\mathcal{N}$ \emph{uniformly in time}.
\\
\\
Given $k,n \in \mathbb{N}$, $t_k>0$ and $\omega \in \Omega$, we define $\sigma^{(k)}_{n;\,\omega}$ by:
\begin{equation}
\label{precisedefinition3}
\sigma^{(k)}_{n;\,\omega}(t_k):=
\end{equation}
$$(-i)^n \int_{0}^{t_k} \int_{0}^{t_{k+1}} \cdots \int_{0}^{t_{n+k-1}} \mathcal{U}^{(k)}(t_k-t_{k+1}) \, [B^{(k+1)}]^{\omega} \,\mathcal{U}^{(k+1)}(t_{k+1}-t_{k+2})$$
$$ [B^{(k+2)}]^{\omega} \cdots \,\mathcal{U}^{(n+k-1)}(t_{n+k-1}-t_{n+k}) \, [B^{(n+k)}]^{\omega} \, \gamma^{(n+k)}(t_{n+k}) \,dt_{n+k} \cdots dt_{k+2} \, dt_{k+1}.$$ 
\\
\\
This is the $k$-density matrix which we obtain after $n$ Duhamel iterations in the  hierarchy \eqref{randomizedGP}.
In this notation, the superscript $k$ denotes the number of particles and the subscript $n$ denotes the length of the Duhamel expansion. Furthermore, $\omega \in \Omega$ is a fixed randomization parameter.
\\
\\
Given $(\sigma^{(k)}_{n;\,\omega})_{n,k}$, we define:
\begin{eqnarray*}
\tilde{\gamma}^{(1)}&:=&\sigma^{(1)}_{n;\,\omega}\\
\tilde{\gamma}^{(2)}&:=&\sigma^{(2)}_{n-1;\,\omega}\\
\tilde{\gamma}^{(3)}&:=&\sigma^{(3)}_{n-2;\,\omega}\\
&&\vdots \\
\tilde{\gamma}^{(n)}&:=&\sigma^{(n)}_{1;\,\omega}.
\end{eqnarray*}

By construction, it follows that:

\begin{equation}
\notag
\begin{cases}
i \partial_t \tilde{\gamma}^{(k)} + (\Delta_{\vec{x}_k}-\Delta_{\vec{x'}_k})\tilde{\gamma}^{(k)}=\sum_{j=1}^{k} [B_{j,k+1}]^{\omega} (\tilde{\gamma}^{(k+1)})\\
\tilde{\gamma}^{(k)}\big|_{t=0}=0.
\end{cases}
\end{equation}
for all $k \in \{1,2,\ldots,n-1\}$. In other words, we obtain an \emph{arbitrarily long subset of solutions} to the full randomized Gross-Pitaevskii hierarchy \eqref{randomizedGP} with homogeneous initial data. Let us note that the property of belonging to the class $\mathcal{N}$ is not necessarily preserved under the evolution of \eqref{randomizedGP}. However, we do would not need to use this fact since we are studying the behavior of the fixed Duhamel expansions defined in \eqref{precisedefinition3}.

We will now state the main result of this section:

\begin{theorem}
\label{Smallness Bound 2} 
Suppose that $\alpha \geq 0$ and $k \in \mathbb{N}$.
Consider $\sigma^{(k)}_{n;\,\omega}$, defined as in \eqref{precisedefinition3} for $(\gamma^{(m)}(t))_m \in \mathcal{N}$, uniformly in time.
\\
\\
There exists $T>0$, depending only on $C_1$ and $\alpha$ such that:
$$\sup_{t \in [0,T]} \, \big\|S^{(k,\alpha)} \sigma^{(k)}_{n;\,\omega}(t)\big\|_{L^2 \big(\Omega \times \Lambda^k \times \Lambda^k \big)} \rightarrow 0\,\,\mbox{as}\,\,n \rightarrow \infty.
$$
\end{theorem}

\begin{proof}
Let us take  $t=t_k \in [0,T]$ and we compute:
$$S^{(k,\alpha)}\sigma^{(k)}_{n;\,\omega}(t_k)=$$
$$(-i)^n \int_{0}^{t_k} \int_{0}^{t_{k+1}} \cdots \int_{0}^{t_{n+k-1}} S^{(k,\alpha)}\,\mathcal{U}^{(k)}(t_k-t_{k+1}) \, [B^{(k+1)}]^{\omega} \,\mathcal{U}^{(k+1)}(t_{k+1}-t_{k+2})$$
$$ [B^{(k+2)}]^{\omega} \cdots \,\mathcal{U}^{(n+k-1)}(t_{n+k-1}-t_{n+k}) \, [B^{(n+k)}]^{\omega} \, \gamma^{(n+k)}(t_{n+k}) \,dt_{n+k} \cdots dt_{k+2} \, dt_{k+1}.$$ 
Hence, by Minkowski's inequality:
$$\big\|S^{(k,\alpha)} \sigma^{(k)}_{n;\,\omega}(t_k)\big\|_{L^2 \big(\Omega \times \Lambda^k \times \Lambda^k\big)}$$
$$\leq \int_{0}^{t_k} \int_{0}^{t_{k+1}} \cdots \int_{0}^{t_{n+k-1}} \big\|S^{(k,\alpha)}\,\mathcal{U}^{(k)}(t_k-t_{k+1}) \, [B^{(k+1)}]^{\omega} \,\mathcal{U}^{(k+1)}(t_{k+1}-t_{k+2})$$
$$ [B^{(k+2)}]^{\omega} \cdots \,\mathcal{U}^{(n+k-1)}(t_{n+k-1}-t_{n+k}) \, [B^{(n+k)}]^{\omega} \, \gamma^{(n+k)}(t_{n+k})\big\|_{L^2 \big(\Omega \times \Lambda^k \times \Lambda^k \big)}$$
$$ \,dt_{n+k} \cdots dt_{k+2} \, dt_{k+1}.$$ 
Since we are working in the class $\mathcal{N}$ of asymmetric non-resonant density matrices, we cannot apply the boardgame argument. Hence, we write out the full collision operators in terms of the individual collision operators. More precisely, we note that the integrand can be bounded from above by a sum of $\frac{(n+k)!}{(k-1)!}$ terms of the type estimated by Proposition \ref{non-resonant1}. We then use Proposition \ref{non-resonant1} to estimate the whole expression from above by:
$$\int_{0}^{t_k} \int_{0}^{t_{k+1}} \cdots \int_{0}^{t_{n+k-1}}  C_2^{\,n+k} \cdot \frac{(n+k)!}{(k-1)!} \cdot \big\|S^{(n+k,\alpha)}\gamma^{(n+k)}(t_{n+k})\big\|_{L^2(\Lambda^{n+k} \times \Lambda^{n+k})} \,dt_{n+k} \cdots dt_{k+2} \, dt_{k+1}.$$
By using the a priori assumption $ii)$ on $\gamma^{(n+k)}(t_{n+k})$, i.e. the fact that $(\gamma^{(m)}(t))_m \in \mathcal{N}$, uniformly in time, it follows that this quantity is:
$$\leq \int_{0}^{t_k} \int_{0}^{t_{k+1}} \cdots \int_{0}^{t_{n+k-1}}  C_2^{\,n+k} \cdot \frac{(n+k)!}{(k-1)!} \cdot C_1^{\,n+k} \,dt_{n+k} \cdots dt_{k+2} \, dt_{k+1},$$
which by \eqref{factorielintegral} is: 
$$\leq \frac{T^n}{n!} \cdot C_2^{\,n+k} \cdot \frac{(n+k)!}{(k-1)!} \cdot C_1^{\,n+k}.$$
From the previous analysis, it follows that there exists $M \sim C_1 \cdot C_2$, independent of $n$ and $k$, such that:
$$\sup_{t \in [0,T]} \, \big\|S^{(k,\alpha)} \sigma^{(k)}_{n;\,\omega}(t)\big\|_{L^2 \big(\Omega \times \Lambda^k \times \Lambda^k \big)} \leq \frac{T^n}{n!} \cdot M^{n+k} \cdot \frac{(n+k)!}{(k-1)!}.$$
Arguing analogously as in the proof of Theorem \ref{Smallness Bound}, this quantity is:
$$\lesssim_{k} \big(2MT\big)^n \cdot M^k.$$
We now choose $T$ sufficiently small such that $2MT \leq 1$ and we deduce that:
$$\sup_{t \in [0,T]} \, \big\|S^{(k,\alpha)} \sigma^{(k)}_{n;\,\omega}(t)\big\|_{L^2 \big(\Omega \times \Lambda^k \times \Lambda^k \big)} \rightarrow 0\,\,\mbox{as}\,\,n \rightarrow \infty.$$
By construction $T$ depends only on $C_1$ and on $C_2$. Since $C_2$ depends only on $\alpha$, it follows that $T$ depends only on $C_1$ and $\alpha$. The claim now follows.
\end{proof}

\begin{remark}
As in the proof of Theorem \ref{Smallness Bound}, we can deduce that $T \sim_{\alpha} \frac{1}{C_1}$, for the constant $C_1$ in the definition of the class $\mathcal{N}$.
\end{remark}

\end{document}